\documentclass[10pt]{amsart}

\usepackage[utf8]{inputenc}
\usepackage{multirow}
\usepackage{longtable}

\usepackage[backend=bibtex,style=alphabetic,sorting=anyt]{biblatex} 
\usepackage{amssymb,amscd,amsthm, verbatim,amsmath,color,fancyhdr, mathrsfs}
\usepackage{graphicx}
\usepackage{turnstile}
\usepackage{changepage}   

\usepackage{tikz-cd} 
\usepackage{enumitem}

\newtheorem{thm}{Theorem}[section]
\newtheorem{prop}[thm]{Proposition}
\newtheorem{lem}[thm]{Lemma}
\newtheorem{cor}[thm]{Corollary}

\newtheorem{conj}[thm]{Conjecture}
\theoremstyle{definition}
\newtheorem{ex}[thm]{Example}
\newtheorem{defn}[thm]{Definition}

\theoremstyle{remark}
\newtheorem{rem}[thm]{Remark}

\title{Extreme Divisors on $\overline{M}_{0,7}$ and differences over characteristic 2}

\author{Mathieu Dutour Sikiri\'c, Eric Jovinelly}

\date{\today}

\begin{document}
\maketitle
\addtocontents{toc}{\setcounter{tocdepth}{1}} 

\begin{abstract}
We find 101,052 new extreme divisors on $\overline{M}_{0,7}$ (in 31 $S_7$-orbits) and millions of extreme nef curves over characteristic 0.  
Over characteristic 2, we identify two more $S_7$-orbits of extreme divisors, and prove $\overline{\text{Eff}}^k (\overline{M}_{0,n})$ is strictly larger over characteristic 2 than it is over characteristic 0, for all $1\leq k \leq n-6$.  
For each such $k$ we provide explicit cycles which are extreme in $\text{Eff}^k(\overline{M}_{0,n})$ over characteristic 2 but external to $\overline{\text{Eff}}^k(\overline{M}_{0,n})$ over characteristic 0.  
We apply our method of finding new extreme divisors to compute $\text{Eff}(\overline{M}_{0,\mathcal{A}})$ for $\mathcal{A}=(\frac{1}{3}, \frac{1}{3}, \frac{1}{3}, \frac{1}{3}, \frac{1}{3}, \frac{1}{3}, 1)$, proving it is polyhedral over any field, and conjecture a description of $\text{Eff}(\text{Bl}_e \overline{LM}_7)$.
%
\end{abstract}

\section{Introduction}
For any variety $X$, extreme rays of $\overline{\text{Eff}}(X)$ 
 are important to the study of the birational geometry of $X$.  
Even when given an explicit projective embedding of $X$, these divisors may be challenging to identify if the Picard rank of $X$ is large.  Such is the case for $X=\overline{M}_{0,n}$, with $n\geq 7$.

For $n\leq 5$, the extreme divisors on $\overline{M}_{0,n}$ are irreducible components of its boundary, $\overline{M}_{0,n}\setminus M_{0,n}$.  In the late twentieth century, Fulton asked \cite{keel1996contractible} whether this remained true for $n>5$.  Keel and Vermiere \cite{VERMEIRE2002780} independently found the same counterexample to this conjecture: an extreme divisor on $\overline{M}_{0,6}$, called the Keel-Vermiere divisor, which cannot be written as an effective sum of boundary divisors.  Hassett and Tschinkel \cite{Hassett_2002} proved that $\text{Eff}(\overline{M}_{0,6})$ is generated by Keel-Vermiere and boundary divisors via a computer check.  Castravet \cite{castravet2008cox} later showed these divisors generate the Cox ring of $\overline{M}_{0,6}$, which proves $\overline{M}_{0,6}$ is a Mori Dream Space \cite{hu2000mori}.  Though not explored in this paper, Fulton's conjecture has many adaptations \cite{2013} \cite{fedorchuk2020symmetric} \cite{Bolognesi2017AVO} \cite{Moon2019OnTS} \cite{Gibney2009NumericalCF}, the most famous of which addresses curves \cite{Gibney2000TowardsTA} and is confirmed for $n\leq 7$ \cite{fontanari2016fultons}.


Fulton's initial conjecture continues to motivate research on $\overline{\text{Eff}}(\overline{M}_{0,n})$.  Castravet and Tevelev realized the Keel-Vermiere divisor as the simplest of the \textit{hypertree divisors}, a family of extreme divisors they constructed \cite{castravet2010hypertrees} on $\overline{M}_{0,n}$ whose equations were later connected to $N = 4$ super-symmetric Yang–Mills theory \cite{arkanihamed2014onshell} \cite{tevelev2021scattering}.  It was initially believed \cite{castravet2010hypertrees} that hypertree divisors were the only extreme, nonboundary divisors on $\overline{M}_{0,n}$.
This conjecture was first disproved by Opie in \cite{opie2016extremal} using extreme rays of $\overline{\text{Eff}}(\overline{\mathcal{M}}_{1,n})$ found in \cite{chen2013extremal}.  Doran, Giansiracusa, and Jensen \cite{doran2016simplicial} constructed the only other previously known counterexamples to this conjecture using simplicial methods. Simplifications and limitations to their methods over characteristic 0 appear in \cite{gonzález2017balanced}. 

Castravet and Tevelev \cite{Castravet_2015} subsequently showed that $\overline{M}_{0,n}$ is not a Mori Dream Space for $n\geq 134$.  González and Karu \cite{2015}, followed by Hausen, Keicher, and Laface \cite{hausen2016blowing}, lowered this bound to $n\geq 13$ and $n\geq 10$, respectively.  Castravet \cite{castravet2017mori} provides a detailed survey of these results.
More recently, Castravet, Laface, Tevelev, and Ugaglia \cite{castravet2020blownup} proved that $\overline{\text{Eff}}(\overline{M}_{0,n})$ is not polyhedral for all $n\geq 10$, both in characteristic 0 and characteristic $p$ for all primes $p$, by studying blow-ups of Toric surfaces at their identity.  For $7\leq n\leq 9$, it is unknown whether $\text{Eff}(\overline{M}_{0,n})$ is polyhedral.  We show that $\overline{\text{Eff}}(\overline{M}_{0,7})$ has many more extreme rays than $\overline{\text{Eff}}(\overline{M}_{0,6})$. 

\begin{thm}\label{divisorThm}
There are more than  $102,123$ ($37$ $S_7$-orbits of) extreme divisors on $\overline{M}_{0,7}$ over characteristic 0 and characteristic $p$ for all but finitely many prime $p$. 
\end{thm}

There are 38 $S_7$-inequivalent divisors appearing in Tables \ref{ExtDivs0}, \ref{ExtDivs1}, and \ref{boundaryCurves}.  We prove 37 of these are uniruled and extreme in $\overline{\text{Eff}}(\overline{M}_{0,7})$ by identifying families of rational curves that cover and pair negatively with each divisor.  We call such curves \textit{negative curves}.  Similarly, we compute free representatives for thousands of $S_7$-inequivalent curve classes to demonstrate that the dual of $\overline{\text{Eff}}(\overline{M}_{0,7})$, $\text{Nef}_1(\overline{M}_{0,7})$, is perhaps even more complex.  
\begin{thm}\label{nefCurvesThm}
There are over $9,875,000$ (2,135 $S_7$-orbits of) extreme nef curves on $\overline{M}_{0,7}$ in characteristic 0 and characteristic $p$ for all but finitely many prime $p$. 
\end{thm}
\noindent A portion of these extreme nef curves appear in Table \ref{nefCurves}.  Each curve class is the image of a free map $f:\mathbb{P}^1\rightarrow \overline{M}_{0,7}$ of anticanonical degree at most $4$ (at most 3 with 5 exceptions) and $\psi_i$-degree between $2$ and $6$ for some $i$.  We recorded only those which were not positive sums of other nef and \textit{negative} curve classes.

The last divisor in Table \ref{ExtDivs0}, $\boldsymbol{D_{31}}$, is extreme in $\text{Eff}(\overline{M}_{0,7})$ over characteristic 0, but specializes to a reducible divisor in characteristic 2.  The unique component $\boldsymbol{D_{37}} = \boldsymbol{D_{31}}-E_{012}$ intersecting $M_{0,7}\subset\overline{M}_{0,7}$ is one of two additional orbits of extreme divisors (also uniruled) we find over characteristic 2 which are not effective over characteristic 0.  We describe how the existence of certain curves implies the second such divisor class, $\boldsymbol{D_{38}}$, lies outside $\overline{\text{Eff}}(\overline{M}_{0,7})$ over characteristic 0.  By extending our construction to $\overline{M}_{0,n}$, we show that $\overline{\text{Eff}}(\overline{M}_{0,n})$ depends on the characteristic of the basefield of $\overline{M}_{0,n}$ for all $n\geq 7$.  


\begin{thm}\label{charThm}
For all $n\geq 7$, there are divisors $D\in N^1(\overline{M}_{0,n})$ which are effective over characteristic 2, but $D\not\in \overline{\text{Eff}}(\overline{M}_{0,n})$ over characteristic 0.
\end{thm}

\noindent Theorem \ref{charThm2} provides a more explicit statement.  While \cite{doran2016simplicial} previously found a divisor class $Q$ that is effective over characteristic 2 but not over characteristic 0, $\boldsymbol{D_5}=2Q$ is effective over characteristic 0 and appears in Table \ref{ExtDivs0}.  Our result is the first evidence that $\overline{\text{Eff}}(\overline{M}_{0,n})$ depends on characteristic.
Algorithmic searches found no difference between $\overline{\text{Eff}}(\overline{M}_{0,7})$ over characteristic 0 and characteristic $p\neq 2$, suggesting it may be different only over characteristic 2. 

Theorem \ref{charThm} implies $\overline{\text{Eff}}^k(\overline{M}_{0,n})$ is strictly larger over characteristic 2 than it is over characteristic 0 for all $1\leq k \leq n-6$.  
However, this does not identify explicit extreme rays of $\overline{\text{Eff}}^k(\overline{M}_{0,n})$ or $\text{Eff}^k(\overline{M}_{0,n})$ over characteristic 2 which are external to $\overline{\text{Eff}}^k(\overline{M}_{0,n})$ over characteristic 0.  Following work by \cite{Chen_2015}, \cite{blankers2021extremality}, and \cite{schaffler2015cone}, we show how to construct extreme effective cycles of higher codimension using clutching morphisms.

\begin{thm}\label{extCycles}
Let $i: \overline{M}_{0,n-m+2}\times \overline{M}_{0,m} \rightarrow \overline{M}_{0,n}$ be the embedding of a boundary divisor.  If $Z\in \text{Eff}^k(\overline{M}_{0,n-m+2})$ is extreme, then $i_*( Z \times \overline{M}_{0,m}) \in \text{Eff}^{k+1}(\overline{M}_{0,n})$ is extreme.
\end{thm}

\noindent In Proposition \ref{Proposition: codim2cycles} we show that when $Z\subset\overline{M}_{0,n}$ is a nonboundary divisor swept out by \textit{negative} curves, the pushforward of $Z$ via $i:\overline{M}_{0,n}\rightarrow \overline{M}_{0,n+1}$ is extreme in $\overline{\text{Eff}}^2(\overline{M}_{0,n+1})$.  We obtain the following explicit corollary of Theorems \ref{extCycles} and \ref{charThm2}.

\begin{cor}\label{higherCodim}
Let $k \leq n-5$, $D\in N^1(\overline{M}_{0,n-k})$ be as in Theorem \ref{charThm2}, and $i:\overline{M}_{0,n-k}\rightarrow \overline{M}_{0,n}$ be the embedding of a codimension $k$ boundary strata.  Over characteristic 0, $i_* D$ is external to $\overline{\text{Eff}}^{k+1}(\overline{M}_{0,n})$.  Over characteristic 2, $i_* D \in \text{Eff}^{k+1}(\overline{M}_{0,n})$ is extreme, while if $k\leq 1$, $i_* D \in \overline{\text{Eff}}^{k+1}(\overline{M}_{0,n})$ is extreme.
\end{cor}


We outline the algorithm used to find new extreme divisors on $\overline{M}_{0,7}$ in Section \ref{Algorithms}.  This algorithm may be applied to blow-ups of smooth Fano hypersurfaces, or more generally blow-ups of any smooth, quasi-projetive variety $X$ whose Cartier divisors are all hypersurface sections.
Using \textit{negative curve classes} (resp. known extreme divisors), we bound an exterior (resp. interior) polyhedral approximation of $\text{Eff}(\overline{M}_{0,7})$.  By feeding these two approximations to certain linear programs, we either determine that $\text{Eff}(\overline{M}_{0,7})$ is polyhedral and generated by known divisors, or generate candidates for new extreme divisors and negative curve classes.  To test these candidates, we require an algorithm that tests for effectivity of a given divisor class, and an algorithm that computes the splitting type of $\mathcal{N}_f$ for general maps $f:\mathbb{P}^1\rightarrow \overline{M}_{0,7}$ of fixed numeric class.  The algorithms we describe for $\overline{M}_{0,7}$ may be adapted to other blow-ups of $\mathbb{P}^n$.  In particular, 
we show the Hassett space \cite{Hassett_2003} $\overline{M}_{0,\mathcal{A}}$ with $\mathcal{A}=(\frac{1}{3}, \frac{1}{3}, \frac{1}{3}, \frac{1}{3}, \frac{1}{3}, \frac{1}{3}, 1)$\footnote{i.e. the blow-up of $\mathbb{P}^4$ along 6 linearly general points and the strict transforms of the 15 lines they span.} has polyhedral effective cone.

\begin{thm}\label{contractionThm}
Let $\mathcal{A}=(\frac{1}{3}, \frac{1}{3}, \frac{1}{3}, \frac{1}{3}, \frac{1}{3}, \frac{1}{3}, 1)$ and $\phi : \overline{M}_{0,7}\rightarrow \overline{M}_{0,\mathcal{A}}$ be the natural reduction map.
Over any basefield, $\text{Eff}(\overline{M}_{0,\mathcal{A}})$ is polyhedral with 976 extreme rays.  It is generated by the $\phi$-pushforward of divisors $S_7$-equivalent to those in Tables \ref{ExtDivs0}, \ref{ExtDivs1}, and \ref{boundaryCurves}.
\end{thm}
Section \ref{Section8} lists all $S_6$-orbits of extreme divisors on $\overline{M}_{0,\mathcal{A}}$.  Recently, He and Yang \cite{he2021birational} showed the blow-up of $\overline{M}_{0,\mathcal{A}}$ in Theorem \ref{contractionThm} at a very-general point has nonpolyhedral effective cone.  He \cite{HeZhuang2020Bgob} further showed the blow-up of the Losev-Manin Space $\overline{LM}_7$ at the identity of its embedded torus, $\text{Bl}_e \overline{LM}_7$, is log-Fano.  This implies $\text{Eff}(\text{Bl}_e \overline{LM}_7)$ is polyhedral.  In Section \ref{Section8}, we also provide a conjectural list of all $S_2\times S_5$-orbits of extreme divisors on $\text{Bl}_e \overline{LM}_7$.  Our conjectured description for $\text{Eff}(\text{Bl}_e \overline{LM}_7)$ has $3.2$ million facets, and we verify that all but a few hundred $S_2 \times S_5$-orbits of these are nef curve classes. 

\begin{conj}\label{losev_manin_conj}
Let $\phi : \overline{M}_{0,7}\rightarrow \text{Bl}_e \overline{LM}_7$ be the natural reduction map.
Over characteristic 0 and all but finitely many prime characteristics, $\text{Eff}(\text{Bl}_e \overline{LM}_7)$ is polyhedral with 138 extreme rays.  It is generated by the $\phi$-pushforward of divisors $S_7$-equivalent to those in Tables \ref{ExtDivs0}, \ref{ExtDivs1}, and \ref{boundaryCurves}.
\end{conj}


In Section \ref{Section2}, we define \textit{negative curves} and review Kapranov's construction of $\overline{M}_{0,n}$ as an iterated blow-up of $\mathbb{P}^{n-3}$.  
In Section \ref{Section3}, we provide deformation-theoretic criteria for showing $f_*[\mathbb{P}^1]$ is a negative curve class on a given divisor.  We extend our study to curves on flat varieties $X_\mathbb{Z} \rightarrow \mathbb{Z}$ and prove Theorem \ref{divisorThm}.  
Section \ref{Section4} contains the proof of Theorem \ref{charThm}, along with a more refined statement, Theorem \ref{charThm2}.  The results in Section \ref{Section4} describe methods for finding differences between the effective cones of special and general fibers of $X_\mathbb{Z}\rightarrow \mathbb{Z}$.
Section \ref{Section5} gives proofs of Theorem \ref{extCycles} and Corollary \ref{higherCodim}.  Section \ref{Algorithms} describes the algorithmic procedure we developed for finding new extreme divisors on arbitrary varieties $X$, 
while Section \ref{M07alg} contains subsidiary algorithms for testing divisor and curve classes on $\overline{M}_{0,7}$.  Section \ref{Section8} proves Theorem \ref{contractionThm} and summarizes progress towards Conjecture \ref{losev_manin_conj}.  Tables of extreme divisors, nef curves, and more appear in Section \ref{Tables}, along with a proof that none of the newly found divisors are proper transforms of Chen-Coskun divisors \cite{opie2016extremal}.  A short proof of Theorem \ref{nefCurvesThm} appears before Table \ref{nefCurves}.

\textbf{Acknowledgments:}  The authors would like to thank Eric Riedl, Ana-Maria Castravet, Izzet Coskun, Jonathan Hauenstein, Gerhard Pfister, Geoff Smith, Paolo Gianni, and countless others for helpful discussions about this project.



\tableofcontents

\section{Background and Notation}\label{Section2}
\subsection{Background}  Both \cite{alma9927409813902959} and \cite{lazarsfeld2004positivity} offer reasonable background for this paper.  We let $X$ be a variety and $\text{Eff}(X)$ be the full-dimensional, pointed cone of effective divisors inside $N^1(X)$, with closure $\overline{\text{Eff}}(X)$.  Its dual, $\text{Nef}_1(X)\subset N_1(X)$ is the cone of nef (free) curves on $X$.  
Let $D\in N^1(X)$ be the class of an effective, rigid, irreducible divisor.  To verify that $D$ is an extreme ray of $\overline{\text{Eff}}(X)$, it suffices to find a curve class $c \in N_1(X)$, satisfying $c . D < 0$, that admits a family of generically integral curves covering $D$ (see \cite[Corollary~6.4]{opie2016extremal} for details).  Such a curve $c$ provides a hyperplane in $N^1(X)$ that separates $D$ from every other extreme ray of $\overline{\text{Eff}}(X)$.  We call such curves \textit{negative curves} (on/sweeping out $D$).

\begin{defn}
Let $D\subset X$ be a Cartier divisor.  A curve class $c\in N_1(X)$ is called a \textit{negative curve} (on/sweeping out $D$) if $c.D < 0$ and there exists a family of integral curves, numerically equivalent to $c$, whose image dominates $D$.  
\end{defn}

In Lemma \ref{deformationLemma} and Proposition \ref{liftingProp}, we describe how to verify that a single curve $f:\mathbb{P}^1\rightarrow X$ satisfying $f_*[\mathbb{P}^1] . D < 0$ deforms in such a family by studying the normal bundle $\mathcal{N}_{f}$.

\begin{defn}
For smooth, quasi-projective $X$, the \textit{(virtual) normal bundle} of $f:\mathbb{P}^1\rightarrow X$ is $\mathcal{N}_{f} = \text{coker}(\mathcal{T}_{\mathbb{P}^1}\rightarrow f^*\mathcal{T}_X)$.
\end{defn}



\subsection{Notation}\label{Notation}  Throughout our paper, $\mathcal{D}$ will be a finite set of extreme divisors on a variety $X$ and $\mathcal{C}$ will be a finite collection of negative curves sweeping out divisors in $\mathcal{D}$.  We let $\mathfrak{E}\subset N^1(X)$ denote the cone generated $\mathcal{D}$, and define the cone 
$$\mathfrak{M}=\{ D\in N^1(X) | D.c \geq 0 \text{ for all } c\in \mathcal{C}\}.$$



We focus primarily upon $X=\overline{M}_{0,7}$.  Kapranov's construction realizes $\overline{M}_{0,n}$ as the iterated blow-up of $\mathbb{P}^{n-3}$ along $n-1$ linearly general points and the strict transforms of their linear spans (of codimension at least 2), in order of increasing dimension.  In particular, this realizes $\overline{M}_{0,7}$ as the iterated blow-up of $\mathbb{P}^4$ at $6$ points $p_i$, followed by the strict transforms of the 15 lines $l_{ij}=\overline{p_i p_j}$ they span, and lastly the strict transforms of the 20 planes $\Delta_{ijk}=\Delta p_i p_j p_k$ they span.  

Let $H$ denote the pullback of the hyperplane class from $\mathbb{P}^4$ under the iterated blow-up $\overline{M}_{0,7}\rightarrow \mathbb{P}^4$, and $E_i,E_{ij},$ and $E_{ijk}$ denote the exceptional divisors lying over $p_i$, the strict transform of $l_{ij}$, and the strict transorm of $\Delta_{ijk}$, respectively.\footnote{$H$ is a $\psi$-class, while $E_i$, $E_{ij}$, and $E_{ijk}$ are components of the boundary $\overline{M}_{0,7}\setminus M_{0,7}$.}  
These comprise a basis for $N^1(\overline{M}_{0,7})\cong\text{Pic}(\overline{M}_{0,7})$.  

A corresponding basis for $N_1(\overline{M}_{0,7})$ consists of $l, e_i ,e_{ij}, \text{ and } e_{ijk}$, where $l$ is the strict transform of a general line under $\overline{M}_{0,7}\rightarrow \mathbb{P}^4$, $e_i$ is the strict transform of a general line under the blow-up $E_i\rightarrow \mathbb{P}^3$, $e_{ij}$ is the strict transform of a general line in $\mathbb{P}^2$ under the blow-up $E_{ij}\rightarrow \mathbb{P}^1 \times \mathbb{P}^2$, and $e_{ijk}$ is the fiber over a general point in the restriction of Kapranov's map to $E_{ijk}\rightarrow \mathbb{P}^2\subset\mathbb{P}^4$.  Note $l.H =1$, $e_i .E_i = -1$, $e_{ij} .E_{ij} = -1$, and $e_{ijk} . E_{ijk} = -1$.  All other pairing of basis elements are $0$.

Choosing the $6$ points $p_i$ as in Section \ref{M07alg}, we may realize $\overline{M}_{0,7}$ as a blow-up of $\mathbb{P}^4_\mathbb{Z}$, flat over $\mathbb{Z}$.  Our choice of basis for $N^1(\overline{M}_{0,7})$ then consists of relative Cartier divisors.  Letting $\overline{M}_{0,n,k}$ denote $\overline{M}_{0,n}$ defined over a field $k$, we therefore obtain compatible isomorphisms $N^1(\overline{M}_{0,n,k})\cong N^1(\overline{M}_{0,n,\mathbb{Q}})$ for all fields $k$.  By convention, we omit the subscript $k$ and let $\overline{M}_{0,n}$ refer to the appropriate space over a characteristic $0$ field, unless otherwise indicated.

\section{Negative Curves and Deformations}\label{Section3}
Given a divisor $D\subset X$ and a curve $f:\mathbb{P}^1 \rightarrow X$ with $f_*[\mathbb{P}^1] . D < 0$, it is desirable to provide criteria on $f$ and $D$ that imply $f_*[\mathbb{P}^1]$ is a negative curve on $D$.  The following outlines such criteria.  Note that statements about the global generation of subsheaves of $\mathcal{N}_{f}$ imply similar statements about $f^*\mathcal{T}_X$.

\begin{lem}\label{deformationLemma}
Let $X$ be a smooth quasi-projective variety defined over a field $k$, $f:\mathbb{P}^1_k\rightarrow X$ be a nonconstant map, and $\mathcal{N}_f := \text{coker}(\mathcal{T}_{\mathbb{P}^1}\rightarrow f^*\mathcal{T}_{X})$. Suppose that $\mathcal{N}_f$ has splitting type
$$\mathcal{N}_f\cong \mathcal{O}_{\mathbb{P}^1}(-n) \oplus V$$
with $V$ globally generated and $n>0$.  Let $D\subset X$ be an effective divisor such that $\text{deg}(f^*D)=-n$.  Then $f(\mathbb{P}^1)$ is either contained inside the singular locus of $D$, or disjoint from it.  
If $n=1$, or $D$ is smooth at any point of $f(\mathbb{P}^1)$, then $f_*[\mathbb{P}^1]$ is a negative curve on some integral component of $D$.
\end{lem}

\begin{proof}
Since $\text{deg}(f^*D)=-n$, $f(\mathbb{P}^1)$ is contained in some component of $D$.  
Let $D^{sm}$ denote the smooth part of $D$.

The natural map $\mathcal{N}_f\rightarrow f^*\mathcal{O}(D)$ is nonzero iff $f(\mathbb{P}^1)$ meets $D^{sm}$, and surjective iff $f(\mathbb{P}^1) \subseteq D^{sm}$.  This immediately implies the first claim.  By standard results in deformation theory \cite{Kollr1996RationalCO}, it follows that when $n=1$ or $f(\mathbb{P}^1) \subseteq D^{sm}$, deformations of $f:\mathbb{P}^1\rightarrow D\subset X$ sweep out an integral component $D'$ of $D$.  
\end{proof}
We remark that it is not enough to find $f:\mathbb{P}^1\rightarrow X$ with $\mathcal{N}_f \cong V \oplus \mathcal{O}(-n)$, even when $n=1$, and hope that deformations sweep out a divisor $D$ with $f_*[\mathbb{P}^1]. D <0$.  Indeed, if $\text{deg}(V) > 0$, a general deformation of $f$ may be free.  Even when $\text{deg}(V)=0$, the divisor swept out by deformations of $f$ may be singular along each such deformation.

\begin{prop}\label{liftingProp}
Adopt the notation of Lemma \ref{deformationLemma}.  In addition, suppose $X= X_\mathbb{Z}\times_{\mathbb{Z}}\text{Spec}(k)$  for some $X_\mathbb{Z}$ quasi-projective and flat over $\mathbb{Z}$, $\text{char}(k)=p > 0$, and $D=\tilde{D}\times_{\mathbb{Z}}\text{Spec}(k)$ for some $\tilde{D}\subset X_\mathbb{Z}$, flat over 
$\mathbb{Z}_{(p)}$, whose generic fiber $\tilde{D}_{\mathbb{Q}}:=\tilde{D}\times_{\mathbb{Z}}\text{Spec}(\mathbb{Q})$ is geometrically integral.  If $n=1$ or $f(\mathbb{P}^1)$ meets the smooth locus of $D$, then $f:\mathbb{P}^1_k \rightarrow X$ lifts to a negative curve in characteristic 0 that sweeps out $\tilde{D}_{\overline{\mathbb{Q}}}$, 
and $D$ is likewise geometrically integral. 
\end{prop}
\begin{figure}[ht]
\includegraphics[width=0.55\linewidth]{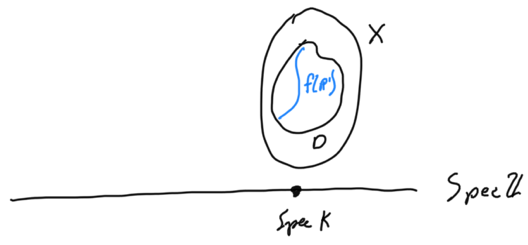}
\caption{A visual depiction of Proposition \ref{liftingProp}.}
\end{figure}
\begin{proof}
If $n=1$, let $Y=X_{\mathbb{Z}}\times_{\mathbb{Z}}\mathbb{Z}_{(p)}$, so that $f:\mathbb{P}^1\rightarrow Y \times_{\mathbb{Z}_{(p)}} \text{Spec}(k)$ is a smooth point of $\text{Mor}_{\mathbb{Z}_{(p)}}(\mathbb{P}^1, Y)$ over $\mathbb{Z}_{(p)}$.  Otherwise, if $f(\mathbb{P}^1)$ meets $D^{sm}$, instead let $Y=\tilde{D}$, so that $f:\mathbb{P}^1\rightarrow Y \times_\mathbb{Z} \text{Spec}(k)$ is a smooth point of $\text{Mor}_{\mathbb{Z}_{(p)}}(\mathbb{P}^1, Y)$ over $\mathbb{Z}_{(p)}$.  In both cases, $f$ lifts to $\tilde{f}:\mathbb{P}^1_A\rightarrow Y\times \text{Spec}(A)$ for $A$ a local domain with residue field $k'$ (a finite seperable extension of $k$) and fraction field $K$ of characteristic 0.  In fact, generic deformations of $f$ as a map into $Y$ are unobstructed, and consequently smooth points of $\text{Mor}_{\mathbb{Z}_{(p)}}(\mathbb{P}^1, Y)$.  These deformations therefore lift similarly to $f$.  
This shows that the irreducible component $Z \subset \text{Mor}_{\mathbb{Z}_{(p)}}(\mathbb{P}^1, Y)$ containing $[f]$ is flat over $\mathbb{Z}_{(p)}$.  Let $U\rightarrow Z$ be the universal family, and $\text{ev}:U\rightarrow Y\rightarrow X_\mathbb{Z}$ be the evaluation map.  Since $\mathcal{N}_f\cong V\oplus \mathcal{O}(-n)$ with $V$ globally generated, and $f$ was unobstructed as a map to $Y$, we see that $\overline{\text{ev}(U)}$ is a relative Cartier divisor in $X_{\mathbb{Z}}$ over $\mathbb{Z}_{(p)}$.

Since $\tilde{D}\times \text{Spec}(A)$ is a flat extension of $D_{k'}=D\times_{k}\text{Spec}(k')$, we see $\text{deg}(\tilde{f}|_{K}^*(\tilde{D}_K))=\text{deg}(\tilde{f}|_{k'}^*(\tilde{D}_{k'}))=\text{deg}(f^*D)=-n$.  Ergo, deformations of $\tilde{f}|_{K}$ over characteristic 0 are contained in $\tilde{D}_K$.  It follows that $\overline{\text{ev}(U)}_K$ is a geometrically integral component of $\tilde{D}_K$.  
Since $\tilde{D}_\mathbb{Q}$ itself is geometrically integral, we must have $\overline{\text{ev}(U)}_K = \tilde{D}_K$.  As $D_{k'}$ is contained in the closure of $\tilde{D}_K\subset Y$, this proves $\overline{\text{ev}(U)}_{k'}=D_{k'}$ as well.  It follows that $D_{k'}$ (and therefore $D$ itself) is geometrically integral.
\end{proof}

This proposition has a few important consequences for $X=\overline{M}_{0,7}$, listed below.  Foremost, we use it to prove Theorem \ref{divisorThm}. 

\begin{proof}[Proof of Theorem \ref{divisorThm}]
As stated after Theorem \ref{divisorThm}, there are 38 $S_7$-inequivalent divisors in Tables \ref{ExtDivs0}, \ref{ExtDivs1}, and \ref{boundaryCurves}.  
Divisors in Table \ref{boundaryCurves} are merely components of the boundary $\overline{M}_{0,7}\setminus M_{0,7}$.  These are extreme in $\overline{\text{Eff}}(\overline{M}_{0,7})$ over any characteristic.  In Tables \ref{ExtDivs0} and \ref{ExtDivs1}, the third column of each row lists negative curves on the corresponding divisor $D$.  More specifically, these curve classes admit representatives $f:\mathbb{P}^1\rightarrow D\subset \overline{M}_{0,7}$ over a finite field that satisfy the hypotheses of Lemma \ref{deformationLemma} and Proposition \ref{liftingProp}.  The component $Z \subset \text{Mor}_{\mathbb{Z}}(\mathbb{P}^1, \overline{M}_{0,7})$ containing $[f]$ dominates $\text{Spec}(\mathbb{Z})$.  General deformations of $f$ have normal bundles as balanced as $\mathcal{N}_{f}$, and intersect $D^{sm}$ if $f(\mathbb{P}^1)$ does.  Together, this implies that for all but finitely many primes $p$, there are curves over a field of characteristic $p$ satisfying the hypotheses of \ref{deformationLemma} and \ref{liftingProp}.  These divisors are therefore extreme in $\overline{\text{Eff}}(\overline{M}_{0,7})$ over characteristic $p$ for all but finitely many $p$.  All but the last divisor in Table \ref{ExtDivs0} have such curve classes, proving our claim.
\end{proof}

\begin{rem}
Let $\pi:\overline{M}_{0,7}\rightarrow \mathbb{P}^4$ be the Kapranov morphism described by Sections \ref{Notation} and \ref{M07alg}.  If $\tilde{D}\subset \overline{M}_{0,7}$ is a rigid, irreducible, nonboundary divisor in characteristic 0, $|\tilde{D}|$ has a unique element, which may be represented as the strict transform of $\pi(\tilde{D})\subset \mathbb{P}^4$, given by a polynomial $g$ with coefficients in $\mathbb{Z}$.  $\tilde{D}_{\mathbb{Q}}$ is geometrically integral iff $g$ is irreducible over $\mathbb{Q}$.  Reducing this polynomial $g$ modulo $p$, we recover a flat extension $D\rightarrow \tilde{D}$ as above.  It is much easier to find curves $f:\mathbb{P}^1\rightarrow D$ because the underlying field is finite.  The proposition above shows that, under the correct hypotheses, such curves lift to curves in characteristic 0.  We may therefore prove $\tilde{D}$ is extreme over characteristic 0 by studying curves in finite characteristic.
\end{rem}
Proposition \ref{liftingProp} also implies the following corollary, which restricts the splitting type of $\mathcal{N}_f$ for a negative curve $f:\mathbb{P}^1\rightarrow X = X_{\mathbb{Z}}\times_{\mathbb{Z}} \text{Spec}(k)$ on a divisor $D$ which does not lift to characteristic 0.  This may be phrased as proving that if such a curve $f:\mathbb{P}^1\rightarrow X$ exists, the divisor class $D$ is necessarily effective over characteristic 0.
\begin{cor}\label{lifting_remark}
Let $X_\mathbb{Z}$ be quasi-projective and flat over $\mathbb{Z}$, $X = X_\mathbb{Z} \times_\mathbb{Z} \text{Spec}(k)$ be smooth, and $k$ be a field with $\text{char}(k)=p>0$.  Suppose $f:\mathbb{P}^1\rightarrow X$ has normal bundle
$$\mathcal{N}_f\cong \mathcal{O}_{\mathbb{P}^1}(-1) \oplus \mathcal{O}_{\mathbb{P}^1}^{\text{dim}(X) -2}.$$
Let $D\subset X$ be an integral divisor with $\text{deg}(f^*D)=-1$.  Then $D$ lifts to a relative Cartier divisor $\tilde{D}$ over $\mathbb{Z}_{(p)}$.
\end{cor}
\begin{proof}
In Proposition \ref{liftingProp}, if no $\tilde{D}$ is given but $n=1$ and $\text{deg}(V)=0$, the splitting type of $\mathcal{N}_g$ for a general deformation $g$ of $f$ is the same as $\mathcal{N}_f$.  Therefore, the proof of Proposition \ref{liftingProp} allows us to construct $\tilde{D}$ as the reduced closure of the image of such deformations.  This relative Cartier divisor $\tilde{D}$ must specialize to a multiple $mD$ of $D$ over $k$.  As $\mathcal{N}_f\cong \mathcal{N}_g$ maps nontrivially to $\mathcal{O}_{\mathbb{P}^1}(g^*\tilde{D})\cong \mathcal{O}_{\mathbb{P}^1}(f^*\tilde{D}) = \mathcal{O}(-m)$, we must have $m = 1$.
\end{proof}
Negative curves of anticanonical degree 1 whose normal bundles have the above splitting type are the easiest classes to compute representatives for on $\overline{M}_{0,7}$.  Corollary \ref{lifting_remark} helps explain the strange behavior of negative curves we found on divisors over characteristic 2.  This is discussed below as well as in Remarks \ref{reducible_remark}, \ref{chardif_remark}, and \ref{strange behavior}.
\begin{ex}
In Table \ref{ExtDivs2}, the degree 9 extreme divisor $D=\boldsymbol{D_{37}}$ on $\overline{M}_{0,7,\mathbb{F}_2}$ 
does not lift to an effective divisor over $\mathbb{Q}$; however, $\boldsymbol{D_{31}} = D+E_{012}$ does lift.  If $c$ is the unique negative curve listed in table \ref{ExtDivs2} for $\boldsymbol{D_{37}}$, then $c.(D+E_{012}) = -1$.  This shows $D+E_{012}$ is rigid over $\mathbb{Q}$, and we prove it is extreme in $\text{Eff}(\overline{M}_{0,7})$.  However, finding a negative curve for $D+E_{012}$ is challenging.  Any such curve must not deform to a negative curve in characteristic 2 by Proposition \ref{liftingProp}, as $D+E_{012}$ is reducible over $\mathbb{F}_2$.
\end{ex}
\begin{ex}
In Table \ref{ExtDivs0}, the divisor $\boldsymbol{D_5}$ is swept out by a unique negative curve class with normal bundle $\mathcal{O}^2 \oplus \mathcal{O}(-2)$.  General deformations of this curve to characteristic 2 have normal bundle $\mathcal{O}(1)\oplus \mathcal{O}(-1)\oplus \mathcal{O}(-2)$.  This reflects the fact that $\frac{1}{2}\boldsymbol{D_5}$ is an effective divisor over characteristic 2, but not characteristic 0.  It is unknown whether $\frac{1}{2}\boldsymbol{D_5}$ is extreme in $\overline{\text{Eff}}(\overline{M}_{0,7,\mathbb{F}_2})$.
\end{ex}

\section{Proof of Theorem \ref{charThm}}\label{Section4}

Theorem \ref{charThm} follows from the more specific statement below. The divisors $\boldsymbol{D_{31}}$ and $\boldsymbol{D_{38}}$ appear in Tables \ref{ExtDivs0} and \ref{ExtDivs2} respectively, and were found using Algorithm \ref{sectionAlgExtDiv}. 

\begin{thm}\label{charThm2}
For all $n\geq 7$, there is an explicit $D\in N^1(\overline{M}_{0,n})$ such that $D$ is effective and extreme in $\overline{\text{Eff}}(\overline{M}_{0,n})$ over characteristic $2$, but $D\not\in \overline{\text{Eff}}(\overline{M}_{0,n})$ over characteristic $0$.  For $n=7$, this explicit class is
\begin{align*}
\boldsymbol{D_{38}}=&12H-8E_{0}-8E_{1}-8E_{2}-6E_{3}-6E_{4}-6E_{5}-5E_{01}-5E_{02}-5E_{03}-4E_{04}\\
&-5E_{05}-5E_{12}-4E_{13}-5E_{14}-5E_{15}-5E_{23}-5E_{24}-4E_{25}-2E_{34}-2E_{35}-2E_{45}\\
&-2E_{012}-3E_{013}-3E_{014}-2E_{015}-2E_{023}-3E_{024}-3E_{025}-E_{034}-E_{035}-E_{045}\\
&-3E_{123}-2E_{124}-3E_{125}-E_{134}-E_{135}-E_{145}-E_{234}-E_{235}-E_{245}-E_{345}. 
\end{align*}
For $n >7$, the explicit class is a pullback of $\boldsymbol{D_{38}}$ under a sequence of forgetful maps $\phi : \overline{M}_{0,n}\rightarrow\overline{M}_{0,n-1}$.

Likewise, for all $n\geq 7$, there is an explicit $F\in N^1(\overline{M}_{0,n})$ such that $F\in \text{Eff}(\overline{M}_{0,n})$ is extreme over characteristic $0$ but not over characteristic $2$.  For $n=7$, this explicit class is 
\begin{align*} 
\boldsymbol{D_{31}}=&9H - \sum_{i=0}^5 5E_i - \sum_{0\leq i < j \leq 5} 3E_{ij} - E_{012}-2E_{013}-E_{014}-E_{015}\\
&-E_{023}-2E_{024}-E_{025}-E_{034}-2E_{035}-2E_{045}-E_{123}-E_{124}\\
&-2E_{125}-2E_{134}-E_{135}-2E_{145}-2E_{234}-2E_{235}-E_{245}-E_{345}.
\end{align*}
For $n >7$, the explicit class is a pullback of $F$ under a sequence of forgetful maps $\phi : \overline{M}_{0,n}\rightarrow\overline{M}_{0,n-1}$.
\end{thm}

We prove the above theorem in this section.  
Propositions \ref{chardif} and \ref{chardif4} provide general criteria for showing the (pseudo)-effective cones of the fibers of a morphism $X\rightarrow \mathbb{Z}$ differ over generic and special fibers.  This requires natural isomorphisms $N^1(X_k) \cong N^1(X_{\overline{\mathbb{Q}}})$ for all fields $k$.  

Let $k$ be a field of positive characteristic and $D\in \text{Eff}(X_k)$ be the class of an integral divisor such that $h^0(D,X_{\overline{\mathbb{Q}}}) = 0$.  In Proposition \ref{chardif}, we show how the existence of integral curves $f:\mathbb{P}^1_{k'} \rightarrow X_{k'}$ of a special class $c$ prove $D\not\in \overline{\text{Eff}}(X_{\overline{\mathbb{Q}}})$.  Here, $k'$ may be any field.  For computational purposes, we used finite fields of characteristic different from $k$.  While it would be convenient to choose $k'$ with $\text{char}(k')=\text{char}(k)$, Remark \ref{reducible_remark} shows how the anticanonical degree of $c$ may be too small for such a choice.

\begin{prop}\label{chardif}
Let $X\rightarrow\mathbb{Z}$ be a variety with natural identifications $N^1(X_k)\cong N^1(X_{\overline{\mathbb{Q}}})=:N^1(X)$ for every field $k$.  Suppose $D\in N^1(X)$ is the class of an integral divisor on $X_k$ for some field $k$ of positive characteristic, but $h^0(D, X_{\overline{\mathbb{Q}}})=0$.  If $c$ is any curve class satisfying
\begin{enumerate}
\item $c.D=-1$,
\item $c.D'\geq 0$ for any integral effective Cartier divisor $D'\neq D$ over $\overline{k}$,
\item There exists a field $k'$ and a morphism $f:\mathbb{P}^1_{k'}\rightarrow X_{k'}$ such that $f_*[\mathbb{P}^1]=c$ and either $\mathcal{N}_f$ is globally generated or $\mathcal{N}_{f}\cong V \oplus \mathcal{O}(-1)$ for some globally generated $V$,
\end{enumerate}
then $D\not\in \overline{\text{Eff}}(X_{\overline{\mathbb{Q}}})$.
\end{prop}
\begin{defn}[\textit{Characteristic $p$ Pair}]
We call a divisor and curve pair $(D,c)$ as in Proposition \ref{chardif} a \textit{characteristic} $\text{char}(k)$ \textit{pair}.  We say $(D,c)$ is defined over $k$.
\end{defn}
\begin{proof}
Let $c,D,k$, and $f:\mathbb{P}^1_{k'}\rightarrow X_{k'}$ be as stated above.  By general results in deformation theory, $[f]$ is a smooth point of $\text{Mor}_{\text{Spec}(\mathbb{Z})}(\mathbb{P}^1,X)\rightarrow \text{Spec}(\mathbb{Z})$.  It therefore lifts to a map $\tilde{f}:\mathbb{P}^1\rightarrow X_{k}$ for some field $k$ of characteristic $0$.  Since the splitting type of $\mathcal{N}_{\tilde{f}}$ is at least as balanced as $\mathcal{N}_f$, it follows that either $\tilde{f}_*(\mathbb{P}^1)$ is nef, or deformations of $\tilde{f}$ sweep out a divisor $D'$ with $\tilde{f}_*(\mathbb{P}^1).D' = -1$.  In the latter case, $\tilde{f}_*(\mathbb{P}^1)$ is a negative curve class for $D'$ over characteristic $0$.  Note that $D' \neq D$, since $h^0(D, X_{\overline{\mathbb{Q}}})=0$.  In all cases, either $\tilde{f}_*(\mathbb{P}^1)$ is nef and separates $D$ from all generators of $\text{Eff}(X_{\overline{\mathbb{Q}}})$ (classes of integral effective Cartier divisors), or $\tilde{f}_*(\mathbb{P}^1)$ pairs nonnegatively with all but one generator $D'$ of $\text{Eff}(X_{\overline{\mathbb{Q}}})$.  It suffices to prove our claim in the latter case, as the former is trivial.

The specialization of $D'$ to characteristic $\text{char}(k)$ must be reducible, since $c.D'=-1$ but $D'\neq D$.  Therefore $D' = D + E$ for some nonzero $E\in \text{Eff}(X_{k})$.  If some sequence $(D_n)\rightarrow D$, with $D_n\in \text{Eff}(X_{\overline{\mathbb{Q}}})$, then for all $n >> 0$, $a_n := c.D_n < 0$.  This implies $(D_n + a_n D')\in \text{Eff}(X_{\overline{\mathbb{Q}}})$, as $c$ is a negative curve class for $D'$ over characteristic 0.  However, $(D_n + a_n D')\rightarrow D - D' = -E \in \overline{\text{Eff}}(X_{\overline{\mathbb{Q}}})\subseteq \overline{\text{Eff}}(X_{k})$.  This is a contradiction, since $\overline{\text{Eff}}(X_{k})$ contains no lines.  Thus, $D\not\in \overline{\text{Eff}}(X_{\overline{\mathbb{Q}}})$.
\end{proof}

\begin{lem}\label{chardif2}
Let $D=\boldsymbol{D_{38}}$ be the divisor in the second row of Table \ref{ExtDivs2}, and $c = 4l - e_{03} -e_{05} -e_{14} -e_{15} -e_{23} -e_{24} -e_{013} -e_{014} -e_{024} -e_{025} -e_{123} -e_{125} -e_{345}$.  Then $(D,c)$ is a characteristic 2 pair.  
In particular, $\boldsymbol{D_{38}}\not\in \overline{\text{Eff}}(\overline{M}_{0,7,\overline{\mathbb{Q}}})$
\end{lem}
\begin{figure}[ht]
\includegraphics[width=0.6\linewidth]{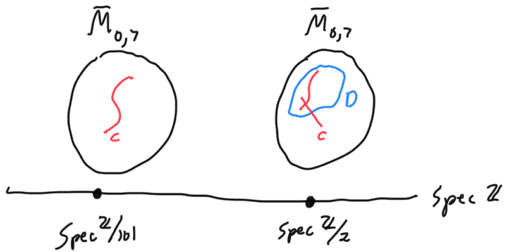}
\caption{A picture of Lemma \ref{chardif2}.}
\end{figure}
\begin{proof}
By Table \ref{ExtDivs2}, $D\in\text{Eff}(\overline{M}_{0,7,\overline{\mathbb{F}}_2})$ is an extreme effective divisor, swept out by rational curves of class $c-e_{345}$.  A quick check shows that $h^0(D)=0$ over characteristic $0$.  To see that $c.D' \geq 0$ for any integral Cartier divisor $D'\neq D$ over $\overline{\mathbb{F}}_2$, we will construct a nodal representative of $c$.

Let $\pi : \overline{M}_{0,7}\rightarrow\mathbb{P}^4$ be iterated blow-up of Kapranov's construction.  Over $\mathbb{F}_{32}$, the curve class $c-e_{345}$ is representable by the strict transform of a map $\mathbb{P}^1\rightarrow\mathbb{P}^4$ given on $\mathbb{A}^1$ by 
\begin{align*}t\rightarrow [&t^4 + t^3 + (z^4 + z + 1)t^2 + (z^4 + z^3 + z^2 + 1)t + z^4, \\
&zt^4 + (z^4 + z^3 + 1)t^3 + (z^3 + z^2 + z + 1)t^2 + (z^4 + z^3 + z^2 + z)t + z^3 + z^2 + z + 1, \\
&z^2t^4 + (z^2 + z + 1)t^3 + (z^4 + z^2)t^2 + (z^2 + z)t + z^3 + z^2 + z + 1, \\
&z^3t^4 + (z^2 + z)t^3 + (z^4 + z^3 + z^2 + z)t^2 + t + z^3 + z^2 + z + 1, \\
&(z^4 + z^3 + z^2 + 1)t^4 + (z^3 + z)t^3 + (z^4 + z^3 + z^2 + z)t^2 + (z^4 + z^3)t + z^3 + z^2 + z + 1] \end{align*}
where $z$ is a root of the Conway Polynomial of degree $5$ over $\mathbb{F}_2$, $x^5 + x^2 + 1$.  Let $g:\mathbb{P}^1\rightarrow\overline{M}_{0,7}$ be this strict transform.  Since $(c-e_{345}).D = -2$, $\mathcal{N}_g \cong \mathcal{O}^2 \oplus \mathcal{O}(-2)$, and $D$ is nonsingular along $g(\mathbb{P}^1)$, it follows (see Lemma \ref{deformationLemma}) that $c-e_{345}$ is a negative curve class on $D$ over characteristic 2.  Similarly, since the curve class $e_{345}$ is a negative curve class on $E_{345}$, $c.E_{345} =1$, and $c.D = -1$, the curve class $c=(c-e_{345})+e_{345}$ is a ``nodal negative curve class'' on $D$ over characteristic 2; that is, every integral effective Cartier divisor $D'\neq D$ over characteristic $2$ satisfies $c.D' \geq 0$.

Over $\mathbb{F}_{101}$, the curve class $c$ is representable by the strict transform of a map $\mathbb{P}^1\rightarrow\mathbb{P}^4$ given on $\mathbb{A}^1$ by 
\begin{align*} t\rightarrow [&33t^4 + 88t^3 + 15t^2 + 32t + 28, \\
&t^4 + 16t^3 + 84t^2 + 75t + 6, \\
&t^4 + 92t^3 + 15t^2 + 53t + 28, \\
&t^4 + 91t^3 + 88t^2 + 75t + 28, \\
&t^4 + 88t^3 + 62t^2 + 3t + 28]. \end{align*}
Let $f:\mathbb{P}^1\rightarrow\overline{M}_{0,7}$ be this strict transform.  Computations show $\mathcal{N}_f \cong \mathcal{O}^2 \oplus \mathcal{O}(-1)$.  With $k'=\mathbb{F}_{101}$, this proves $D,c$ satisfy all criteria of Proposition \ref{chardif}. Therefore, $(D,c)$ is a characteristic 2 pair.
\end{proof}

\begin{rem}{\label{reducible_remark}}
There are no irreducible curves of class $c$ (from Lemma \ref{chardif2}) over characteristic 2 with normal bundle $\mathcal{N}_{f}\cong V \oplus \mathcal{O}(-1)$ for some globally generated $V$.  Indeed, necessarily $\text{deg}(V) = 0$, so by Corollary \ref{lifting_remark} $D=\boldsymbol{D_{38}}$ would be effective over characteristic 0.  By Theorem \ref{charThm}, this is a contradiction.
\end{rem}

\begin{rem}\label{chardif_remark}
It is notable that every representative of $c-e_{345}$ (as above) over $\mathbb{F}_{32}$ is the strict transform, under $\pi :\overline{M}_{0,7}\rightarrow \mathbb{P}^4$, of a curve tangent to $\pi(E_{345})$ at their unique point of intersection.  These strict transforms meet $E_{345}$ at a single point with multiplicity two.  We may obtain a nodal curve of class $c$ by attaching the $\mathbb{P}^1$-fiber of $\pi$ over this unique intersection point.  The resulting curve has a normal bundle that restricts to $\mathcal{O}^2\oplus \mathcal{O}(-1)$ on the component of class $c-e_{345}$, and to $\mathcal{O}^3$ on the component of class $e_{345}$.  Using Hilbert schemes instead of Morphism schemes, it follows that this curve is unobstructed and lifts to a rational curve over characteristic $0$.  Since a generic deformation of the component of class $c-e_{345}$ has no lift to characteristic 0, a generic lift of this nodal curve must smooth the node.  

It is natural to expect that irreducible curves of class $c$ sweep out a divisor $D'$ of class $D+E_{345}$ over characteristic 0; however, this does not happen.  Though we could not find $D'$ explicitly, we verified that it must be the strict transform of degree $d>23$ divisor in $\mathbb{P}^4$.  The specialization of $D'$ to characteristic 2 must be reducible, and the locus $D'\setminus D$ must be swept out by reducible curves of class $c$ (for example, the union of two conics, one of which sweeps out a component of $D'\setminus D$, the other contained in $D$).  In particular, $c$ admits other nodal representatives over characteristic 2, whose components sweep out divisors other than $D$. 
\end{rem}

\begin{rem}\label{strange behavior}
Let $M\subset \overline{M}_{0,0}(\overline{M}_{0,7,\overline{\mathbb{F}}_2},c-e_{345})$ be a component of the Kontsevich space on $\overline{M}_{0,7}$ generically parameterizing negative curves on $\boldsymbol{D_{38}}$ of class $c-e_{345}$.  As Remark \ref{chardif_remark} indicates, it appears that a generic curve paramterized by $M$ is not transverse to $E_{345}$.  If true, this is quite anomolous.  By \cite[Proposition~2.8]{beheshti2020moduli}, over characteristic 0 and on a smooth variety, free rational curves general in their deformation classes intersect any chosen divisor transversely.  The negative curves $M$ parameterizes are not free in $\overline{M}_{0,7}$, but are free in $\boldsymbol{D_{38}}$.  Recall that the smooth locus of $\boldsymbol{D_{38}}$ contains the image of each negative curve.  Consequently, the curves $c-e_{345}$ and $\boldsymbol{D_{38}}$ would realize a scenario impossible over characteristic 0: a deformation class of free curves whose general member is not transverse to the divisor $E_{345}|_{\boldsymbol{D_{38}}}$.  
\end{rem}

Proposition \ref{chardif} and Lemma \ref{chardif2} prove Theorem \ref{charThm} when $n=7$.  To prove $\overline{\text{Eff}}(\overline{M}_{0,n})$ is different over characteristic 2 and characteristic 0 for $n > 7$, one could apply an argument of Castravet, Laface, Tevelev, and Ugaglia (Lemma 2.2 in \cite{castravet2020blownup}).  Recall that $\overline{\text{Mov}}_1(X)$ is dual to $\overline{\text{Eff}}(X)$ for any smooth, projective $X$ \cite{Boucksom_2012}.  Since a composition of forgetful maps $\pi : \overline{M}_{0,n}\rightarrow \overline{M}_{0,7}$ is surjective, 
Corollary 3.12 of \cite{fulger2016zariski} shows $\pi_*(\overline{\text{Mov}}_1(\overline{M}_{0,n}))=\overline{\text{Mov}}_1(\overline{M}_{0,7})$.  If $\overline{\text{Eff}}(\overline{M}_{0,n})$ were the same over characteristic $2$ and characteristic $0$ for some $n > 7$, by duality the same is true of $\overline{\text{Mov}}_1(\overline{M}_{0,n})$, but then $\pi_*(\overline{\text{Mov}}_1(\overline{M}_{0,n}))=\overline{\text{Mov}}_1(\overline{M}_{0,7})$ and the duality of $\overline{\text{Eff}}(\overline{M}_{0,7})$ implies $\overline{\text{Eff}}(\overline{M}_{0,7})$ is the same over characteristic $0$ and characteristic $2$, a contradiction.  This is a clean argument, but does not provide an explicit difference between $\overline{\text{Eff}}(\overline{M}_{0,n})$ over characteristic $2$ and characteristic $0$.  To obtain this, we must lift curves from $\overline{M}_{0,7}$.  We proceed inductively on $n$.  



\begin{prop}\label{chardif3}
Suppose $(D,c)$ is a characteristic $p$ pair on $\overline{M}_{0,n}$ defined over a field $k$.  
Suppose further that over $\overline{k}$, $c$ has a nodal representative with two components $f:\mathbb{P}^1 \rightarrow \overline{M}_{0,n}$ and $g:\mathbb{P}^1 \rightarrow \overline{M}_{0,n}$ satisfying the following conditions:
\begin{enumerate}
\item $D$ is nonsingular along $f(\mathbb{P}^1)$ and $\mathcal{N}_{f}\cong V \oplus \mathcal{O}(f_*[\mathbb{P}^1].D)$ for some globally generated $V$, 
\item $g(\mathbb{P}^1) \subset E \cong \overline{M}_{0,n-1}\times \overline{M}_{0,4}$ is the $\mathbb{P}^1$-fiber of the first projection map over a point in $M_{0,n-1}\subset \overline{M}_{0,n-1}$, and $E\subset \overline{M}_{0,n}$ is an irreducible boundary divisor,
\item $f(\mathbb{P}^1)\cap M_{0,n} \neq \emptyset$. 
\end{enumerate}
Let $\phi : \overline{M}_{0,n+1}\rightarrow \overline{M}_{0,n}$ be a forgetful morphism.  Then there exists a curve class $\tilde{c}\in N_1(\overline{M}_{0,n+1})$ with representatives over $\overline{k}$ and $\overline{k'}$, mapping to the corresponding representatives of $c$ via composition with $\phi$, such that 
$(\phi^*D,\tilde{c})$ is a characteristic $p$ pair on $\overline{M}_{0,n+1}$ satisfying the additional conditions of this proposition.
\end{prop}
\begin{rem}
Under slightly different assumptions, we may remove reference to the representatives over $k'$ and $\overline{k'}$.  When $\text{deg } V = 0$, Corollary \ref{lifting_remark} implies $f_*[\mathbb{P}^1].D < -1$ since $D$ is not effective in characteristic 0.  For degree reasons, we must have $f_*[\mathbb{P}^1].D = -2$ and $g_*[\mathbb{P}^1].D =1$.  As in Remark \ref{chardif_remark}, as long as the image of $f$ is transverse to the image of $g$ at a point of intersection, we may lift this nodal representative to characteristic 0 using the Hilbert scheme, smoothing the node in the process.  A general lift provides a representative of $c$ over characteristic 0 satisfying condition \ref{chardif}(3).  We may likewise lift and smooth the nodal representative for $\tilde{c}$.  
\end{rem}

\begin{proof}
The map $\phi$ is a blow-up of the projection $\text{Bl}_{\text{pt}}\mathbb{P}^{n-2}\rightarrow \mathbb{P}^{n-3}$, where the preimage of $\text{pt}=[0:\cdots : 0:1]\in\mathbb{P}^{n-2}$ under $\overline{M}_{0,n+1}\rightarrow \mathbb{P}^{n-2}$ is the locus of curves whose $n^{th}$ and $(n+1)^{st}$ marking collide. Let $\pi : \overline{M}_{0,n}\rightarrow \mathbb{P}^{n-3}$ denote the iterated blow-up 
of $\mathbb{P}^{n-3}$ at $n-1$ linearly general points $p_i$, $0\leq i \leq n-2$, followed by the strict transforms of the linear subspaces they span, in order of increasing dimension.  Likewise, let $\tilde{\pi}:\overline{M}_{0,n+1}\rightarrow \mathbb{P}^{n-2}$ denote the analogous iterated blow-up of $\mathbb{P}^{n-2}$ at $n$ linearly general points $\tilde{p_i}$, $0\leq i \leq n-1$, with $\tilde{p}_{n-1}=\text{pt}$ (followed by the strict transforms of the linear spaces they span), chosen such that $\tilde{p_i}$ maps to $p_i$ under $\text{Bl}_{\text{pt}}\mathbb{P}^{n-2}\rightarrow \mathbb{P}^{n-3}$.  Let $E_I$ and $\tilde{E}_I$ denote the irreducible exceptional divisors lying over the linear span of $\{p_i \text{ for } i\in I\}$ and $\{\tilde{p_i} \text{ for } i\in I\}$, respectively.  Note that $\phi^*(E_I)=\tilde{E}_I\cup \tilde{E}_{I\cup \{n-1\}}$.

Let $h:\mathbb{P}^1\rightarrow \overline{M}_{0,n,k'}$ be the representative of $c= h_*[\mathbb{P}^1]$ described by condition \ref{chardif}(3).  The map $\pi \circ h : \mathbb{P}^1 \rightarrow \mathbb{P}^{n-3}$ is given by $t\rightarrow [h_0(t) : h_1(t) : \cdots : h_{n-3}(t)]$, where we may assume $\pi \circ h(t=\infty)\in M_{0,n}\subset \mathbb{P}^{n-3}$.  Consider $\tilde{h}:\mathbb{P}^1\rightarrow\overline{M}_{0,n+1}$ given by the strict transform of $\tilde{\pi}\circ \tilde{h} : \mathbb{P}^1\rightarrow \mathbb{P}^{n-2}$, $t\rightarrow  [h_0(t) : h_1(t) : \cdots : h_{n-3}(t): h_{n-2}]$, where $h_{n-2} \in \overline{k'}$ is a general constant.  Note that $\phi \circ \tilde{h} = h$, as desired.

First, we argue that for general $h_{n-2}$, $\tilde{h}(\mathbb{P}^1)$ does not intersect $\tilde{E}_I$ with $n-1\not\in I$, and intersects those $\tilde{E}_I$ with $I=I'\cup  \{n-1\}$ only at points $t=t_0$ where $h(\mathbb{P}^1)$ intersects $E_{I'}$, and with the same multiplicity.  If $p_i$ and $\tilde{p_i}$ are the $i^{th}$ coordinate points for $i\leq n-3$ and $\tilde{p}_{n-2}=[1:\cdots : 1]$, this happens when $h_{n-3}\neq 0$ and  $h_{n-3}\neq h_i(t_0)$ for all roots $t_0$ of $h_i(t) - h_j(t)$.

Next, we show that $\mathcal{N}_{\tilde{h}}$ satisfies condition \ref{chardif}(3).  Let $d= \text{deg}(\pi_*(c))$, and consider the pullback of the Euler sequences on $\mathbb{P}^{n-3}$ and $\mathbb{P}^{n-2}$ by $\pi\circ h$ and $\tilde{\pi}\circ\tilde{h}$.  Since $\phi \circ \tilde{h} = h$ and $\phi$ is the blow-up of the projection $\pi_{\text{pt}}: \text{Bl}_{\text{pt}}\mathbb{P}^{n-2}\rightarrow \mathbb{P}^{n-3}$, it follows that $\pi_{\text{pt}}\circ \tilde{\pi}\circ\tilde{h} = \pi \circ h$ and there is a commutative diagram of Euler-sequence pullbacks, extending the natural map $\theta_2: (\tilde{\pi}\circ\tilde{h})^*\mathcal{T}_{\mathbb{P}^{n-2}} \rightarrow (\pi_{\text{pt}}\circ \tilde{\pi}\circ\tilde{h})^* \mathcal{T}_{\mathbb{P}^{n-3}}$.
\[ \begin{tikzcd}
0 \arrow{r} & \mathcal{O}_{\mathbb{P}^1} \arrow{r} \arrow[swap]{d}{\text{id}} & \oplus_{n-1}\mathcal{O}_{\mathbb{P}^1}(d) \arrow{r} \arrow{d}{\theta_1} & (\tilde{\pi}\circ\tilde{h})^*\mathcal{T}_{\mathbb{P}^{n-2}} \arrow{r} \arrow{d}{\theta_2} & 0 \\%
0 \arrow{r} & \mathcal{O}_{\mathbb{P}^1} \arrow{r}& \oplus_{n-2}\mathcal{O}_{\mathbb{P}^1}(d) \arrow{r} & (\pi \circ h)^* \mathcal{T}_{\mathbb{P}^{n-3}} \arrow{r} & 0
\end{tikzcd}
\]
The map of middle terms 
$\theta_1: \oplus_{n-1}\mathcal{O}(d)\rightarrow \oplus_{n-2}\mathcal{O}(d)$ is projection onto the first $n-2$ factors, by the choice of $\text{pt}$.  Moreover, the natural map $\mathcal{T}_{\mathbb{P}^1}\rightarrow (\pi\circ h)^* \mathcal{T}_{\mathbb{P}^{n-3}}$ is the composition of $\mathcal{T}_{\mathbb{P}^1}\rightarrow (\tilde{\pi}\circ \tilde{h})^* \mathcal{T}_{\mathbb{P}^{n-2}}$ with $\theta_2$, which defines a surjection $\theta_3: \mathcal{N}_{\tilde{\pi}\circ \tilde{h}}\rightarrow \mathcal{N}_{\pi \circ h}$.  By repeated application of the snake lemma, we find $\ker{\theta_3}\cong \ker{\theta_1}\cong \mathcal{O}(d)$.

We claim that $\mathcal{N}_{\tilde{h}}\subset \mathcal{N}_{\tilde{\pi}\circ \tilde{h}} $ is the preimage of $\mathcal{N}_{h}\subset \mathcal{N}_{\pi\circ h} $ under $\theta_3$. Indeed, this follows directly from the observation that each exceptional locus $\tilde{\pi}(\tilde{E}_I)$ which $\tilde{\pi}\circ\tilde{h}(\mathbb{P}^1)$ intersects is a linear space contracted by $\pi_{\text{pt}}$.  More concretely, Algorithm \ref{curveAlg} describes how to replace $\mathcal{N}_{\tilde{h}}$ and $\mathcal{N}_{h}$ with subsheaves $\tilde{\mathcal{F}}\subset \oplus_{n-1}\mathcal{O}(d)$ and $\mathcal{F}\subset \oplus_{n-2}\mathcal{O}(d)$, respectively.  There, it is easy to see that $\tilde{\mathcal{F}}$ is the preimage of $\mathcal{F}$ under $\theta_1$.  Thus, we obtain an exact sequence $$0\rightarrow \mathcal{O}(d)\rightarrow \mathcal{N}_{\tilde{h}} \rightarrow \mathcal{N}_{h} \rightarrow 0.$$
As $d\geq 0$, $h^1(\mathcal{N}_{\tilde{h}}) = h^1(\mathcal{N}_{h}) = 0$ and  $h^1(\mathcal{N}_{\tilde{h}}(-1)) = h^1(\mathcal{N}_{h}(-1)) \leq 1$.  Thus, $\mathcal{N}_{\tilde{h}}$ satisfies condition \ref{chardif}(3).

Since $\tilde{c} = \tilde{h}_*[\mathbb{P}^1]$ and $\tilde{h}_*[\mathbb{P}^1].\phi^*(D)=(\phi \circ \tilde{h})_*[\mathbb{P}^1].D= h_*[\mathbb{P}^1].D=-1$, $\tilde{c}$ also satisfies condition \ref{chardif}(1).  Finally, constructing a nodal representative satisfying the hypotheses of this proposition will show $\tilde{c}$ satisfies condition \ref{chardif}(2).

To construct such a nodal representative of $\tilde{c}$, consider the representative of $c$ with components  $f:\mathbb{P}^1 \rightarrow \overline{M}_{0,n}$ and $g:\mathbb{P}^1 \rightarrow \overline{M}_{0,n}$.  We may construct $\tilde{f}$ with $\phi \circ \tilde{f} = f$ and $0\rightarrow \mathcal{O}(d)\rightarrow \mathcal{N}_{\tilde{f}} \rightarrow \mathcal{N}_{f} \rightarrow 0$ the same way we constructed $\tilde{h}$.  As $\tilde{f}_*[\mathbb{P}^1]. \phi^*(D) = f_*[\mathbb{P}^1].D$, this proves $f$ satisfies the first and third criteria of this proposition.  Let $E = E_I$.  We may pick $\pi: \overline{M}_{0,n}\rightarrow \mathbb{P}^{n-3}$ such that the restriction of $\pi$ to $E$ is the projection $\overline{M}_{0,n-1}\times \overline{M}_{0,4}\rightarrow \overline{M}_{0,n-1}$. Since $\tilde{f}_*[\mathbb{P}^1].\tilde{E}_{I\cup \{n-1\}} = f_*[\mathbb{P}^1].E > 0$, we may attach the $\mathbb{P}^1$ fiber (giving $\tilde{g}$) over a point of $\tilde{\pi}(\tilde{f}(\mathbb{P}^1)\cap \tilde{E}_{I\cup \{n-1\}})\subset \tilde{\pi}(\tilde{E}_{I\cup \{n-1\}}) = \mathbb{P}^{n-3}$.  That we may choose such a point in $M_{0,n}\subset\mathbb{P}^{n-3}$ is guaranteed by our requirements on $f$ and the construction of $\tilde{f}$.
\end{proof}

This proves the first half of Theorem \ref{charThm2}.  To prove the latter half of Theorem \ref{charThm2}, we apply the following proposition.

\begin{prop}\label{chardif4}
Let $X\rightarrow\mathbb{Z}$ be a variety with natural identifications $N^1(X_k)\cong N^1(X_{\overline{\mathbb{Q}}})=:N^1(X)$ for every field $k$.  Suppose $F\in N^1(X)$ is the class of an integral divisor on $X_{\overline{\mathbb{Q}}}$, with $h^0(F, X_{\overline{\mathbb{Q}}})=1$.  Let $k$ be a field of positive characteristic such that $h^0(F,X_k) = 1$ and $F=F_1 + F_2$, with each $F_i$ the class of an integral divisor over $k$.  If $c$ is any curve class satisfying
\begin{enumerate}
\item $c = a c_1 + b c_2$ for $a,b \geq 0$ and $c_i$ the class of a negative curve on $F_i$ over $\overline{k}$,
\item $c.F = c.F_1 < 0$, while $c.F_2 \geq 0$,
\end{enumerate}
then $F\in \text{Eff}(X_{\overline{\mathbb{Q}}})$ is extreme, while $F\in \text{Eff}(X_{k})$ is not.
\end{prop}

\begin{proof}
As $F$ is the class of an integral divisor over $\overline{\mathbb{Q}}$, the claim is equivalent to showing $h^0(nF, X_{\overline{\mathbb{Q}}})=1$ for all $n> 0$.  Indeed, if $nF = A + B$ with $A,B\in \text{Eff}(X_{\overline{\mathbb{Q}}})$ not proportional to $F$, then $h^0(nF,X_{\overline{\mathbb{Q}}})\geq 2$.

Suppose to the contrary that $h^0(nF,X_{\overline{\mathbb{Q}}})\geq 2$. for some $n$.  It follows that $h^0(nF,X_{\overline{k}})\geq 2$.  Since $ c. F = c.F_1$,  the natural map $H^0(X_{\overline{k}}, n (F - F_1)) \rightarrow H^0(X_{\overline{k}}, n F)$ is a surjection.  Hence, $h^0(X_{\overline{k}}, n (F - F_1)) \geq 2$.  But $n(F-F_1) = n F_2$, and $h^0(X_{\overline{k}}, n F_2) = 1$ for all $n\geq 0$, since $c_2$ is a negative curve for $F_2$.
\end{proof}

\begin{lem}\label{chardif5}
Let $F = \boldsymbol{D_{31}}$ be the last entry of Table \ref{ExtDivs0}, and $c =5l-e_{01}-e_{02} -e_{12}-2e_{013}-2e_{024}-2e_{035}-2e_{045}-2e_{125}-2e_{134}-2e_{145}-2e_{234}-2e_{235} - e_{345}$.  Then $F,c$ satisfy the criteria of Proposition \ref{chardif4}.
\end{lem}

\begin{proof}
Note that $F\in \text{Eff}(\overline{M}_{0,7,\mathbb{Q}})$ is effective, integral, and rigid (see Table \ref{ExtDivs0}).  Its unique representative over $\mathbb{Z}$ specializes over $\mathbb{F}_2$ to a reducible divisor with two components of class $F'= F - E_{012}$ and $E_{012}$.  $F'$ appears in Table \ref{ExtDivs2} and is swept out by curves of class $c_1 =2l-e_{012}-e_{013}-e_{024}-e_{035}-e_{045}-e_{125}-e_{134}-e_{145}-e_{234}-e_{235}$.  By Table \ref{boundaryCurves}, over any characteristic, $E_{012}$ is swept out by curves of class $c_2 = l+2e_{012} -e_{01} -e_{02} -e_{12} + e_{345}$.  It follows that over characteristic 2, $F'$ is the only integral Cartier divisor which pairs negatively with curves of class $c = 2c_1 + c_2$, and $c.F = c.F' = -1$. 
\end{proof}

This proves our claim for $n=7$.  The claim for $n > 7$ follows from the fact that the pullback of and extreme ray $Z\in \text{Eff}^k(\overline{M}_{0,m})$ via a composition of forgetful morphisms $\pi: \overline{M}_{0,n}\rightarrow \overline{M}_{0,m}$ is again extreme, provided $Z$ generically parameterizes irreducible curves. 
To see this, reduce the the case $n=m+1$, note that the generic fiber of $\pi|_{\pi^{-1}(Z)}$ is an integral curve.  Thus, the index $e_\pi (\pi^{-1}(Z)) = \text{dim } \pi^{-1}(Z) - \text{dim } Z = 1$.  By Proposition 2.1 of \cite{Chen_2015}, if $\pi^{-1}(Z) = \sum a_i Z_i$ with $a_i > 0$ and $Z_i \subset \overline{M}_{0,n}$, then $e_\pi (Z_i)  \geq 1$ for all $i$. With $\text{gl}:\overline{M}_{0,m}\cong \overline{M}_{0,m}\times\overline{M}_{0,3} \rightarrow \overline{M}_{0,n}$ such that $\pi \circ \text{gl} = \text{id}$, this implies $Z=\text{gl}^* \circ \pi^* (Z) = \text{gl}^*(\sum a_i Z_i) = \sum_i a_i \text{gl}_*( \overline{M}_{0,m}). Z_i$.  As the intersection between $\text{gl}(\overline{M}_{0,m})$ and $Z_i$ is proper and $Z$ is extreme, the claim follows.

\section{Proof of Theorem \ref{extCycles} and Corollary \ref{higherCodim}}\label{Section5}
Theorem \ref{charThm2} implies $\overline{\text{Eff}}^k(\overline{M}_{0,n})$ is strictly larger over characteristic 2 than it is over characteristic 0, for all $1\leq k \leq n-6$.  This follows from the observation that an embedding of a boundary divisor $i:\overline{M}_{0,n}\rightarrow \overline{M}_{0,n+1}$ provides a section of the forgetful morphism $\overline{M}_{0,n+1}\rightarrow \overline{M}_{0,n}$.  This section identifies explicit extreme rays of $\overline{\text{Eff}}^k(\overline{M}_{0,n})$ and $\text{Eff}^k(\overline{M}_{0,n})$ over characteristic 2 that are not pseudoeffective in characteristic 0.  We begin this section with a proof of Theorem \ref{extCycles}.  We subsequently prove Corollary \ref{higherCodim}.  Lemma \ref{intersectionLem} and Proposition \ref{Proposition: codim2cycles} may be of independent interest.

As stated in our introduction, our proof of Theorem \ref{extCycles} follows work by \cite{Chen_2015}, \cite{schaffler2015cone}, and \cite{blankers2021extremality}.  Given a morphism $f: X\rightarrow Y$ between two complete varieties, for each subvariety $Z\subset X$ Chen and Coskun \cite{Chen_2015} define the index 
$$e_f(Z) = \text{dim } Z - \text{dim } f(Z).$$
Their Proposition 2.1 states that if $Z$ is numerically equivalent to an effective sum of subvarieties $Z_i$, i.e. $Z = \sum a_i Z_i \in N_k(X)$, with each $a_i >0$, then $e_f(Z_i) \geq e_f(Z)$ for all $i$.  For example, if $f:\text{Bl}_p \mathbb{P}^2 \rightarrow \mathbb{P}^2$ is the blow-up and $E$ denotes the exceptional curve, 
then \cite[Proposition~2.1]{Chen_2015} asserts $e_f(H-E),e_f(E)\geq e_f(H)$.  Chen and Coskun apply their result to various divisorial contractions $\overline{\mathcal{M}}_{g,n}\rightarrow Y$ to prove extremality of certain higher codimension cycles.  Below, we repeatedly reference \cite[Proposition~2.1]{Chen_2015} and its direct corollary \cite[Corollary~2.4]{Chen_2015}: for projective varieties $X$, $Y$, and $[Z]\in \text{Eff}^k(X)$ extremal, $[Z\times Y] \in \text{Eff}^k(X\times Y)$ is extremal.

Schaffler and Blankers \cite{schaffler2015cone} \cite{blankers2021extremality} extend results in \cite{Chen_2015} using certain birational reduction maps $\pi_\mathcal{A} : \overline{\mathcal{M}}_{g,n}\rightarrow \overline{\mathcal{M}}_{g,\mathcal{A}}$ defined in \cite{Hassett_2003}.  Our proof of the case $m>3$ below is identical to the $g=0$ case of \cite[Lemma~4.4]{blankers2021extremality}; however, a different approach is required for the case $m=3$.  Here, Schaffler \cite[Lifting~Lemma~(ii)]{schaffler2015cone} introduced the idea of using the reduction map $\pi: \overline{M}_{0,n}\rightarrow \overline{LM}_n$.  We patch an oversight in the original proof that mischaracterized the exceptional locus of $\pi$ by homotoping any $Z_i$ appearing in an expression $Z \times \overline{M}_{0,3}= \sum a_i Z_i$ to a rationally equivalent $V_i \subset \overline{M}_{0,n-1}\times \overline{M}_{0,3}$.  This verifies the original statement of \cite[Lifting~Lemma~(ii)]{schaffler2015cone} holds without alteration.

\begin{proof}[Proof of Theorem \ref{extCycles}]
\textbf{Case 1:} First, suppose $m > 3$.  
We will show $Z\times \overline{M}_{0,m}$ is extreme in $\text{Eff}^{k+1}(\overline{M}_{0,n})$ using a particular reduction map $\overline{M}_{0,n}\rightarrow \overline{M}_{0,\mathcal{A}}$.
By symmetry, we may suppose that $\overline{M}_{0,m}$ parameterizes the first $m-1$ markings.  
Assigning these weight $\frac{1}{m-1}$, we attain the space $Y=\overline{M}_{0,\mathcal{A}}$ where $\mathcal{A}=(\frac{1}{m-1},\ldots, \frac{1}{m-1}, 1, \ldots , 1)$.  Let $\pi :\overline{M}_{0,n}\rightarrow Y$ be the birational reduction map. By construction, the restriction $\pi \circ i : \overline{M}_{0,n-m+2}\times \overline{M}_{0,m}\rightarrow Y$ factors as the projection $\pi_1:\overline{M}_{0,n-m+2}\times \overline{M}_{0,m}\rightarrow \overline{M}_{0,n-m+2}$ followed by an inclusion $\overline{M}_{0,n-m+2}\rightarrow Y$.

Suppose $i_*(Z\times \overline{M}_{0,m})= \sum a_i Z_i$ for $a_i > 0$ and subvarieties $Z_i \subset \overline{M}_{0,n}$.  By Proposition 2.1 in \cite{Chen_2015}, since $e_{\pi} (Z\times \overline{M}_{0,m}) = m-3 > 0$, we must have $e_{\pi} (Z_i) \geq m-3$ for all $i$. 
However, the only points in $Y$ where the fiber dimension of $\pi$ is at least $m-3$ is precisely the locus $\pi \circ i (\overline{M}_{0,n-m+2}\times \overline{M}_{0,m})$.  Since $\pi^{-1}(\pi \circ i (\overline{M}_{0,n-m+2}\times \overline{M}_{0,m})) =  i (\overline{M}_{0,n-m+2}\times \overline{M}_{0,m})$, this proves each $Z_i$ must be a subvariety of $i (\overline{M}_{0,n-m+2}\times \overline{M}_{0,m})$, i.e. $Z_i = i_*(V_i)$.  Letting $\psi : \overline{M}_{0,n}\rightarrow \overline{M}_{0,n-m+2}\times \overline{M}_{0,m}$ be the product of forgetful maps (forgetting the first $m-2$ and last $n-m$ markings), we see $\psi_* \circ i_* = \text{id}_* : A(\overline{M}_{0,n-m+2}\times \overline{M}_{0,m} )\rightarrow A(\overline{M}_{0,n-m+2}\times \overline{M}_{0,m} )$.  Thus, $Z\times \overline{M}_{0,m}= \sum a_i V_i$, and by \cite[Corollary~2.4]{Chen_2015} we see $V_i = Z\times \overline{M}_{0,m}$ for all $i$.

\textbf{Case 2:} Instead, suppose $m=3$.  We have $Z\subset \overline{M}_{0,n-1} \cong \overline{M}_{0,n-m+2}\times \overline{M}_{0,m}\xrightarrow{i} \overline{M}_{0,n}$.  
By symmetry, we may assume that $\overline{M}_{0,n-m+2}=\overline{M}_{0,n-1}$ parameterizes the first $n-2$ markings. Assigning these weight $\frac{1}{n-2}$, we attain the Losev-Manin space \cite{10.1307/mmj/1030132728} $\overline{LM}_n=\overline{M}_{0,\mathcal{A}}$ where $\mathcal{A}=(\frac{1}{n-2},\ldots, \frac{1}{n-2}, 1 , 1)$. Let $\pi :\overline{M}_{0,n}\rightarrow \overline{LM}_n$ be the birational reduction map.  We may suppose that $\text{dim}(Z) = \ell > 0$, as otherwise $Z$ is a point and the claim is trivial.

Suppose $i_*(Z\times \overline{M}_{0,3})= \sum a_i Z_i$ for $a_i > 0$ and $\ell$-dimensional subvarieties $Z_i \subset \overline{M}_{0,n}$.  By Proposition 2.1 in \cite{Chen_2015}, since $e_{\pi} (Z\times \overline{M}_{0,3}) = \ell > 0$, we must have $e_{\pi} (Z_i) = \ell$ for all $i$.  Thus, each $Z_i$ is contracted to a point $p_i$ by $\pi$.  By analyzing the fiber of $\pi$ over $p_i$, we will construct a rationally equivalent subvariety $V_i \subset \overline{M}_{0,n-1}\times \overline{M}_{0,3}$ for each $Z_i$.  Our claim follows directly from this fact, as $\psi_* \circ i_* = \text{id}_* : A(\overline{M}_{0,n-1})\rightarrow A(\overline{M}_{0,n-1})$ implies $Z= \sum a_i V_i$, which shows $V_i$ is proportional to $Z$ for all $i$.

Consider the nodal curve $C$ corresponding to $\pi(Z_i)=p_i \in \overline{LM}_n$.  Let $s_i$, $1\leq i \leq n-2$ be the marked points with weight $\frac{1}{n-2}$, and $s_{n-1}, s_n$ be the marked points with weight $1$.  The curve $C$ is a chain of rational curves, $C_1, C_2, \ldots C_\alpha$, with $s_{n-1} \in C_1$ and $s_n \in C_\alpha$.
\begin{figure}[ht]
\includegraphics[width=0.7\linewidth]{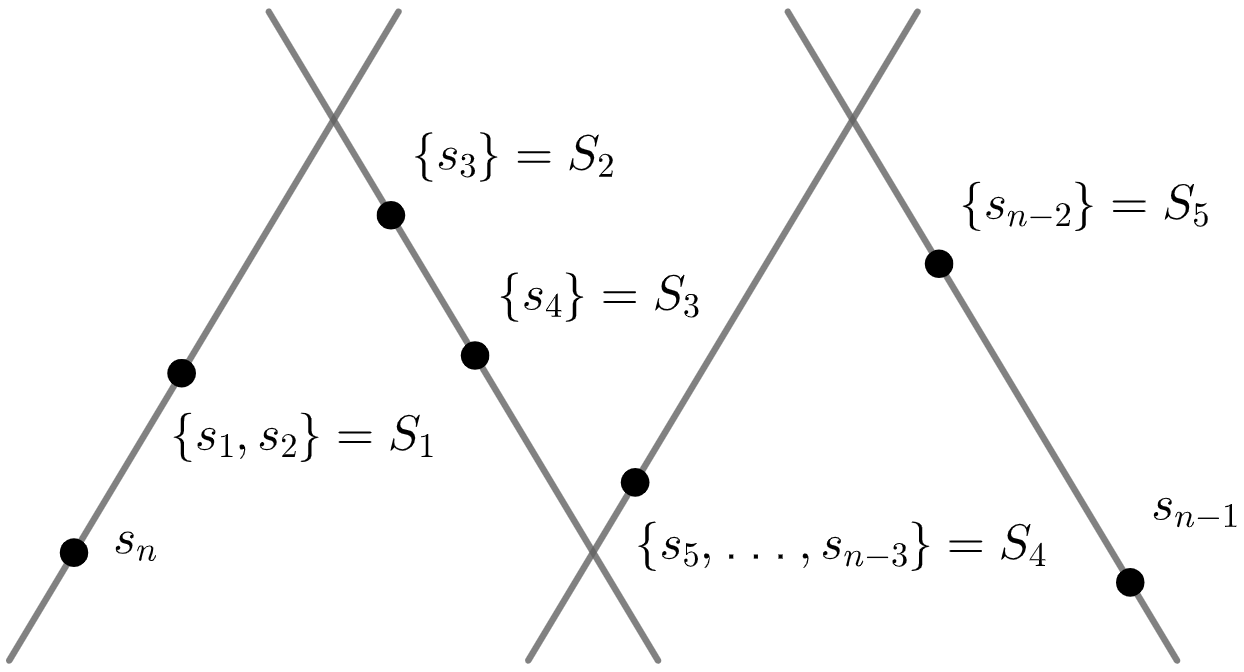}
\caption{An example of $C$ corresponding to $\pi(Z_i) \in \overline{LM}_n$.}
\label{LM image}
\end{figure}
Let $S_1, S_2, \ldots , S_l$ partition the remaining $s_i$ such that $s_i, s_j \in S_k$ iff they correspond to the same point in $C$.  It follows that the fiber $\pi^{-1}(p_i)\subset \overline{M}_{0,n}$ is isomorphic to $\prod_k \overline{M}_{0,S_k\cup \{ *_k \}}$, where if $|S_k| =1$ we let $\overline{M}_{0,S_i\cup \{ *_i \}}$ be a point.  If $C$ were irreducible and all markings of weight $\frac{1}{n-2}$ coincided, i.e. $S_1=\{1,\ldots , n-2\}$, the point $p_i$ would already be the point $\pi(\overline{M}_{0,n-1}\times \overline{M}_{0,3})$, and there would be no need to homotope $Z_i$ to $V_i$.  To obtain such a stable marked curve from $C$, we may slide the $n^{th}$ marking along each component until it coincides with the $(n-1)^{st}$ marking.  Stabilizing the result produces our desired curve.  We then translate the fiber $\pi^{-1}(p_i)$ along $C$ to obtain an appropriate $V_i \subset \overline{M}_{0,n-1}\times \overline{M}_{0,3}$.

To make this rigorous, we modify our base $C$ slightly if necessary.  If only one $S_i \subset C_\alpha$, let $C\rightarrow B$ be the contraction of $C_\alpha$; otherwise let $B = C$.  We will construct a map $B\times \prod_k \overline{M}_{0,S_k\cup \{ *_k \}} \rightarrow \overline{M}_{0,n}$ which embeds $s_{n} \times \prod_k \overline{M}_{0,S_k\cup \{ *_k \}}$ in  $\pi^{-1}(p_i)$ and $s_{n-1} \times \prod_k \overline{M}_{0,S_k\cup \{ *_k \}}$ in $\overline{M}_{0,n-1}\times \overline{M}_{0,3}$.

To define this map, we construct a family of stable curves over $B\times \prod_k \overline{M}_{0,S_k\cup \{ *_k \}}$.  The construction is natural: 
over each point $p\in B$, we begin with the nodal curve $B$ with marks given by $S_k \in C$, $k \leq l$, $s_{n-1}$, and  $p$.  If $p$ is a node of $B$ or $p$ collides with $S_k$ for some $k$, we stabilize $B$ by adding an extra component in the standard way.  This gives us a family of stable curves with $l+2$ markings.  Lastly, at each point $S_k$ we glue a stable curve, parameterized by $\overline{M}_{0,S_k\cup \{ *_k \}}$, along $*_k$, by considering $B$ as embedded in $\overline{M}_{0,l+2}$ and using the appropriate clutching map.

This family corresponds to a map $\phi : B\times \prod_k \overline{M}_{0,S_k\cup \{ *_k \}} \rightarrow \overline{M}_{0,n}$.  It defines an isomorphism $s_{n} \times \prod_k \overline{M}_{0,S_k\cup \{ *_k \}} \cong \pi^{-1}(p_i)$ and maps $s_{n-1} \times \prod_k \overline{M}_{0,S_k\cup \{ *_k \}}$ into $\overline{M}_{0,n-1}\times \overline{M}_{0,3}$.  Thus, by taking $B \times Z_i \subset B\times \pi^{-1}(p_i)$, we may replace $Z_i$ with the rationally equivalent cycle $V_i = \phi_*[s_{n-1}\times Z_i ]$.
\end{proof}
To prove Corollary \ref{higherCodim}, we need the following lemma as well.  This provides a numeric criterion that proves a given integral subvariety $V\subset\overline{M}_{0,n}$ must be contained in a specific boundary divisor.  For $I\subset \{1,\ldots ,n\}$, we let $\delta_I = \text{gl}(\overline{M}_{0,I \cup \{\bullet\}} \times \overline{M}_{0,I^{c} \cup \{*\}})$ be the boundary divisor generically parameterizing nodal curves with one component marked by $I$.

\begin{lem}\label{intersectionLem}
Let $I\subset \{1,\ldots , n\}$, $i\in I^c$, and $\pi : \overline{M}_{0,n}\rightarrow \overline{M}_{0,k}\cong \overline{M}_{0,I \cup \{\bullet\}}$ be 
the map forgetting all markings in $I^c\setminus\{i\}$.  Let $\psi_i$ denote the $i^{th}$ $\psi$-class and $\text{gl}: \overline{M}_{0,I \cup \{\bullet\}} \times \overline{M}_{0,I^{c} \cup \{*\}} \rightarrow \overline{M}_{0,n}$ be the clutching morphism.  
Suppose $V\subset \overline{M}_{0,n}$ is an irreducible $\ell$-fold s.t. $V.[\pi^*(\psi_i)^\ell - \psi_i^\ell] > 0$.  Then $V \subset \delta_I = \text{gl}(\overline{M}_{0,I \cup \{\bullet\}} \times \overline{M}_{0,I^{c} \cup \{*\}})$ and  $\pi_*(V) \neq 0$, .
\end{lem}

\begin{proof}
We will induct on $n$.  For $n = k+1$, $\pi^*(\psi_i) = \psi_i - \delta_{i,\alpha}$, where $\alpha$ is the marking forgotten by $\pi$ and $\delta_{i,\alpha}=\delta_I$ is the boundary divisor where the $i^{\text{th}}$ and $\alpha^{\text{th}}$ markings collide.  Since $(\psi_i - \delta_{i,\alpha})^\ell = \psi_i^\ell  + (- \delta_{i,\alpha})^\ell$, we see $V. (- \delta_{i,\alpha})^\ell > 0$.  But $ \delta_{i,\alpha} = \text{gl}(\overline{M}_{0,I \cup \{\bullet\}} \times \overline{M}_{0,I^{c} \cup \{*\}})$.  As $n-k+2 = 3$,  $\text{gl}^*( \delta_{i,\alpha}) = -\psi_{\bullet}$.  Thus $0 < V. (- \delta_{i,\alpha})^\ell = - \text{gl}_*(\text{gl}^*(V). \psi_{\bullet}^{\ell-1}))$.  However, this implies $\text{gl}^*(V). \psi_{\bullet}^{\ell-1} < 0$.  As $\psi_{\bullet}^{\ell-1}$ is nef, $\text{gl}^*(V)$ cannot be an effective $(\ell-1)$-cycle.  Thus, $V \subset \text{gl}(\overline{M}_{0,I \cup \{\bullet\}} \times \overline{M}_{0,I^{c} \cup \{*\}})$.  As $\text{gl}$ is a section of $\pi$, $\pi_*(V) \neq 0$

Suppose the lemma holds for $m < n$.  Consider a factorization of $\pi$ into $\pi = \pi' \circ \pi_\alpha$, where $\pi_\alpha : \overline{M}_{0,n}\rightarrow \overline{M}_{0,n-1}$ is a forgetful morphism forgetting $\alpha \in I^c$.  It is easy to derive from Kapranov's model that $\psi_i^\ell = \pi_\alpha^*(\psi_i)^{\ell-1}( \pi_\alpha^*(\psi_i) + \delta_{i,\alpha})$.  Thus,
\begin{align*}
0 < & \pi_{\alpha*}(V.[\pi^*(\psi_i)^\ell - \psi_i^\ell]) \text{ \hspace{6pt}  (recall } V.[\pi^*(\psi_i)^\ell - \psi_i^\ell] \text{ is a 0-cycle)}\\
 = & \pi_{\alpha*}(V.[\pi_\alpha^*(\pi'^*(\psi_i)^\ell - \psi_i^\ell) - (\pi_\alpha^*(\psi_i^{\ell-1}) . \delta_{i,\alpha})]) \\
 = & \pi_{\alpha*}(V).[\pi'^*(\psi_i)^\ell - \psi_i^\ell]  - (\pi_{\alpha*}(V.\delta_{i,\alpha}).\psi_i^{\ell-1}). \\
\end{align*}
Thus, as $\psi_i^{\ell-1}$ is nef, either 1) $V.\delta_{i,\alpha}$ is not effective, or  2) $\pi_{\alpha*}(V).[\pi'^*(\psi_i)^\ell - \psi_i^\ell] > 0$.  This statement remains true for any choice of a forgetful map $\pi_\alpha$ which factors $\pi$ (there are $n-k \geq 2$ choices for $\alpha$).  When 1) holds, we find $V\subset \delta_{i,\alpha}$.  When 2) holds, by induction we find $\pi_\alpha (V) \subset \text{gl}(\overline{M}_{0,I \cup \{\bullet\}} \times \overline{M}_{0,(I^{c}\setminus\{\alpha\}) \cup \{*\}})$.  Thus $V$ itself lies within one of the two irreducible components of $\pi_\alpha^{-1}(\text{gl}(\overline{M}_{0,I \cup \{\bullet\}} \times \overline{M}_{0,(I^{c}\setminus\{\alpha\}) \cup \{*\}}))$, which are $\delta_{I} = \text{gl}(\overline{M}_{0,I \cup \{\bullet\}} \times \overline{M}_{0,I^{c} \cup \{*\}})$ and $\delta_{I\cup \{\alpha\}}$. 
Thus, for each $\alpha \in I^c$ forgotten by $\pi$, either $V \subset \delta_{i,\alpha}$, $V \subset \delta_I$, or $V\subset \delta_{I\cup\{\alpha\}}$.

As mentioned, there are at least two choices for $\alpha \in I^c$ forgotten by $\pi$.  If $\alpha\neq \alpha'$,  $\delta_{i,\alpha} \cap \delta_{i,\alpha'} = \emptyset$, $\delta_{I\cup\{\alpha\}} \cap \delta_{I\cup\{\alpha'\}} = \emptyset$, and $\delta_{i,\alpha} \cap \delta_{I\cup\{\alpha'\}} = \emptyset$.  
It follows that we must have $V \subset \delta_I = \text{gl}(\overline{M}_{0,I \cup \{\bullet\}} \times \overline{M}_{0,I^{c} \cup \{*\}})$.  By induction, $\pi_*(V) = \pi'_*\circ \pi_{\alpha*}(V)\neq 0$.
\end{proof}

Lemma \ref{intersectionLem} allows us to prove the proposition below.  While the proof applies more generally for some boundary divisors $V$, it does not hold for all choices of boundary divisors $V$.  For instance, standard relations in the Chow ring of $\overline{M}_{0,n}$ show certain codimension 2 boundary strata are not rigid, a property implied by the proof below.  Notably, each codimension 2 boundary strata is extreme in $\text{Eff}^2(\overline{M}_{0,n})$ by Theorem \ref{extCycles}, even when not rigid.

\begin{prop}\label{Proposition: codim2cycles}
Suppose $V\subset \overline{M}_{0,n}$ is a nonboundary divisor swept out by negative curves of class $c$.  Let $i:\overline{M}_{0,n}\rightarrow \overline{M}_{0,n+1}$ be the embedding of a boundary divisor.  Then $i_* V \in \overline{\text{Eff}}^2(\overline{M}_{0,n+1})$ is extreme.
\end{prop}
\begin{rem}
Note that the proof of Proposition \ref{chardif3} allows us to lift negative curves on a divisor $V\subset \overline{M}_{0,n}$ to negative curves on $\phi^*V$, where $\phi : \overline{M}_{0,n+1}\rightarrow \overline{M}_{0,n}$ is a forgetful morphism.  In particular, we will use the liftings $\tilde{c}$ of the negative curve $c$ on $\boldsymbol{D_{38}}$ to prove Corollary \ref{higherCodim}.
\end{rem}
\begin{proof}

Let $\phi : \overline{M}_{0,n+1}\rightarrow \overline{M}_{0,n}$ be a forgetful map such that $\phi \circ i = \text{id}$.  Furthermore, for $I\subset \{1,\ldots , n+1\}$ such that $i(\overline{M}_{0,n-1})=\delta_I$ and $|I|=2$, let $\psi$ be the $\psi$-class corresponding to the marking in $I$ not forgotten by $\phi$.  The Kapranov map associated to $\psi$ realizes $V$ as the strict transform of a degree $d>0$ hypersurface.  With $\ell = \text{dim }V = n-4$, it follows that $\psi^{\ell} . V = \phi^*(\psi)^\ell . i_*(V) = d$, while $\psi^\ell . i_*(V) = 0$.  By the projection formula, $\phi^*(c).i_*(V)= c.V \leq -1$.

We show $i_*V\in \overline{\text{Eff}}^2(\overline{M}_{0,n+1})$ is extreme by proving every other integral $\ell$-fold $D$ satisfies $D . [\psi^\ell -(1+c.V)\phi^*\psi^\ell + \phi^*(c)] \geq 0$.  Let $D\subset \overline{M}_{0,n+1}$ be an irreducible, reduced $\ell$-fold.  Since $\psi^\ell$ and $\phi^*\psi^\ell$ are nef, if $\phi^*(c).D \geq 0$, our claim holds.  So, we may assume $\phi^*(c).D < 0$.  Note that $c -(c.V)\psi^\ell$ is a nef curve class on $\overline{M}_{0,n}$.  If $D . [\phi^*(\psi)^\ell - \psi^\ell] \leq 0$, then our claim follows from $D. [\phi^*(\psi)^\ell - \psi^\ell] \leq 0 \leq D.\phi^*(c -(c.V)\psi^\ell)$.  Hence, we may also assume $D . [\phi^*(\psi)^\ell - \psi^\ell] > 0$.  By Lemma \ref{intersectionLem}, this shows $D'\subset \text{gl}(\overline{M}_{0,n} \times \overline{M}_{0,3}) = i(\overline{M}_{0,n})$.  Since $c.\phi_*(D)=\phi^*(c).D <0$, $\phi_*(D)$ is a nonzero multiple of $V\subset \overline{M}_{0,n}$.  This implies $D = i_*V$.
\end{proof}

The proposition below proves Corollary \ref{higherCodim}.  

\begin{prop}
Let $\pi: \overline{M}_{0,n-k+1}\rightarrow \overline{M}_{0,7}$ be the composition of $n-k-6$ forgetful maps, and $i: \overline{M}_{0,n-k+1}\rightarrow \overline{M}_{0,n}$ be a composition of $k-1$ embeddings $\overline{M}_{0,n-k+m}\rightarrow \overline{M}_{0,n-k+m+1}$ of boundary divisors.  With $F=\boldsymbol{D_{31}}$ as in Theorem \ref{charThm2}, $i_*\pi^*F \in \text{Eff}^k(\overline{M}_{0,n})$ is extreme over characteristic 0 but not over characteristic 2.  If $k\leq 2$, then with $D=\boldsymbol{D_{38}}$ as in Theorem \ref{charThm2},  $i_*\pi^*D\in \overline{\text{Eff}}^k(\overline{M}_{0,n,\overline{\mathbb{F}}_2})$ is effective and extreme, but $i_*\pi^*D\not\in \overline{\text{Eff}}^k(\overline{M}_{0,n,\overline{\mathbb{Q}}})$.

\end{prop}

\begin{proof}
The statement about $i_*\pi^*F\in \text{Eff}^k(\overline{M}_{0,n})$ follows immediately from Theorems \ref{extCycles}, \ref{charThm2}, and Lemma \ref{chardif5}, as $F= (F-E_{012}) + E_{012}$ is reducible over characteristic 2.  For $k=1$, the statement about $D$ is Theorem \ref{charThm2}.  To prove the statement about $D$ for $k=2$, we use the negative curves of class $\tilde{c}$ on $\overline{M}_{0,n-k+1}\cong \overline{M}_{0,n-1}$ (constructed in Proposition \ref{chardif3}) that sweep out $\pi^*D$.  It follows immediately from Proposition \ref{Proposition: codim2cycles} that $i_*\pi^*D\in \overline{\text{Eff}}^k(\overline{M}_{0,n,\overline{\mathbb{F}}_2})$ is effective and extreme.  Were $i_*\pi^*D\in \overline{\text{Eff}}^k(\overline{M}_{0,n,\overline{\mathbb{Q}}})$, then by forgetting markings we find $\pi^*D\in \overline{\text{Eff}}(\overline{M}_{0,n,\overline{\mathbb{Q}}})$, which contradicts Theorem \ref{charThm2}.
\end{proof}



\section{Finding New Extreme Divisors on Arbitrary Varieties}\label{Algorithms}

As described in Section \ref{Notation}, let $X$ be a variety, $\mathcal{D}$ be a finite set of extreme divisors on $X$, and $\mathcal{C}$ be a finite collection of negative curves sweeping out divisors in $\mathcal{D}$.  We let $\mathfrak{E}$ be the cone generated by $\mathcal{D}$, and
$$\mathfrak{M}=\{ D\in N^1(X) | D.c \geq 0 \text{ for all } c\in \mathcal{C}\}.$$
To determine whether $\mathcal{D}$ contains every extreme divisor on $X$, we may check whether $\mathfrak{M}$ is contained in $\mathfrak{E}$.  If it is, this proves $\mathfrak{E}=\text{Eff}(X)$, since $\text{Eff}(X)\subseteq \mathfrak{E}+\mathfrak{M}$.  Otherwise, we may search rays in $\mathfrak{M}\setminus\mathfrak{E}$ to find new extreme divisors on $X$.  
Subsection \ref{sectionAlgExtDiv} describes algorithms for finding new extreme divisors and negative curves on $X$.  These rely upon subsidiary algorithms that test effectiveness of divisor classes and determine splitting types of $\mathcal{N}_{f}$ for general maps $f:\mathbb{P}^1\rightarrow X$ of fixed numeric class.
\begin{description}
\item[Algorithm \ref{divAlg}] Given $D\in N^1(X)$,\footnote{Technically, we need to lift $D$ to $\text{Pic}(X)$, but this is simple to do when $X$ is a blow-up of projective space.} compute $h^0(\mathcal{O}(D),X)$ and return a general member $D\in H^0(\mathcal{O}(D),X)$, factored into irreducible components.
\item[Algorithm \ref{curveAlg}] Given $c\in N_1(X)$, find a representative map $f:\mathbb{P}^1\rightarrow X$ with $f_*[\mathbb{P}^1]=c$ and compute the splitting type of $\mathcal{N}_f$.  If a divisor $D$ is also given, find $f(\mathbb{P}^1)$ meeting $D^{sm}$.
\end{description}
Section \ref{M07alg} describes our implementations of these algorithms for $X=\overline{M}_{0,7}$. 
These implementations are adaptable to other blow-ups $X$ of $Y\subset \mathbb{P}^n$ whose Cartier class group $\text{CaCl}(Y)$ is generated by hypersurface sections.

Wherever possible, it is best to incorporate symmetries of $X$ in the algorithms below. 
For $X=\overline{M}_{0,7}$, the natural action of $S_7$ on $\overline{M}_{0,7}$ was used to simplify linear computations involving $\mathfrak{E}$ and $\mathfrak{M}$, and to pick a simplest representative from a given orbit of divisors or curves as input to Algorithms \ref{divAlg} and \ref{curveAlg}.  This is outlined in section \ref{symmetry}.

\textbf{Overview:} The algorithm below uses linear programming to automatically generate guesses for new extreme divisors and negative curves on $X$.  Figure \ref{Ext Divs Alg} provides an example. 
\begin{figure}[ht]
\includegraphics[width=0.8\linewidth]{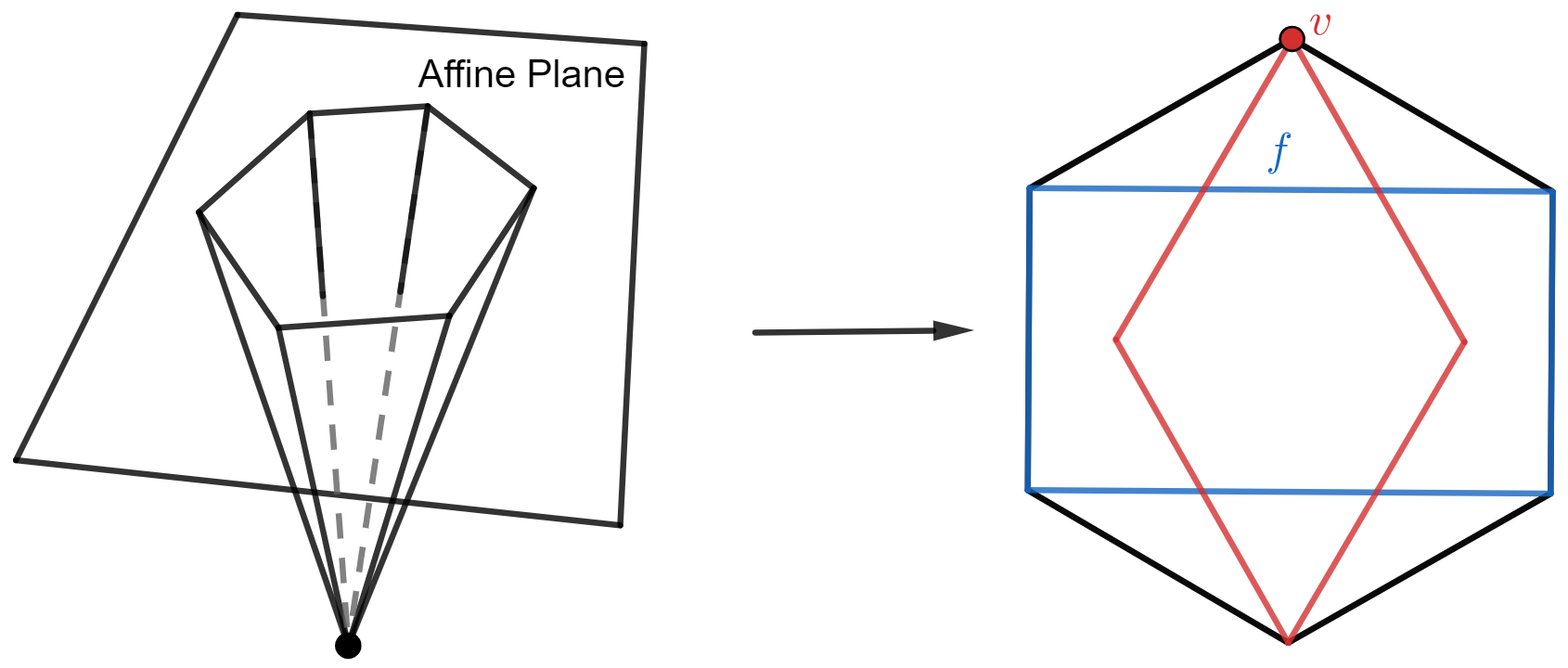}
\caption{An affine hyperplane section of $N^1(X)$ when $\rho(X) =3$, with $\mathfrak{E}$ outlined in blue, $\mathfrak{M}$ outlined in red, and $\text{Eff}(X)$ outlined in black.}
\label{Ext Divs Alg}
\end{figure}
By running a linear program minimizing the facet $f$, we identify the ray $v\in \mathfrak{M}$.  Repeating this procedure for other facets of $\mathfrak{E}$ generates further guesses $(v,f)$ for $(\text{extreme divisor}, \text{ negative curve})$ pairs on $X$.  Frequently, only one of $v$ or $f$ is an accurate guess.  We include protocols for finding a corresponding divisor or curve in such circumstances.

\subsection{Algorithm for Finding New Extreme Divisors}\label{sectionAlgExtDiv}  
To find new extreme divisors on $X$, generate a list of facets of $\mathfrak{E}$.  We used a polyhedral software package\footnote{see http://mathieudutour.altervista.org/Polyhedral/index.html } based on cdd\footnote{see https://github.com/cddlib } for this purpose.    
For each facet $f\in N_1(X)$ generated this way, run a linear program minimizing $f.v$ for $v\in\mathfrak{M}$.  If $f.v\geq 0$, then $f$ is a nef curve class; otherwise, if $f.v < 0$, then $v$ is a candidate for a new extreme divisor on $X$ and $f$ is a candidate for a negative curve class on $v$.  Multiple facets of $\mathfrak{E}$ may pair negatively with $v$, so we record pairs $(v,F)$, where $F$ is a set of such facets.  As the speed of Algorithms \ref{divAlg} and \ref{curveAlg} may vary greatly with input, there are two protocols for testing such data: \textit{divisor-first} and \textit{curve-first}.  We found that testing facets $f$ of anticanonical degree $f.(-K_X) = 1,2$, or $3$ using the \textit{curve-first} protocol was the most successful approach on $\overline{M}_{0,7}$.

\textbf{Divisor-First Protocol}: Test integer multiples of $v$ for effectiveness using Algorithm \ref{divAlg}.  If some multiple $nv$ is effective, pick $D\in |nv|$ and factor it into irreducible components.  If $f\in F$ pairs negatively with exactly one irreducible component $D'$ of $D$, then $f$ is a candidate for a negative curve on $D'$, which would therefore be an extreme divisor on $X$.  
Test $f$ with Algorithm \ref{curveAlg}. If the result shows $f$ is a negative curve via Lemma \ref{deformationLemma}, we may add $D'$ and $f$ to $\mathcal{D}$ and $\mathcal{C}$, respectively, and repeat this algorithm.

Certain irreducible components $D'$ of $D$ may not pair negatively with any facets $f\in F$ which Algorithm \ref{curveAlg} confirms are negative curves.  If small integer multiples of $D'$ are rigid and linear programming shows $D'$ is not contained in $\mathfrak{E}$, we may search for negative curves on $D'$ by finding more facets of $\mathfrak{E}$ that pair negatively with $D'$, or similarly by minimizing $f.D'$ for $f\in \mathfrak{E}^*$.  Testing integer points near such outputs $f$ may produce negative curves for $D'$ as well.

\begin{rem}
If $D$ is an integral and rigid divisor class but $D+E$ is not rigid, then any negative curve on $D$ pairs positively with $E$.  This may be used to inform guesses for negative curves on $D$.  On $\overline{M}_{0,7}$, we tested the rigidity of $D+E_I$ to deduce which boundary divisors $E_I$ negative curves on $D$ need meet.  
\end{rem}

\textbf{Curve-First Protocol}: 
Given a pair $(v,F)$, test facets $f\in F$ with Aglorithm \ref{curveAlg}.  If Algorithm \ref{curveAlg} returns a map $g:\mathbb{P}^1\rightarrow X$ with ($g_*[\mathbb{P}^1]=f$ and) $\mathcal{N}_g\cong \mathcal{O}(-n) \oplus V$ for some globally generated $V$, then $f$ is likely a negative curve class.  If it is, then the divisor it sweeps out belongs to the base locus of any effective divisor with which it pairs negatively.  So, it suffices to find any divisor $D\in \text{Eff}(X)$ satisfying $f.D < 0$.  If $v\notin \text{Eff}(X)$, we may check nearby integer points or use linear programming to find other $v\in\mathfrak{M}_{f<0}$.  Alternatively, computing the image of the universal family over $\text{Mor}_{f}(\mathbb{P}^1,X)$ may work for projective $X$ and low degree $f$.  This may be approximated by identifying multiple maps $g:\mathbb{P}^1\rightarrow X$, and finding a integral divisor $D$ that contains each image.  If some $g(\mathbb{P}^1)$ intersects $D^{sm}$ 
and $f.D = -n$, then deformation of $g$ sweep out $D$ and $f=g_*[\mathbb{P}^1]$ is a negative curve class on $D$.  As before, we may add $D$ and $f$ to $\mathcal{D}$ and $\mathcal{C}$, respectively, and repeat this algorithm.


\begin{rem}\label{addnefCurves}
For $X=\overline{M}_{0,7}$ and $f$ a curve of small degree under some Kapranov map, Algorithm \ref{curveAlg} frequently returned $g:\mathbb{P}^1\rightarrow \overline{M}_{0,7}$ with $\mathcal{N}_g$ globally generated, which shows that $f$ is nef.  Examples of such classes appear in Table \ref{nefCurves}.  These nef classes are of interest because they are necessarily extreme rays of $\text{Nef}_1(X)$.  Moreover, these nef classes are not expressible as positive linear combinations of negative curves in $\mathcal{C}$, so adding them to $\mathcal{C}$ shrinks $\mathfrak{M}$.  

All algorithms and statements in this section remain true if $\mathcal{C}$ is extended to include nef curves.
\end{rem}

\textbf{Cleanup}:  After finding more negative curves and appending them to $\mathcal{C}$, some curves in $\mathcal{C}$ may become redundant as inequalities defining $\mathfrak{M}$.  Before continuing computations, we recommend running linear programs to identify and discard these curves, as they are merely positive sums of other negative curves.  Incorporating an action of $\text{Aut}(X)$ here is crucial, as we may partition $\mathcal{C}$ into $\text{Aut}(X)$-orbits and test one representative from each orbit for redundancy.

\begin{rem}
Rather than generating a list of facets of $\mathfrak{E}$, we may initially generate extreme rays of $\mathfrak{M}$.  The algorithm for processing such a list is essentially dual to the algorithm above, but is not recommended because $\mathfrak{M}$ is often more complex than $\mathfrak{E}$.  Additionally, instead of the above approach, one may enumerate small integer points of $\mathfrak{M}$ (or $\mathfrak{E}^*$) and test those outside of $\mathfrak{E}$ (or $\mathfrak{M}^*$).  This may be used to find initial extreme divisors and negative curves on $X$, and worked for $X=\overline{M}_{0,7}$. 
\end{rem}

\begin{rem}
\textit{What if $\text{Eff}(X)$ is not polyhedral?}  If $\text{Eff}(X)$ is not polyhedral, sometimes it is possible to use extreme rays of $\text{Eff}(X)$ to find an infinite-order pseudo-automorphism of $X$.  If $X$ is a blow-up of $\mathbb{P}^n$, this may be done algorithmically on a computer.  We refer the reader to \cite{he2021birational} for an example.
\end{rem}


\section{Algorithms on $\overline{M}_{0,7}$}\label{M07alg}
This section describes the algorithms used to test divisor and curve classes on $\overline{M}_{0,7}$.\footnote{Code for these algorithms may be found at https://github.com/EricJovinelly/M07algorithms}  Let the 6 points $p_i\in \mathbb{P}^4$ (mentioned in Section \ref{Notation}) be
$$p_0= [1:0:0:0:0],\hspace{7pt} p_1= [0:1:0:0:0],\hspace{7pt} \ldots \hspace{7pt},\hspace{7pt} p_4= [0:0:0:0:1],\hspace{7pt} p_5= [1:1:1:1:1].$$


\subsection{Algorithm for Divisors}\label{divAlg} 
Given a divisor class 
$$D= dH -\sum_{0\leq i\leq 5} m_i E_i - \sum_{0\leq i < j \leq 5} m_{ij} E_{ij} -\sum_{0\leq i < j < k\leq 5} m_{ijk} E_{ijk},$$
the following describes an efficient procedure to test whether $D$ is effective, and to split a member of $|D|$ into integral components.  
\begin{rem}
This algorithm ran on a data set of 26,000 divisors.  When $d\leq 20$, the run time was usually a few milliseconds.  When $d\sim 30$, the algorithm took a few seconds.  As $d\rightarrow 50$, the run time grew to a few hundred seconds.  For $d\geq 40$, this algorithm is only practical to run on small sets of divisors.
\end{rem}

The space of sections of $D$ is identified via $\overline{M}_{0,7}\rightarrow\mathbb{P}^4$ with the subspace of $\text{H}^0(\mathbb{P}^4 , \mathcal{O}(d))$ that vanishes at $p_i$ with multiplicity at least $m_i$, vanishes along $l_{ij}$ with multiplicity at least $m_{ij}$, and vanishes along $\Delta_{ijk}$ with multiplicity at least $m_{ijk}$.  Our goal is to determine whether $\text{H}^0(\overline{M}_{0,7}, \mathcal{O}(D))\subset  \text{H}^0(\mathbb{P}^4 , \mathcal{O}(d))$ contains a nonzero element.  Let 
$$f(x_0,\ldots, x_4)=\sum_{I=(i_0,\ldots , i_4), |I|=d} a_I x^I$$
be a generic global section of $\mathcal{O}_{\mathbb{P}^4}(d)$.  
The subspace $\text{H}^0(\overline{M}_{0,7}, \mathcal{O}(D))\subset  \text{H}^0(\mathbb{P}^4 , \mathcal{O}(d))$ is defined by the following linear conditions imposed by $m_h, m_{hj}, m_{hjk}$ on the coefficients $a_I$:
\begin{itemize}
\item if $5\not\in \{h,j,k\}$, the linear conditions are monomial and specify the vanishing of $a_I$ with $i_h + i_j + i_k > d - m_{hjk}$ (similarly for $m_{h}$ and $m_{hj}$).
\item if $5\in \{h,j,k\}$, permute the marked points so that it isn't (this corresponds to replacing $f(x_0,\ldots ,x_4)$ with $g=f(x_\alpha -x_0, \ldots , x_\alpha, \ldots , x_\alpha - x_4)$ for some $\alpha\notin \{h,j,k\}$) and read off the corresponding monomial conditions on $g$ in terms of the $a_I$.
\end{itemize}
Computing $g$ becomes expensive as $d$ increases.  If $b_I$ is the coefficient of $x^I$ in $g$, and $\alpha =4$, then
$$b_I = (-1)^{i_0 + i_1 + i_2 + i_3}\sum_{J\geq I} {j_0 \choose i_0}{j_1 \choose i_1}{j_2 \choose i_2}{j_3\choose i_3} a_J$$
where $J\geq I$ if $j_\beta \geq i_\beta$ for all $\beta\neq 4$.  The formula is similar if $\alpha\neq 4$.  Looping over the indices of $b_I$ that must vanish and storing the corresponding linear constraints in a matrix $M$ allows us to compute the linear series of $D$ as a subspace of $\mathcal{O}(d)$ given by the kernel of $M$. 
If $|D|$ contains a nonzero element, factor this polynomial over $\mathbb{Q}$.  The class of the strict transform of each irreducible component may be easily computed by looking at its monomials and permuting marked points (as above).

\textbf{Computational Improvements}: Preprocess $D$ by finding an $S_7$-representative minimizing $d$, $m_5$, $m_{i5}$, and $m_{ij5}$ using Section \ref{symmetry}.  When computing $H^0(\overline{M}_{0,7}, \mathcal{O}(D))$, first perform computations over $\mathbb{Z}/2$.  This is considerably faster, and reliably detects whether $|D|$ is nonempty over characteristic 0.  To construct the matrix $M$, we recommend the following general procedure.  First, generate multi-indices $I = (i_0, \ldots , i_4)$, $|I| = d$, corresponding to monomials whose vanishing is \textit{not} imposed by the multiplicities $m_h, m_{hj}, m_{hjk}$ (with $h,j,k\neq 5$).  These form a basis for the domain of $M$.  The multiplicities $m_5, m_{h5}, m_{hj5}$ with $h,j \neq 4$ impose monomial conditions after the coordinate change $g_4 = f(x_4 - x_0, \ldots, x_4 -x_3, x_4)$.  We may obtain all conditions they impose by iterating through the monomials $x^I$ of $g_4$ whose vanishing is imposed by some $m_J$ with $5\in J$ and $4\notin J$, but not imposed by any multiplicity $m_J$ with $5\notin J$.  For each such $x^I$,  \textit{efficiently}\footnote{Fixing $i_0, i_1$, and $i_2$ in $I = (i_0, \ldots , i_4)$, the expression for each $b_I$ differs only by the values of ${j_3 \choose i_3 }$ for each $a_J$, and so may be computed from the same expression.} generate the expression for $b_I$ above and append it to $M$ as a row.
Since the binomial coefficients above are used many times, we recommend computing them in advance.  Similarly, the multiplicities $m_{45}, m_{h45}$, with $3\neq h$ impose monomial conditions after the coordinate change $g_3 = f(x_3 - x_0, \ldots, x_3 -x_2, x_3, x_3 - x_4)$, so we may likewise iterate though vanishing monomials $x^I$ of $g_3$; however, we need not store the expression for $b_I$ whose vanishing is implied by some multiplicity $m_J$ with $\{4,5\} \not\subseteq J$.  Lastly, the multiplicity $m_{345}$ imposes monomial conditions in coordinates $g_2 = f(x_2 -x_0, x_2 - x_1, x_2, x_2-x_3, x_2 -x_4)$, and we record the necessary constraints as before.

Constructing the aforementioned matrix $M$ is quite expensive.  There are many optimizations that reduce the number of $b_I$ necessary to consider, 
and which streamline their computation.  For instance, curve classes in Table \ref{boundaryCurves} can reveal certain $b_I$ are redundant row entries to $M$, without computing $b_I$.  
Through meticulous coding, we were able to test divisor classes with $d\leq 50$, and even higher.  

\subsection{Algorithm for Curves}\label{curveAlg}
%
%
%
%
%
%
Given a curve class 
$$c= dl -\sum_{0\leq i\leq 5} m_i e_i - \sum_{0\leq i < j \leq 5} m_{ij} e_{ij} -\sum_{0\leq i < j < k\leq 5} m_{ijk} e_{ijk}$$
with $m_i, m_{ij}, m_{ijk}\geq 0$, the following algorithm finds a representative map $f:\mathbb{P}^1\rightarrow \overline{M}_{0,7}$ with $f_*[\mathbb{P}^1] = c$ and computes the splitting type of $\mathcal{N}_f$.  If a divisor $D$ is also given, we find $f(\mathbb{P}^1)$ meeting $D^{sm}$.

\begin{rem}
This algorithm ran on a data set of $55,000$ curves.  The run time varies widely but generally increases with $d$ and the anticanonical degree of $c$.  When $d\leq 4$, the algorithm usually finishes within seconds, but when $d=6$ and $(-K_{\overline{M}_{0,7}}).c =3$, the algorithm often took minutes.  We were able to compute representatives for nef curves of class $c$ with $d\leq 6$ and  $(-K_{\overline{M}_{0,7}}).c \leq 3$.  Given a divisor $D$, we were able to compute representatives for negative curves on $D$ of class $c$ with $d\leq 8$.
\end{rem}

\begin{rem}{\textit{Instability.}} To find representative maps $f$, this algorithm makes frequent use of Singular's \cite{DGPS} algorithms MinAssGTZ, facstd, and Groebner basis calculations.  As the complexity of these algorithms fluctuates wildly with random choices we make, this portion of our algorithm takes variable time.  Computations may hang or finish in milliseconds for the same curve class with $d\geq 5$.  Running multiple computations in parallel helps to find quicker solutions and avoid hanging Groebner basis computations.

\end{rem}


\begin{rem} The algorithm is bottlenecked by our ability to find a representative map $\mathbb{P}^1\rightarrow \overline{M}_{0,7}$.  All other computations in the algorithm take insignificant time.  Methods in numeric algebraic geometry could allow for the computation of representatives for higher degree curves \cite{HAUENSTEIN2017499} \cite{doi:10.1080/10586458.2013.737640}.
\end{rem}

\textbf{Overview:} Let $\pi:\overline{M}_{0,7}\rightarrow\mathbb{P}^4$ denote the Kapranov map.  The space of irreducible curves numerically equivalent to $c$ may be identified with the space of irreducible curves of degree $d$ in $\mathbb{P}^4$ passing through $p_i$, $l_{ij}$, and $\Delta_{ijk}$ with multiplicities given by $m_i$, $m_{ij}$, and $m_{ijk}$.
\begin{figure}[ht]
\includegraphics[width=0.3\linewidth]{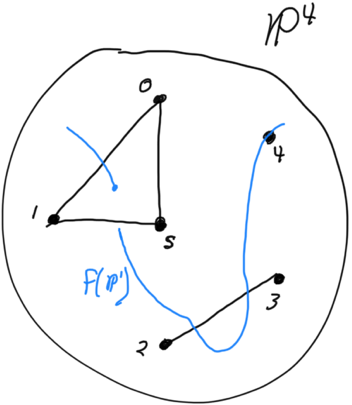}
\caption{An example of $f:\mathbb{P}^1\rightarrow \mathbb{P}^4$ meeting $\pi(E_4), \pi(E_{23})$, and $\pi(E_{015})$.}
\end{figure}
This algorithm tries to find a rational curve $f:\mathbb{P}^1\rightarrow \mathbb{P}^4$ whose strict transform $\tilde{f}:\mathbb{P}^1\rightarrow \overline{M}_{0,7}$ is of class $c$, and computes the splitting type of the virtual normal bundle $\mathcal{N}_{\tilde{f}}:=\text{coker}(\mathcal{T}_{\mathbb{P}^1}\rightarrow \tilde{f}^* \mathcal{T}_{\overline{M}_{0,7}})$.  By Lemma \ref{deformationLemma}, if $\mathcal{N}_{\tilde{f}}\cong \mathcal{O}_{\mathbb{P}^1}(-n)\oplus V$ with $V$ globally generated, and $D$ is an effective divisor such that $\text{deg}(\tilde{f}^*D)=-n$, verifying that $\pi(D)$ is nonsingular at a single point of $f(\mathbb{P}^1)$ proves that $c$ is a negative curve on (an integral component of) $D$.  If $n=1$, we may skip this check.  When $\mathcal{N}_{\tilde{f}}$ itself is globally generated, this algorithm shows that $c$ is nef.  Computations are done over a finite field, usually an extension of $\mathbb{Z}/101$. 


\textbf{Finding a Representative Map:} To find a degree $d$ map $f:\mathbb{P}^1\rightarrow\mathbb{P}^4$ intersecting $p_i$, $l_{ij}$, and $\Delta_{ijk}$ at $m_i$, $m_{ij}$, and $m_{ijk}$ distinct marked points (respectively), let 
$\mathbb{P}^1=\text{Proj}(k[s,t])$, and $\mathbb{P}^4=\text{Proj}(k[x_0,\dots , x_4])$.  
Such a map $f:\mathbb{P}^1\rightarrow \mathbb{P}^4$ is given by $[s:t]\rightarrow [f_0(s,t): \dots : f_4(s,t)]$, with each $f_i$ homogeneous of degree $d$ (sharing no common root).  Setting $s=1$, we may write each $f_i$ as 
$$f_i(t) =a_i \prod_{j=1}^d (t - r_{i,j})$$
The condition that $f$ meets the appropriate exceptional divisor $\pi (E_I)$ at $t=z$ can be expressed as
\begin{itemize}
\item if $5\not\in I$, then $f_i (z) = 0$ for all $i\not\in I$
\item if $5\in I$, then $f_i(z) - f_j (z) = 0$ for all $i,j\not\in I$
\end{itemize}
Let $z_{I,1},\dots z_{I,m_{I}}$ be the $m_{I}$ marked points along which $f(\mathbb{P}^1)$ meets $\pi (E_I)$.  Provided that we localize away from solutions where $a_i = 0$, the $f_i$ share a common root, or $z_{I,i}= z_{J,j}$ for some $(I,i)\neq (J,j)$, the above expressions give polynomial equations whose solutions correspond to $f:\mathbb{P}^1\rightarrow \mathbb{P}^4$ with the desired properties.  We describe efficient means of finding such a solution later under \textit{Computational Improvements}.

\textbf{Computing $\mathcal{N}_{\tilde{f}}$:} Given a map $f:\mathbb{P}^1\rightarrow \mathbb{P}^4$ as above, instead of computing $\mathcal{N}_{\tilde{f}}$, we compute the splitting type of a related bundle $V$, which is the extension of a globally generated vector bundle on $\mathbb{P}^1$ by $\mathcal{N}_{\tilde{f}}$.  This allows us to infer all necessary information about $\mathcal{N}_{\tilde{f}}$.  To find $V$, consider $\mathcal{N}_{\tilde{f}}$ as a subsheaf of $\mathcal{N}_{f}:=\text{coker}(\mathcal{T}_{\mathbb{P}^1}\rightarrow f^* \mathcal{T}_{\mathbb{P}^4})$.  Since $\mathcal{T}_{\mathbb{P}^1}$ is ample, the preimage of $\mathcal{N}_{\tilde{f}}\subset \mathcal{N}_{f}$ in $f^*\mathcal{T}_{\mathbb{P}^4}$, say $V'$, will have the same number (and degree) of negative factors in its splitting type as $\mathcal{N}_{\tilde{f}}$.  The same statement is true for the preimage of $V'$, denoted by $V$, in the middle term of the Euler sequence for $\mathbb{P}^4$:

$$0\rightarrow f^*\mathcal{O}_{\mathbb{P}^4}\rightarrow f^*(\oplus_5 \mathcal{O}_{\mathbb{P}^4}(1))\rightarrow f^* \mathcal{T}_{\mathbb{P}^4}\rightarrow 0.$$

Thus, we replace $\mathcal{N}_{\tilde{f}}$ with its preimage $V$ under the map $\oplus_5 \mathcal{O}_{\mathbb{P}^1}(d) \cong f^*(\oplus_5 \mathcal{O}_{\mathbb{P}^4}(1)) \rightarrow f^*\mathcal{T}_{\mathbb{P}^4} \rightarrow \mathcal{N}_f$.  Note that the natural map $\mathcal{T}_{\mathbb{P}^1}\rightarrow f^* \mathcal{T}_{\mathbb{P}^4}$ extends to a map between Euler sequences by considering each as a quotient of $\mathbb{A}^n\setminus{0}$ by $k^{\times}$.  The map  $\oplus_2 \mathcal{O}_{\mathbb{P}^1}(1)\rightarrow \oplus_5 \mathcal{O}_{\mathbb{P}^1}(d)$ is multiplication by the affine Jacobian.


To compute $V$, consider that by the identification of $\mathcal{N}_{\tilde{f}}\subset\mathcal{N}_{f}$, its preimage in $\oplus_5 \mathcal{O}_{\mathbb{P}^1}(d)$ is the intersection of subsheaves, corresponding to $z_{I,i}$, by whose quotient $\oplus_5 \mathcal{O}_{\mathbb{P}^1}(d)$ becomes a reduced torsion sheaf supported at $t=z_{I,i}$ in directions normal to both $f(\mathbb{P}^1)$ and $\pi (E_I)$.   As indicated, the tangent direction to $f(\mathbb{P}^1)$ at $t=z_{I,i}$ is the image of the natural map $\oplus_2 \mathcal{O}_{\mathbb{P}^1}(1)\rightarrow \oplus_5 \mathcal{O}_{\mathbb{P}^1}(d)$.  The tangent directions to $\pi (E_I)$ are also easy to identify as linear subspaces of $\mathbb{A}^5$, since the 5 copies of $\mathcal{O}_{\mathbb{P}^1}(d)\cong f^* \mathcal{O}_{\mathbb{P}^4}(1)$ correspond to 5 linear coordinate functions on $\mathbb{A}^5$.  On the affine patch $s=1$, $\oplus_5 \mathcal{O}_{\mathbb{P}^1}(d)\cong \oplus_5 k[t]=: M$, and the submodule corresponding to $z_{I,i}$ is spanned by $(t-z_{I,i})M$ and the indicated tangent vectors. Intersect these using Singular's algorithms to obtain the vector bundle $V$.  Frequently, we may even identify the splitting type of $\mathcal{N}_{\tilde{f}}$ from $V$: over characteristic $p$, if $p|\text{deg}(f)$, then $V$ is an extension of $\mathcal{O}\oplus \mathcal{O}(2)$ by $\mathcal{N}_{\tilde{f}}$; otherwise, it extends $\oplus_2 \mathcal{O}(1)$ by  $\mathcal{N}_{\tilde{f}}$.
\begin{rem}
If rational curves of class $c$ are free curves or negative curves on a nonboundary divisor effective over characteristic 0, then \cite[Proposition~2.8]{beheshti2020moduli} implies a general such curve is transverse to each boundary divisor at every point of intersection (see Remark \ref{strange behavior}).  Therefore, in these cases our transversality assumptions do not prevent the identification of a general $f:\mathbb{P}^1\rightarrow \mathbb{P}^4$ with strict transform of class $c$.  When $c$ is a negative curve on an effective divisor over characteristic $p$ that is not effective over characteristic 0, general representatives $\tilde{f}:\mathbb{P}^1\rightarrow \overline{M}_{0,7}$ of $c$ may not be transverse to each boundary divisor (see Remark \ref{strange behavior}).  In these cases, we may still compute $\mathcal{N}_{\tilde{f}}$.  To do this, we replace $(t-z_{I,i})M$ with $(t-z_{I,i})^{k_{I,i}}M$ and the tangent direction of $f(\mathbb{P}^1)$ with the $k_{I,i}^{th}$ Hasse derivative of $f$ at $t=z_{I,i}$, where $k_{I,i}$ is the multiplicity with which $f$ and $\pi(E_I)$ meet at $t=z_{I,i}$.
\end{rem}
\textbf{Verifying $\tilde{f}_*[\mathbb{P}^1]=c$}:  There are a few potential problems with the above approach.   What if $f(\mathbb{P}^1)$ is contained in some $\Delta_{ijk}$, intersects boundary divisors at additional points (counting multiplicity), or hits a boundary divisor along one of its boundary strata? 
Localizing away from solutions where an unexpectedly large number of $a_i$ or $r_{i,j}$ coincide (exceeding relations implied by $f(z_{I,i})\in \pi (E_I)$) prevents most of these issues.  To ensure the strict transform $\tilde{f}:\mathbb{P}^1\rightarrow \overline{M}_{0,7}$ represents $c$, we can simply check that various $\text{gcd}(f_i : i \in I)$ and $\text{gcd}(f_i - f_j : i,j\in I)$ have the expected degree.  This verifies that none of the above problems occur.


\textbf{Computational Improvements}: Preprocess $c$ by finding an $S_7$-representative minimizing $d$, $m_5$, $m_{i5}$, and $m_{ij5}$ using Section \ref{symmetry}.  When finding a representative of $c$, it is expensive to identify a solution to $f(z_{I,i})\in \pi(E_I)$ satisfying $a_i\neq 0$ for all $i$ and $r_{i,\alpha} = r_{j,\beta}$ iff $i\neq j$ and $r_{i,\alpha} = z_{I,h}$ for some $I\ni i,j$.  Localizing away from all bad solutions is impractical, as for more complicated $c$ this introduces far too many variables.  As a start, we may use $\text{Aut}(\mathbb{P}^1)$ to fix three roots as $0,1,\infty$.  
If given a divisor $D$, specify $f(t=\infty)\in D^{sm}\cap M_{0,7}\subset\mathbb{P}^4$ to be a random point; similarly, if we expect $c$ to be nef and $f$ to be free, we may fix $f(t=\infty) \in M_{0,7}\subset\mathbb{P}^4$ as a random point.  This fixes the variables $a_i$, which improves performance and avoids bad solutions.  Otherwise set $z_{I,i}=\infty$ for $I\ni 5$ and act later by $\text{Aut}(\mathbb{P}^1)$ to move all marked points into the affine open $s=1$.

To further reduce the ambient dimension of the problem, note that if $5\not\in I$, then each $z_{I,i}$ is a common root of $f_j$ for all $j\not\in I$.  As long as $c$ pairs nonnegatively with each boundary divisor on $\overline{M}_{0,7}$, no more than $d$ $z_{I,i}$ will be identified as roots of $f_j$ in this way, and so we may set distinct $r_{j,\alpha}=z_{I,\beta}$ for each $\beta$ and $I\not\ni 5$ (and each $j\not\in I$).  To account for $I\ni 5$, let $g_{ij}(t)= \frac{f_i - f_j}{\text{gcd}(f_i,f_j)}$ and $q(t)=\prod_{5\in I \not\ni i,j} \prod_{k=1}^{m_I} (t-z_{I,k})$.  If $g_{ij}=\alpha q + r(t)$ with $\text{deg}(r) < \text{deg}(q)$, then the condition that $f(z_{I,\beta})\in\pi(E_I)$ is equivalent to $r\equiv 0$ as a polynomial in $t$.  This yields polynomial conditions in $a_i, r_{i,j},$ and $z_{I,i}$ defining a subscheme $X\subset \text{Spec}(k[a_i, r_{i,j}, z_{I,i}])$.  Let $X_0\subset X$ be the component corresponding to nondegenerate solutions $f:\mathbb{P}^1\rightarrow\mathbb{P}^4$.  At this step (and others) in the algorithm, we may try processing $X$ with a factorizing Groebner basis algorithm to remove degenerate components.

If $c$ is a negative curve and $D$ is an integral divisor with $c.D=-n$ (set $n=3$ if no such divisor is known, or $n=0$ if we expect $c$ to be nef), then the expected dimension of $X_0$ is $e(X_0)=5d + (n-2) -3\sum m_i -2\sum m_{ij} -\sum m_{ijk}$.  Let $X'$ be the intersection of $X$ with $e(X_0)$ random coordinate hyperplanes of the form $z_{I,i}=\alpha$, $r_{i,j}=\alpha$, or $a_i=\alpha$ for distinct $1\neq\alpha\in k^{\times}$.  If $X'$ isn't 0 dimensional, cut with more such hyperplanes to obtain a 0-dimensional $X''$.  Likely, every point in $X''$ corresponds to a degenerate solution.  If so, localize $X'$ away from a few equations $z_{I,i}=z_{J,j}$ (or other equations defining degenerate components of solutions) until it contains none of these points, and recompute its dimension.  Repeating this process, or choosing other random hyperplanes to define $X'$, until $\text{dim }X'=0$ almost always yields nondegenerate solutions, as is expected.\footnote{Paul Larsen \cite{Larsen2012FultonsCF} shows that any curve class pairing nonnegatively with each boundary divisor is in $NE(\overline{M}_{0,7})$.  Moreover, the expected dimension of the Kontsevich space for each such rational curve class is always at least 1.}  Given a zero dimensional $X'$ containing nondegenerate points, base extend to the splitting field of one of its closed points, and take $f:\mathbb{P}^1\rightarrow \mathbb{P}^4$ to be one of these solutions.




\subsection{Symmetry on $\overline{M}_{0,7}$}\label{symmetry} 
Let $\overline{M}_{0,7}$ parameterize rational curves with markings labeled $\{0,\ldots ,6\}$.  $S_7$ acts naturally on $\overline{M}_{0,7}$ by permutation of the 7 marked points.  We wish to describe this action on $\overline{M}_{0,7}$ as a blow-up of $\mathbb{P}^4$.  Recall that this identification, due to Kapranov, is given by the map $\overline{M}_{0,7}\rightarrow \mathbb{P}^4$ corresponding to the complete linear system of a $\psi$-class $\psi_i$.  We let $i=6$, so that the exceptional divisors are $E_I = \delta_{I\cup\{6\}}$ with $1\leq |I| \leq 3$.  By fixing an automorphism of $\mathbb{P}^4$, we may assume $E_i$ is contracted to the point $p_i$ given at the beginning of Section \ref{M07alg}.  Under this identification, the $S_7$-action is well known and completely described by the action of $\text{Sym}(\{0,\ldots , 5\})$ and the transposition $(56)$.
%
\begin{lem}
With the above map $\overline{M}_{0,7}\rightarrow \mathbb{P}^4$, the natural action of $S_6 < S_7$ permuting the $6$ markings $\{0,\ldots ,5\}$ coincides with the action of $S_6$ on $\mathbb{P}^4$ permuting $p_i$ (and therefore the labelings of $E_i, E_{ij}, \text{ and } E_{ijk}$).  The transposition swapping markings $5$ and $6$ resolves the Cremona transform $[x_0:\ldots : x_4]\rightarrow [\frac{1}{x_0}:\ldots :\frac{1}{x_4}]$ on $\mathbb{P}^4$.  This acts on $\text{Pic}(\overline{M}_{0,7})$ by permutation of boundary divisors.
\end{lem}


\textbf{Identifying a Canonical $S_7$-Representative:} As divisor and curve classes of higher degree in $\mathbb{P}^4$ and with higher multiplicities along $p_5$, $l_{i5}$, and $\Delta_{ij5}$ are more difficult to compute, $S_7$ symmetry was used to choose a representative minimizing the $\mathbb{P}^4$-degree $d$ and the multiplicites $m_5, m_{i5}, m_{ij5}$.  
The stabilizer of each boundary divisor is of the form $S_k \times S_{7-k} < S_7$, while the stabilizer of the $\mathbb{P}^4$-degree is the stabilizer of marking $6$, $S_1\times S_6 < S_7$.  
If $S_k \times S_{7-k}$ is $\text{Sym}\{i_1,\ldots ,i_k\} \times \text{Sym}\{j_1, \ldots j_{7-k}\}$, then we may use the ${7 \choose k}$ transpositions $\tau$ that swap at most $\text{min}(k,7-k)$ $i_h$ with some $j_h$ to traverse cosets of $S_k \times S_{7-k}\subset S_7$.  Acting on a divisor or curve class by each transversal $\tau$, we find a collection of $S_7$-representatives minimizing (degree or) multiplicity along a chosen boundary divisor.  This may be extended to fixing multiplicities along many boundary divisors at once.  
\begin{lem}\label{symmetry lem}
Let $G$ be the intersection of stabilizers of markings $i$ or boundary divisors $\delta_I$ on $\overline{M}_{0,7}$.  $G= S_{k_1} \times \ldots \times S_{k_n}< S_7$ is the subgroup stabilizing partitions $I\cup I^c = \{0, \ldots , 6\}$ corresponding to markings $\{i\} = I$ or boundary divisors $\delta_I$.  The intersection $G'$ of $G$ with any additional stabilizer $\text{Sym}(J) \times \text{Sym}(J^c)$ is the direct product of all $S_{k_i'}\times S_{k_i''}=S_{k_i}\cap (\text{Sym}(J) \times \text{Sym}(J^c)) < S_7$.  Cosets of $G'<G$ are traversed by $\prod_{i=1}^n {k_i \choose k_i'}$ transpositions $\tau$ given by the products of transversals of each $S_{k_i'}\times S_{k_i''} < S_{k_i}$. 
\end{lem}
\noindent Acting on a divisor or curve class $D$ with each transversal $\tau$ in Lemma \ref{symmetry lem} fixes ($\mathbb{P}^4$-degree and) multiplicities $m_{I}$ along boundary divisors stabilized by $G$, while identifying all values of $m_J$ in the $G$-orbit of $D$.  
If $O$ is the collection of $S_7/G$-representatives of $D$ minimizing $\mathbb{P}^4$-degree and multiplicities $m_I$ stabilized by $G$, then we may obtain a corresponding collection $O'$ of $S_7/G'$-representatives of $D$, minimizing $m_J$ as well, by acting on each member of $O$ with the above transversals $\tau$.  
We use this to efficiently identify a canonical representative for each curve and divisor class in this paper, ensuring that no two classes appearing in tables in Section \ref{Tables} are $S_7$-equivalent.  

\section{Proof of Theorem \ref{contractionThm} and Progress towards Conjecture \ref{losev_manin_conj}}\label{Section8}
Algorithm \ref{sectionAlgExtDiv} can be used to show $\text{Eff}(X)$ is polyhedral and generated by known extreme divisors.  Recall that $\mathfrak{E}\subset N^1(X)$ is the cone generated by known extreme effective divisors on $X$, while $\mathfrak{M}\subset N^1(X)$ is a cone bounded by curve classes (see Section \ref{Notation}).  If linear programming generates no new guesses for extreme divisors on $X$, this shows $\mathfrak{M}\subseteq \mathfrak{E}$, which implies $\mathfrak{E}=\text{Eff}(X)$.  We apply this to $X= \overline{M}_{0,\mathcal{A}}$ for $\mathcal{A}=(\frac{1}{3},\frac{1}{3},\frac{1}{3},\frac{1}{3},\frac{1}{3},\frac{1}{3},1)$ and $X= \text{Bl}_e \overline{LM}_7$.  Both spaces may be realized as intermediate blow-ups factorizing $\overline{M}_{0,7}\rightarrow \mathbb{P}^4$.  We prove $\text{Eff}(\overline{M}_{0,\mathcal{A}})$ is polyhedral and generated by divisors in Table \ref{cone22}.  We conjecture that Table \ref{cone27} lists all extreme divisors on $\text{Bl}_e \overline{LM}_7$.

\subsection{$\overline{M}_{0,\mathcal{A}}$}  The Hassett space $\overline{M}_{0,\mathcal{A}}$ for $\mathcal{A}=(\frac{1}{3},\frac{1}{3},\frac{1}{3},\frac{1}{3},\frac{1}{3},\frac{1}{3},1)$ is the iterated blow-up of $\mathbb{P}^4$ along 6 linearly general points, $p_i$, followed by the strict transforms of the 15 lines they span.  As with $\overline{M}_{0,7}$, we let $H$ denote the pullback of the hyperplane class, and $E_i, E_{ij}$ denote the exceptional divisors lying over the point $p_i$ and line $\overline{p_i p_j}$, for $0\leq i,j \leq 5$.  Via the contraction $\phi : \overline{M}_{0,7}\rightarrow \overline{M}_{0,\mathcal{A}}$, we obtain a natural map $\phi_*: N^1(\overline{M}_{0,7})\rightarrow N^1(\overline{M}_{0,\mathcal{A}})$ which takes $H$ to $H$, $E_i$ to $E_i$, $E_{ij}$ to $E_{ij}$, and sends $E_{ijk}$ to 0.  The map $\phi^* : N_1(\overline{M}_{0,\mathcal{A}})\rightarrow N_1(\overline{M}_{0,7})$ injects $N_1(\overline{M}_{0,\mathcal{A}})$ as the subspace spanned by $l, e_i, \text{ and } e_{ij}$.  By permuting the 6 points $p_i$, we obtain a natural action of $S_6$ on $\overline{M}_{0,\mathcal{A}}$.  The divisors in Table \ref{cone22} are projections via $\phi_*$ of divisors $S_7$-equivalent to those in Tables \ref{ExtDivs0}, \ref{ExtDivs1}, and \ref{boundaryCurves}.  They generate $\text{Eff}(\overline{M}_{0,\mathcal{A}})$.

\begin{longtable} [c] { || p{29em} | p{5em} || }
\hline
\multicolumn{2} { | c | }{\textbf{Table \ref{cone22}:} Extreme Divisor Classes on $\overline{M}_{0,\mathcal{A}}$ for $\mathcal{A}=(\frac{1}{3},\frac{1}{3},\frac{1}{3},\frac{1}{3},\frac{1}{3},\frac{1}{3},1)$}\label{cone22}\\
\hline
Divisor & Size of $S_6$ Orbit \\
\hline
\hline
\endfirsthead
\hline
\multicolumn{2}{ | c | }{\textbf{Table \ref{cone22}:} Extreme Divisor Classes on $\overline{M}_{0,\mathcal{A}}$ for $\mathcal{A}=(\frac{1}{3},\frac{1}{3},\frac{1}{3},\frac{1}{3},\frac{1}{3},\frac{1}{3},1)$}\\
\hline
Divisor & Size of $S_6$ Orbit \\
\hline
\hline
\endhead
\hline
\endfoot
\hline
\hline
\multicolumn{2}{ | c | }{End of Extreme Divisor Classes on $\overline{M}_{0,\mathcal{A}}$ for $\mathcal{A}=(\frac{1}{3},\frac{1}{3},\frac{1}{3},\frac{1}{3},\frac{1}{3},\frac{1}{3},1)$}\\
\hline
\endlastfoot
\hline  $E_{5}$  &  6  \\ 
\hline  $E_{45}$  &  15  \\ 
\hline  $H-E_{0}-E_{1}-E_{2}-E_{3}-E_{01}-E_{02}-E_{03}-E_{12}-E_{13}-E_{23}$  &  15  \\ 
\hline  $2H-2E_{0}-E_{1}-E_{2}-E_{3}-E_{4}-E_{5}-E_{01}-E_{02}-E_{03}-E_{04}-E_{05}-E_{12}-E_{13}-E_{24}-E_{34}$  &  90  \\ 
\hline  $2H-E_{0}-E_{1}-E_{2}-E_{3}-E_{4}-E_{5}-E_{01}-E_{02}-E_{03}-E_{12}-E_{14}-E_{25}-E_{34}-E_{35}-E_{45}$  &  60  \\ 
\hline  $4H-3E_{0}-3E_{1}-3E_{2}-2E_{3}-2E_{4}-2E_{5}-2E_{01}-2E_{02}-2E_{03}-2E_{04}-E_{05}-2E_{12}-2E_{13}-2E_{14}-E_{15}-2E_{23}-2E_{24}-E_{25}$  &  60  \\ 
\hline  $4H-2E_{0}-2E_{1}-2E_{2}-2E_{3}-2E_{4}-2E_{5}-2E_{01}-2E_{02}-2E_{03}-2E_{14}-2E_{15}-2E_{24}-2E_{25}-2E_{34}-2E_{35}$  &  10  \\ 
\hline  $4H-3E_{0}-2E_{1}-2E_{2}-2E_{3}-2E_{4}-2E_{5}-2E_{01}-2E_{02}-2E_{03}-2E_{04}-E_{05}-2E_{12}-2E_{13}-2E_{25}-2E_{35}-E_{45}$  &  360  \\ 
\hline  $6H-4E_{0}-3E_{1}-3E_{2}-3E_{3}-3E_{4}-3E_{5}-3E_{01}-3E_{02}-E_{03}-E_{04}-E_{05}-3E_{13}-3E_{14}-2E_{23}-2E_{24}-2E_{25}-3E_{35}-3E_{45}$  &  360  \\ 
\end{longtable}
\begin{proof}[Proof of Theorem \ref{contractionThm}]
This proof is computer-based.  Generate the $S_7$-orbits of divisor classes listed in Tables \ref{ExtDivs0}, \ref{ExtDivs1}, and \ref{boundaryCurves}.  Their pushforward under the reduction map $\phi: \overline{M}_{0,7}\rightarrow \overline{M}_{0,\mathcal{A}}$ generates a cone $\mathfrak{E}\subset N^1(\overline{M}_{0,\mathcal{A}})$.  After computing a dual description of $\mathfrak{E}$, lift each facet $f$ to an element of $\phi^*(f)\in N_1(\overline{M}_{0,7})$.  Linear programming shows each $\phi^*(f)$ is a nonnegative linear combination of curve classes in Table \ref{boundaryCurves}. 
Since such curve classes only pair negatively with the boundary divisors on $\overline{M}_{0,7}$ over all characteristics, this proves $\phi^*(f)$ is nef over any characteristic.
\end{proof}

\subsection{$\text{Bl}_e \overline{LM}_7$}
The Losev-Manin space $\overline{LM}_7$ is the iterated blow-up of $\mathbb{P}^4$ at 5 points, $p_1, \ldots , p_5$, followed by the strict transforms of the 10 lines $l_{ij} = \overline{p_i p_j}$, and lastly the strict transforms of the 10 planes $\Delta_{ijk} = \Delta p_i p_j p_k$.  The complement of these blown-up strata is an embedded torus, with identity $e=p_0$.  We blow up $p_0$ to obtain $\text{Bl}_e \overline{LM}_7$ with exceptional divisors $E_0$ lying over $e=p_0$, $E_i$ lying over $p_i$, $E_{ij}$ lying over $l_{ij}$, and $E_{ijk}$ lying over $\Delta_{ijk}$, for $1\leq i,j,k \leq 5$.

We may also realize $\text{Bl}_e \overline{LM}_7$ as the Hassett Space $\overline{M}_{0,\mathcal{A}}$ for $\mathcal{A} = (1,\frac{1}{4}, \frac{1}{4}, \frac{1}{4}, \frac{1}{4}, \frac{1}{4}, 1)$.  This yields a natural contraction map $\phi : \overline{M}_{0,7}\rightarrow \text{Bl}_e \overline{LM}_7$.  The induced map $\phi_* : N^1 (\overline{M}_{0,7}) \rightarrow N^1(\text{Bl}_e \overline{LM}_7)$ is the quotient map sending $E_{0i}$ and $E_{0ij}$ to 0.  The divisor $E_0$ parameterizes nodal curves wherein one component contains the first and last marking (the heavy markings).  By permuting marked points of equal weight, we obtain a natural action of $S_2 \times S_5$ on $\text{Bl}_e \overline{LM}_7$, which agrees with the action of $S_2 \times S_5 < S_7$ on $\overline{M}_{0,7}$.

As before, we generate $S_7$-orbits of divisor classes in Tables \ref{ExtDivs0}, \ref{ExtDivs1}, and \ref{boundaryCurves}, and split them into $S_2\times S_5$ sub-orbits.  After taking the pushforward of all such divisors to $N^1(\text{Bl}_e \overline{LM}_7)$ by $\phi_*$, we study the cone $\mathfrak{E}$ they generate.  We conjecture $\mathfrak{E} = \text{Eff}(\text{Bl}_e \overline{LM}_7)$.  
As $\mathfrak{E}$ is $S_2 \times S_5$ invariant, we may determine which $S_2 \times S_5$-orbits consist of extreme rays of $\mathfrak{E}$ by testing one representative from each orbit for redundancy (as a generator of $\mathfrak{E}$) using linear programming.  The extreme orbits we find appear below.




\begin{longtable} [c] { || p{29em} | p{5em} || }
\hline
\multicolumn{2} { | c | }{\textbf{Table \ref{cone27}:} Extreme Divisor Classes on $\text{Bl}_e(\overline{LM}_7)$}\label{cone27}\\
\hline
Divisor & Size of $S_2 \times S_5$ Orbit \\
\hline
\hline
\endfirsthead
\hline
\multicolumn{2}{ | c | }{\textbf{Table \ref{cone27}:} Extreme Divisor Classes on $\text{Bl}_e(\overline{LM}_7)$}\\
\hline
Divisor & Size of $S_2 \times S_5$ Orbit \\
\hline
\hline
\endhead
\hline
\endfoot
\hline
\hline
\multicolumn{2}{ | c | }{End of Extreme Divisor Classes on $\text{Bl}_e(\overline{LM}_7)$}\\
\hline
\endlastfoot
\hline  $E_{345}$  &  20  \\ 
\hline  $E_{0}$  &  1  \\ 
\hline  $E_{5}$  &  10  \\ 
\hline  $H-E_{0}-E_{1}-E_{2}-E_{3}-E_{12}-E_{13}-E_{23}-E_{123}$  &  10  \\ 
\hline  $2H-2E_{0}-E_{1}-E_{2}-E_{3}-E_{4}-E_{5}-E_{12}-E_{13}-E_{24}-E_{34}$  &  30  \\ 
\hline  $2H-E_{0}-2E_{1}-E_{2}-E_{3}-E_{4}-E_{5}-E_{12}-E_{13}-E_{14}-E_{15}-E_{23}-E_{24}-E_{35}-E_{45}-E_{123}-E_{124}-E_{135}-E_{145}$  &  15  \\ 
\hline  $3H-E_{0}-2E_{1}-2E_{2}-2E_{3}-2E_{4}-E_{5}-2E_{12}-2E_{13}-2E_{14}-E_{23}-E_{24}-E_{25}-E_{34}-E_{35}-E_{45}-E_{123}-E_{124}-E_{134}-E_{235}-E_{245}-E_{345}$  &  20  \\ 
\hline  $3H-2E_{0}-2E_{1}-2E_{2}-2E_{3}-E_{4}-E_{5}-E_{12}-E_{13}-E_{14}-E_{15}-E_{23}-E_{24}-E_{25}-E_{34}-E_{35}-E_{124}-E_{125}-E_{134}-E_{135}-E_{234}-E_{235}$  &  20  \\ 
\hline  $10H-7E_{0}-6E_{1}-6E_{2}-6E_{3}-6E_{4}-6E_{5}-4E_{12}-4E_{13}-3E_{14}-3E_{15}-3E_{23}-4E_{24}-3E_{25}-3E_{34}-4E_{35}-4E_{45}-E_{123}-E_{124}-2E_{125}-2E_{134}-E_{135}-2E_{145}-2E_{234}-2E_{235}-E_{245}-E_{345}$  &  12  \\ 
\end{longtable}

\begin{lem}
The divisors in Table \ref{cone27} are pushforwards of divisors $S_7$-equivalent to $E_{01}$, $E_0$, $E_0$, $E_0$, $\boldsymbol{D_1}$, $\boldsymbol{D_1}$, $\boldsymbol{D_3}$, $\boldsymbol{D_3}$, and $\boldsymbol{D_{28}}$.  They are extreme in $\overline{\text{Eff}}(\text{Bl}_e \overline{LM}_7)$.
\end{lem}
\begin{proof}
Other divisors $D'$ in the $S_7$-orbit of $\boldsymbol{D_i} = (dH -\sum m_j E_j - \ldots)$ have $\mathbb{P}^4$ degree $d' = 4d - \sum_{j\neq k} m_j$ for some $k$, and multiplicities $m_k' = m_k$ and $m_j' = 3d - \sum_{h \neq j,k} m_h$.  This allows us to determine which divisors pushforward to the divisors contained in Table \ref{cone27}.  We can see that each pushforward above is extreme in $\overline{\text{Eff}}(\text{Bl}_e \overline{LM}_7)$ by studying negative curves on that divisor.  Aside from the two pushforwards of $\boldsymbol{D_1}$, the negative curves on each $\boldsymbol{D_i}$ are pullbacks via $\phi^*$ of curves on $\text{Bl}_e \overline{LM}_7$.  The extremality of both pushforwards of $\boldsymbol{D_1}$ follows easily.
\end{proof}

As in the proof of Theorem \ref{contractionThm}, we may attempt to prove $\mathfrak{E} = \text{Eff}(\text{Bl}_e \overline{LM}_7)$ by generating a dual description for $\mathfrak{E}$ and lifting each facet $f$ to $\phi^*(f) \in N_1(\overline{M}_{0,7})$.  If each $\phi^*(f)$ were a positive linear combination of curves in Tables \ref{ExtDivs0}, \ref{ExtDivs1}, \ref{nefCurves}, and \ref{boundaryCurves}, then it would follow that $\mathfrak{E} = \text{Eff}(\text{Bl}_e \overline{LM}_7)$ over characteristic 0 and all but finitely many prime characteristics.  However, linear programming shows this does not hold.  
Instead, Table \ref{LMcurves} lists 154 facets of $\mathfrak{E}$ that may or may not be nef.  
Conjecture \ref{losev_manin_conj} holds if and only if each curve class in Table \ref{LMcurves} is nef.

\begin{rem}
Demonstrating that curve classes $c$ in Table \ref{LMcurves} are nef may be difficult to do directly by finding a representative map, especially when $c.(-K_{\text{Bl}_e \overline{LM}_7})$ is large.  Instead, it is possible that $c=c_1 + c_2$, where $c_1$ is a negative curve on some divisor in Table \ref{ExtDivs0}, \ref{ExtDivs1}, or \ref{boundaryCurves}, and $c_2$ is another such negative curve, or a nef curve on $\overline{M}_{0,7}$.  Such expressions have already been observed for many other facets of $\mathfrak{E}$, and appear more likely to exist when $c.(-K_{\text{Bl}_e \overline{LM}_7})$ is larger. Indeed, when $c.(-K_{\text{Bl}_e \overline{LM}_7}) \geq 3$, any free curve $f:\mathbb{P}^1\rightarrow \text{Bl}_e \overline{LM}_7$ of class $c$ deforms to a stable map $g:C\rightarrow \text{Bl}_e \overline{LM}_7$ from a nodal curve with at least one free component.  This follows from viewing $f$ as the strict transform of $\mathbb{P}^1\xrightarrow{f} \text{Bl}_e \overline{LM}_7 \xrightarrow{\pi} \overline{LM}_7$.  All deformations of $f$ meet the exceptional locus $E_0$ of $\pi$, which is contracted to a point.  By applying Mori's Bend and Break to a family of deformations of $f$ passing through a general point of $\text{Bl}_e \overline{LM}_7$, we obtained the desired map $g:C\rightarrow \text{Bl}_e \overline{LM}_7$.
\end{rem}
\setlength\LTleft{-6.75em}
\begin{longtable}  { || p{41.4em} | p{3.6em} || }
\hline
\multicolumn{2} { | c | }{\textbf{Table \ref{LMcurves}:} Curve Classes on $\text{Bl}_e\overline{LM}_7$ Related to Conjecture \ref{losev_manin_conj}}\label{LMcurves}\\
\hline
Curve Class & Anti-canonical Degree \\ 
\hline
\hline
\endfirsthead
\hline
\multicolumn{2}{ | c | }{(Continuation) \textbf{Table \ref{LMcurves}:} Curve Classes on $\text{Bl}_e\overline{LM}_7$ Related to Conjecture \ref{losev_manin_conj}}\\
\hline
Curve Class & Anti-canonical Degree \\ 
\hline
\hline
\endhead
\hline
\endfoot
\hline
\hline
\multicolumn{2}{ | c | }{End of Curve Classes on $\text{Bl}_e \overline{LM}_7$ Related to Conjecture \ref{losev_manin_conj}}\\
\hline
\endlastfoot
\hline  $7l-4e_{0}-e_{25}-e_{34}-e_{45}-e_{123}-2e_{124}-2e_{125}-e_{134}-3e_{135}-2e_{234}-e_{245}-e_{345}$  &  $4$  \\ 
\hline  $7l-4e_{0}-e_{15}-e_{25}-e_{34}-2e_{123}-2e_{124}-e_{125}-e_{134}-2e_{145}-e_{234}-2e_{235}-2e_{345}$  &  $4$  \\ 
\hline  $8l-4e_{0}-e_{15}-e_{25}-e_{34}-3e_{45}-2e_{123}-3e_{124}-3e_{135}-2e_{234}-e_{235}$  &  $5$  \\ 
\hline  $8l-4e_{0}-e_{15}-e_{24}-e_{34}-3e_{35}-2e_{123}-e_{124}-2e_{125}-3e_{134}-3e_{245}$  &  $5$  \\ 
\hline  $9l-5e_{0}-e_{25}-e_{34}-3e_{45}-2e_{123}-3e_{124}-e_{125}-4e_{135}-e_{145}-3e_{234}-e_{235}$  &  $5$  \\ 
\hline  $9l-5e_{0}-e_{15}-e_{25}-2e_{34}-e_{45}-2e_{123}-2e_{124}-2e_{125}-e_{134}-3e_{135}-2e_{234}-2e_{245}-e_{345}$  &  $5$  \\ 
\hline  $9l-5e_{0}-e_{15}-e_{25}-3e_{34}-2e_{123}-3e_{124}-e_{125}-e_{135}-3e_{145}-e_{234}-3e_{235}-e_{345}$  &  $5$  \\ 
\hline  $9l-5e_{0}-e_{15}-2e_{25}-2e_{34}-2e_{123}-3e_{124}-e_{125}-e_{134}-3e_{135}-e_{234}-2e_{245}-2e_{345}$  &  $5$  \\
\end{longtable}
\normalsize


\section{Tables of Divisors and Curves on $\overline{M}_{0,7}$}\label{Tables}

This section contains all curves and divisors found using our algorithms.  No two curve or divisor classes appearing in tables in this paper are equivalent under the natural $S_7$ action on $\overline{M}_{0,7}$.

We identify 31 new extreme divisors and 58 negative curve classes sweeping out these divisors in Tables \ref{ExtDivs0} and \ref{ExtDivs1}.  We also verify that the divisor $\boldsymbol{D_5}$, first found by \cite{doran2016simplicial}, is an extreme ray of $\overline{\text{Eff}}(\overline{M}_{0,7})$.  Before our work, there were only 5 $S_7$-distinct divisors known to be extreme rays of $\overline{\text{Eff}}(\overline{M}_{0,7})$: the two orbits of boundary divisors containing $E_0$ and $E_{01}$, the hypertree divisors $\boldsymbol{D_1}$ and $\boldsymbol{D_2}$, and Opie's divisor $\boldsymbol{D_3}$.  All divisors appearing in Tables \ref{ExtDivs0} and \ref{ExtDivs1} are extreme over characteristic 0 and characteristic $p$ for all but finitely many primes $p$.  

Over characteristic 2, we identify two further extreme divisors that are not effective over characteristic 0.  These divisors are listed in Table \ref{ExtDivs2} as $\boldsymbol{D_{37}}$ and $\boldsymbol{D_{38}}$.  We use $\boldsymbol{D_{38}}$ to prove Theorem \ref{charThm}.  The last divisor in Table \ref{ExtDivs0}, $\boldsymbol{D_{31}} = \boldsymbol{D_{37}} + E_{012}$, is extreme in $\text{Eff}(\overline{M}_{0,7})$ over characteristic 0 but is reducible in characteristic 2.  

We also include a short list of rigid, integral divisors on $\overline{M}_{0,7}$ that are not extreme (Table \ref{notExt}), curve classes that sweep out unknown divisors over characteristic 0 (Table \ref{unpairedNeg}), and low degree nef curves which cannot be expressed as positive linear combinations of known negative curves (Table \ref{nefCurves}).  The proof of Theorem \ref{nefCurvesThm} appears before Table \ref{nefCurves}.

\textbf{On Proper Transform of Chen-Coskun Divisors:} As both $\boldsymbol{D_1}$ and $\boldsymbol{D_3}$ are proper transforms $\Lambda_a$ of Chen-Coskun divisors $D_a$ \cite{opie2016extremal} \cite{chen2013extremal}, it is natural to wonder whether other $\boldsymbol{D_i}$ are as well.  By expressing each $\Lambda_a$ as a pullback of effective divisors on either $\overline{M}_{0,6}$ or $\text{Bl}_e\overline{LM}_{7}$ with explicit coefficents, we show $\boldsymbol{D_i} \neq \Lambda_a$ for all other $i$.  To fix notation, let $\varphi: \overline{M}_{0,n+2}\rightarrow \overline{\mathcal{M}}_{1,n}$ be a clutching morphism gluing markings $\{0,n+1\} \subset \{0,1,\ldots n+1\}$. Let $f:\overline{M}_{0,n+2}\rightarrow \mathbb{P}^{n-1}$ be the morphism determined by the complete linear system of the $\psi_{n+1}$.  As usual, we may realize $f$ as the composition of several blow-ups, with exceptional divisors $E_I$ for $I\subset\{0,\ldots n\}$, $1\leq |I| \leq n-2$, corresponding to boundary divisors $\delta_{I\cup \{n+1\}}$.  This description of $f$ as an iterated blow-up factors as $f=g \circ \pi$, where $\pi: \overline{M}_{0,n+2}\rightarrow \text{Bl}_{e}\overline{LM}_{n+2}$ has exceptional divisors $E_I$ for $0\in I$, $|I|\geq 2$, and $g:\text{Bl}_{e}\overline{LM}_{n+2} \rightarrow \mathbb{P}^{n-1}$.  
For ease of notation, let $N=\{1,\ldots , n-1\}$, and for $I\subseteq N$ let $S(I) = \min(\sum_{0 \leq a_i, i\in I} |a_i|, \sum_{0 \geq a_i, i\in I} |a_i|)$.

\begin{lem}
Let 
$a=(a_1 , \ldots , a_n)$ satisfy $\sum a_i = 0$ and $\Lambda_a$ be the proper transform of $D_a$ under $\varphi$.
\begin{enumerate}
    \item if $a_i = 0$ for some $i$, then $\Lambda_a$ is the pullback of an effective divisor class on $\overline{M}_{0,n+1}$.
    \item if $a_i \neq 0$ for all $i$, then $\Lambda_a$ is the pullback via $\pi$ of an effective divisor class on $\text{Bl}_e\overline{LM}_{n+2}$ given by 
    $$ \Lambda_a = dH -E_0 - \sum_{i=1}^n (d-|a_i|)E_i -\sum_{I \subsetneq N, |I|\geq 2}(S(I)+d - \sum_{i\in I} |a_i|)E_I  -\sum_{\emptyset \neq I \subsetneq N} S(N\setminus I)E_{I\cup \{n\}}$$
    with $d=\frac{1}{2}\sum_{i=1}^n|a_i|$.
\end{enumerate}
\end{lem}
\begin{proof}
Corollary 4.4 of \cite{opie2016extremal} proves our first claim, so suppose $a_i \neq 0$ for all $i$.  Corollary 4.5 of \cite{opie2016extremal} then provides a formula for $\Lambda_a$ under a different Kapranov basis of $N^1(\overline{M}_{0,n+2})$ given by $\psi_n$ and $E_I^n = \delta_{I\cup \{n\}}$ for $I\subset  \{0,1,\ldots ,n-1, n+1\}$ with $|I|\leq n-2$:
$$\Lambda_a = d'\psi_n -S(N)(E_0^n +E_{n+1}^n) -\sum_{\emptyset \neq I \subsetneq N} [ (d' - \sum_{i\in I} |a_i|) E_I^n +S(N\setminus I)(E_{I\cup \{0\}}^n +E_{I\cup \{n+1\}}^n)]$$
where $d' = (\sum_{i<n}|a_i|)-1$.  To express $\Lambda_a$ in terms of $H=\psi_{n+1}$ and $E_I$, we use the relations $\psi_n = (n-1)\psi_{n+1} - \sum_{\emptyset \neq I\subsetneq N\cup\{0\}}(n-1-|I|)E_I$, $E_{I\cup \{n+1\}}^n = E_{I\cup \{n\}}$, and $E_{I}^n = \delta_{I\cup \{n\}}$ for $I\subset N\cup\{0\}$.  This last relation gives $E_I^n = E_{N\cup \{0\}\setminus I}$ for $|I|>1$, and $E_i^n = H - \sum_{J\subset N\cup \{0\} \setminus \{i\}} E_J$.  After noting that $S(N) = \frac{d'+1-|a_n|}{2}$, our claim follows from simple algebra.
\end{proof}

\begin{cor}
$\boldsymbol{D_i}$ is the proper transform of a Chen-Coskun divisor iff $i=1,3$.  Only finitely many proper transforms $\Lambda_a$ may be extreme in $\overline{\text{Eff}}(\overline{M}_{0,7})$.
\end{cor}

\textbf{Description of Tables \ref{ExtDivs0} and \ref{ExtDivs1}:} Tables \ref{ExtDivs0} and \ref{ExtDivs1} contain all known extreme divisors on $\overline{M}_{0,7}$ over characteristic 0 (aside from the boundary divisors).  Listed alongside each divisor $\boldsymbol{D_i}$ is a polynomial $f_i$ such that the strict transform of $V(f_i)$ under the blow-up $\overline{M}_{0,7}\rightarrow \mathbb{P}^4$ is the unique member of $|\boldsymbol{D_i}|$.  Each $S_7$-orbit representative $\boldsymbol{D_i}$ was chosen to minimize the degree of $f_i$ and the number of monomials in its expression.  The curve classes listed are negative curve classes for the corresponding divisor, and the normal bundles we computed for specific representatives of each curve class are written in the rightmost column.  

\begin{rem}{\textit{A Polynomial Equation on Markings.}}  The polynomials $f_i$ below explicitly describe the relation satisfied by markings on an irreducible curve $[(C,p_0, \ldots , p_6)]\in M_{0,7}$ contained in the corresponding divisor $\boldsymbol{D_i}$.  When $H = \psi_6$ and $C\cong \text{Proj}(k[s,t])$, we may fix $p_5 = V(t) = [0,1]$ and $p_6=V(s) = [1,0]$ using automorphisms of $C$.  When $p_j = [x_j,1]$ for $0\leq j \leq 4$, $f_i = 0$ is precisely the condition describing $[(C,p_0, \ldots , p_6)] \in \boldsymbol{D_i}$.  This may be derived from Kapranov's construction by looking at the universal family $\overline{M}_{0,8}\rightarrow \overline{M}_{0,7}$.

\end{rem}

\setlength\LTleft{-9.25em}



\newpage
\textbf{Description of Table \ref{ExtDivs2}:} Table \ref{ExtDivs2} lists extreme divisors on $\overline{M}_{0,7}$ over characteristic 2 which are not effective over characteristic 0.  Unlike the extreme divisor found in \cite{doran2016simplicial}, whose second multiple is effective over characteristic 0, by Theorem \ref{charThm2} these divisors lie outside $\text{Eff}(\overline{M}_{0,7})$ over characteristic 0.  
\begin{longtable} { | p{12em} || p{23em} | p{9em} | p{6em} | }
\hline
\multicolumn{4} { | c | }{\textbf{Table \ref{ExtDivs2}:} Extreme Divisors and Negative Curves on $\overline{M}_{0,7}$ in Characteristic 2}\label{ExtDivs2}\\
\hline
\textit{(Orbit Size)} Divisor & Polynomial & Class of $f_*[\mathbb{P}^1]$ & $\mathcal{N}_{f}$ \\
\hline
\hline
\endfirsthead
\hline
\multicolumn{4}{ | c | }{(Continuation) \textbf{Table \ref{ExtDivs2}:} Extreme Divisors and Negative Curves on $\overline{M}_{0,7}$ in Characteristic 2}\\
\hline
\textit{(Orbit Size)} Divisor & Polynomial & Class of $f_*[\mathbb{P}^1]$ & $\mathcal{N}_{f}$ \\
\hline
\hline
\endhead
\hline
\endfoot
\hline
\hline
\multicolumn{4}{ | c | }{End of Extreme Divisors and Negative Curves on $\overline{M}_{0,7}$ in Characteristic 2}\\
\hline
\endlastfoot

\hline
{$ \textit{(84) } \boldsymbol{D_{37}} = 9H-5E_{0}-5E_{1}-5E_{2}-5E_{3}-5E_{4}-5E_{5}-3E_{01}-3E_{02}-3E_{03}-3E_{04}-3E_{05}-3E_{12}-3E_{13}-3E_{14}-3E_{15}-3E_{23}-3E_{24}-3E_{25}-3E_{34}-3E_{35}-3E_{45}-2E_{012}-2E_{013}-E_{014}-E_{015}-E_{023}-2E_{024}-E_{025}-E_{034}-2E_{035}-2E_{045}-E_{123}-E_{124}-2E_{125}-2E_{134}-E_{135}-2E_{145}-2E_{234}-2E_{235}-E_{245}-E_{345} $} & {$ x_{0}^{2} x_{1} x_{2}^{4} x_{3}^{2} + x_{0} x_{1}^{2} x_{2}^{4} x_{3}^{2} + x_{0}^{2} x_{1} x_{2}^{2} x_{3}^{4} + x_{0} x_{1}^{2} x_{2}^{2} x_{3}^{4} + x_{0}^{4} x_{1}^{2} x_{2} x_{3} x_{4} + x_{0}^{2} x_{1}^{4} x_{2} x_{3} x_{4} + x_{0}^{4} x_{1} x_{2}^{2} x_{3} x_{4} + x_{0} x_{1}^{4} x_{2}^{2} x_{3} x_{4} + x_{0}^{2} x_{1} x_{2}^{4} x_{3} x_{4} + x_{0} x_{1}^{2} x_{2}^{4} x_{3} x_{4} + x_{0}^{4} x_{1} x_{2} x_{3}^{2} x_{4} + x_{0} x_{1}^{4} x_{2} x_{3}^{2} x_{4} + x_{0}^{4} x_{2}^{2} x_{3}^{2} x_{4} + x_{1}^{4} x_{2}^{2} x_{3}^{2} x_{4} + x_{0} x_{1} x_{2}^{4} x_{3}^{2} x_{4} + x_{1}^{2} x_{2}^{4} x_{3}^{2} x_{4} + x_{0}^{2} x_{1} x_{2} x_{3}^{4} x_{4} + x_{0} x_{1}^{2} x_{2} x_{3}^{4} x_{4} + x_{0}^{2} x_{2}^{2} x_{3}^{4} x_{4} + x_{0} x_{1} x_{2}^{2} x_{3}^{4} x_{4} + x_{0}^{2} x_{1}^{4} x_{2} x_{4}^{2} + x_{0} x_{1}^{4} x_{2}^{2} x_{4}^{2} + x_{0}^{4} x_{1}^{2} x_{3} x_{4}^{2} + x_{0}^{4} x_{1} x_{2} x_{3} x_{4}^{2} + x_{0} x_{1}^{4} x_{2} x_{3} x_{4}^{2} + x_{1}^{4} x_{2}^{2} x_{3} x_{4}^{2} + x_{0} x_{1} x_{2}^{4} x_{3} x_{4}^{2} + x_{1}^{2} x_{2}^{4} x_{3} x_{4}^{2} + x_{0}^{4} x_{1} x_{3}^{2} x_{4}^{2} + x_{0}^{4} x_{2} x_{3}^{2} x_{4}^{2} + x_{0}^{2} x_{2} x_{3}^{4} x_{4}^{2} + x_{0} x_{1} x_{2} x_{3}^{4} x_{4}^{2} + x_{0}^{2} x_{1}^{2} x_{2} x_{4}^{4} + x_{0} x_{1}^{2} x_{2}^{2} x_{4}^{4} + x_{0}^{2} x_{1}^{2} x_{3} x_{4}^{4} + x_{0}^{2} x_{1} x_{2} x_{3} x_{4}^{4} + x_{0} x_{1}^{2} x_{2} x_{3} x_{4}^{4} + x_{0} x_{1} x_{2}^{2} x_{3} x_{4}^{4} + x_{0}^{2} x_{1} x_{3}^{2} x_{4}^{4} + x_{0} x_{1} x_{2} x_{3}^{2} x_{4}^{4}    $} & $   2l-e_{012}-e_{013}-e_{024}-e_{035}-e_{045}-e_{125}-e_{134}-e_{145}-e_{234}-e_{235}     $ & $\mathcal{O}^2\oplus \mathcal{O}(-2)$    \\

\hline
{$\textit{(210) } \boldsymbol{D_{38}} =12H-8E_{0}-8E_{1}-8E_{2}-6E_{3}-6E_{4}-6E_{5}-5E_{01}-5E_{02}-5E_{03}-4E_{04}-5E_{05}-5E_{12}-4E_{13}-5E_{14}-5E_{15}-5E_{23}-5E_{24}-4E_{25}-2E_{34}-2E_{35}-2E_{45}-2E_{012}-3E_{013}-3E_{014}-2E_{015}-2E_{023}-3E_{024}-3E_{025}-E_{034}-E_{035}-E_{045}-3E_{123}-2E_{124}-3E_{125}-E_{134}-E_{135}-E_{145}-E_{234}-E_{235}-E_{245}-E_{345}  $} & {$   x_{0}^{4} x_{1}^{3} x_{2}^{3} x_{3}^{2} + x_{0}^{3} x_{1}^{3} x_{2}^{4} x_{3}^{2} + x_{0}^{4} x_{1}^{2} x_{2}^{3} x_{3}^{3} + x_{0}^{3} x_{1}^{2} x_{2}^{4} x_{3}^{3} + x_{0}^{3} x_{1}^{3} x_{2}^{3} x_{3}^{2} x_{4} + x_{0}^{2} x_{1}^{3} x_{2}^{4} x_{3}^{2} x_{4} + x_{0}^{3} x_{1}^{3} x_{2}^{2} x_{3}^{3} x_{4} + x_{0}^{3} x_{1}^{2} x_{2}^{3} x_{3}^{3} x_{4} + x_{0}^{2} x_{1}^{3} x_{2}^{3} x_{3}^{3} x_{4} + x_{0}^{2} x_{1}^{2} x_{2}^{4} x_{3}^{3} x_{4} + x_{0}^{3} x_{1}^{2} x_{2}^{2} x_{3}^{4} x_{4} + x_{0}^{2} x_{1}^{2} x_{2}^{3} x_{3}^{4} x_{4} + x_{0}^{3} x_{1}^{4} x_{2}^{3} x_{4}^{2} + x_{0}^{3} x_{1}^{3} x_{2}^{4} x_{4}^{2} + x_{0}^{3} x_{1}^{3} x_{2}^{3} x_{3} x_{4}^{2} + x_{0}^{3} x_{1}^{2} x_{2}^{4} x_{3} x_{4}^{2} + x_{0}^{4} x_{1}^{3} x_{2} x_{3}^{2} x_{4}^{2} + x_{0}^{3} x_{1}^{4} x_{2} x_{3}^{2} x_{4}^{2} + x_{0}^{4} x_{1} x_{2}^{3} x_{3}^{2} x_{4}^{2} + x_{0} x_{1}^{4} x_{2}^{3} x_{3}^{2} x_{4}^{2} + x_{0}^{4} x_{1}^{2} x_{2} x_{3}^{3} x_{4}^{2} + x_{0}^{3} x_{1}^{3} x_{2} x_{3}^{3} x_{4}^{2} + x_{0}^{2} x_{1}^{3} x_{2}^{2} x_{3}^{3} x_{4}^{2} + x_{0}^{4} x_{2}^{3} x_{3}^{3} x_{4}^{2} + x_{0}^{3} x_{1} x_{2}^{3} x_{3}^{3} x_{4}^{2} + x_{0} x_{1}^{3} x_{2}^{3} x_{3}^{3} x_{4}^{2} + x_{0} x_{1}^{4} x_{2} x_{3}^{4} x_{4}^{2} + x_{0}^{3} x_{2}^{3} x_{3}^{4} x_{4}^{2} + x_{0}^{2} x_{1}^{2} x_{2} x_{3}^{5} x_{4}^{2} + x_{0} x_{1}^{2} x_{2}^{2} x_{3}^{5} x_{4}^{2} + x_{0}^{2} x_{1}^{4} x_{2}^{3} x_{4}^{3} + x_{0}^{2} x_{1}^{3} x_{2}^{4} x_{4}^{3} + x_{0}^{3} x_{1}^{3} x_{2}^{2} x_{3} x_{4}^{3} + x_{0}^{3} x_{1}^{2} x_{2}^{3} x_{3} x_{4}^{3} + x_{0}^{2} x_{1}^{3} x_{2}^{3} x_{3} x_{4}^{3} + x_{0}^{2} x_{1}^{2} x_{2}^{4} x_{3} x_{4}^{3} + x_{0}^{3} x_{1}^{3} x_{2} x_{3}^{2} x_{4}^{3} + x_{0}^{2} x_{1}^{4} x_{2} x_{3}^{2} x_{4}^{3} + x_{0}^{3} x_{1}^{2} x_{2}^{2} x_{3}^{2} x_{4}^{3} + x_{0}^{3} x_{1} x_{2}^{3} x_{3}^{2} x_{4}^{3} + x_{0} x_{1}^{3} x_{2}^{3} x_{3}^{2} x_{4}^{3} + x_{1}^{4} x_{2}^{3} x_{3}^{2} x_{4}^{3} + x_{0}^{3} x_{1} x_{2} x_{3}^{4} x_{4}^{3} + x_{0} x_{1}^{3} x_{2} x_{3}^{4} x_{4}^{3} + x_{1}^{4} x_{2} x_{3}^{4} x_{4}^{3} + x_{0}^{2} x_{1} x_{2}^{2} x_{3}^{4} x_{4}^{3} + x_{0} x_{1} x_{2}^{3} x_{3}^{4} x_{4}^{3} + x_{1}^{2} x_{2}^{3} x_{3}^{4} x_{4}^{3} + x_{0}^{2} x_{1}^{2} x_{3}^{5} x_{4}^{3} + x_{0}^{2} x_{1} x_{2} x_{3}^{5} x_{4}^{3} + x_{0} x_{1}^{2} x_{2} x_{3}^{5} x_{4}^{3} + x_{0} x_{1} x_{2}^{2} x_{3}^{5} x_{4}^{3} + x_{0} x_{1}^{2} x_{3}^{6} x_{4}^{3} + x_{1}^{2} x_{2} x_{3}^{6} x_{4}^{3} + x_{0}^{2} x_{1}^{3} x_{2}^{2} x_{3} x_{4}^{4} + x_{0}^{2} x_{1}^{2} x_{2}^{3} x_{3} x_{4}^{4} + x_{0}^{4} x_{1} x_{2} x_{3}^{2} x_{4}^{4} + x_{1}^{3} x_{2}^{3} x_{3}^{2} x_{4}^{4} + x_{0}^{4} x_{2} x_{3}^{3} x_{4}^{4} + x_{0}^{3} x_{1} x_{2} x_{3}^{3} x_{4}^{4} + x_{0} x_{1}^{3} x_{2} x_{3}^{3} x_{4}^{4} + x_{0} x_{1}^{2} x_{2}^{2} x_{3}^{3} x_{4}^{4} + x_{0}^{2} x_{2}^{3} x_{3}^{3} x_{4}^{4} + x_{0} x_{1} x_{2}^{3} x_{3}^{3} x_{4}^{4} + x_{0}^{3} x_{2} x_{3}^{4} x_{4}^{4} + x_{1}^{3} x_{2} x_{3}^{4} x_{4}^{4} + x_{0} x_{2}^{3} x_{3}^{4} x_{4}^{4} + x_{1} x_{2}^{3} x_{3}^{4} x_{4}^{4} + x_{0}^{2} x_{1} x_{3}^{5} x_{4}^{4} + x_{0} x_{1} x_{2} x_{3}^{5} x_{4}^{4} + x_{0} x_{1} x_{3}^{6} x_{4}^{4} + x_{1} x_{2} x_{3}^{6} x_{4}^{4} + x_{0}^{2} x_{1}^{2} x_{2} x_{3}^{2} x_{4}^{5} + x_{0}^{2} x_{1} x_{2}^{2} x_{3}^{2} x_{4}^{5} + x_{0}^{2} x_{1}^{2} x_{3}^{3} x_{4}^{5} + x_{0}^{2} x_{1} x_{2} x_{3}^{3} x_{4}^{5} + x_{0} x_{1}^{2} x_{2} x_{3}^{3} x_{4}^{5} + x_{0} x_{1} x_{2}^{2} x_{3}^{3} x_{4}^{5} + x_{0} x_{1}^{2} x_{3}^{4} x_{4}^{5} + x_{0} x_{1} x_{2} x_{3}^{4} x_{4}^{5} + x_{0}^{2} x_{1} x_{3}^{3} x_{4}^{6} + x_{0}^{2} x_{2} x_{3}^{3} x_{4}^{6} + x_{0} x_{1} x_{3}^{4} x_{4}^{6} + x_{0} x_{2} x_{3}^{4} x_{4}^{6}   $} & $4l - e_{03} -e_{05} -e_{14} -e_{15} -e_{23} -e_{24} -e_{013} -e_{014} -e_{024} -e_{025} -e_{123} -e_{125} -2e_{345}$ & $\mathcal{O}^2\oplus \mathcal{O}(-2)$ \\

\end{longtable}

\newpage
\normalsize

\textbf{Description of Table \ref{notExt}:} Table \ref{notExt} lists all known integral, rigid divisors on $\overline{M}_{0,7}$ which are not extreme over characteristic 0.  These are necessary generators of the Cox ring of $\overline{M}_{0,7}$, and the monoid of effective divisor classes, but not of $\text{Eff}(\overline{M}_{0,7})$.  The quadric divisor appears in \cite{doran2016simplicial}, but the two quintic divisors are new.
\setlength\LTleft{-6.75em}
\begin{longtable} { || p{21em} | p{24em} || }
\hline
\multicolumn{2} { | c | }{\textbf{Table \ref{notExt}:} Irreducible, Rigid, Effective Divisor Classes on $\overline{M}_{0,7}$ which are not Extreme}\label{notExt}\\
\hline
Divisor & Polynomial \\
\hline
\hline
\endfirsthead
\hline
\multicolumn{2}{ | c | }{Irreducible, Rigid, Effective Divisor Classes on $\overline{M}_{0,7}$ which are not Extreme}\\
\hline
Divisor & Polynomial \\
\hline
\hline
\endhead
\hline
\endfoot
\hline
\hline
\multicolumn{2}{ | c | }{End of Irreducible, Rigid, Effective Divisor Classes on $\overline{M}_{0,7}$ which are not Extreme}\\
\hline
\endlastfoot
\hline  $2H-E_{0}-E_{1}-E_{2}-E_{3}-E_{4}-E_{5}-E_{03}-E_{04}-E_{05}-E_{13}-E_{14}-E_{15}-E_{23}-E_{24}$  &  $x_{0} x_{1} - x_{0} x_{2} - x_{1} x_{2} + x_{3} x_{4}$  \\ 
\hline  $5H-3E_{0}-3E_{1}-3E_{2}-3E_{3}-3E_{4}-3E_{5}-2E_{01}-2E_{02}-E_{03}-2E_{04}-E_{05}-2E_{12}-2E_{13}-E_{14}-E_{15}-E_{23}-E_{24}-2E_{25}-2E_{34}-E_{35}-E_{45}-E_{012}-E_{013}-E_{014}-E_{024}-E_{035}-E_{045}-E_{123}-E_{135}-E_{145}-E_{234}$  &  $x_{0}^{2} x_{1} x_{2} x_{3} - 2 x_{0} x_{1} x_{2}^{2} x_{3} - x_{0}^{2} x_{2} x_{3}^{2} + x_{0} x_{1} x_{2} x_{3}^{2} + x_{0} x_{2}^{2} x_{3}^{2} - x_{0} x_{1}^{2} x_{2} x_{4} + 2 x_{0} x_{1} x_{2}^{2} x_{4} - 2 x_{0}^{2} x_{1} x_{3} x_{4} + 2 x_{0} x_{1}^{2} x_{3} x_{4} + x_{0}^{2} x_{2} x_{3} x_{4} - x_{1}^{2} x_{2} x_{3} x_{4} - x_{0} x_{2}^{2} x_{3} x_{4} + x_{1} x_{2}^{2} x_{3} x_{4} + x_{0}^{2} x_{3}^{2} x_{4} - x_{0} x_{1} x_{3}^{2} x_{4} - x_{0} x_{2} x_{3}^{2} x_{4} - x_{0} x_{1} x_{2} x_{4}^{2} + x_{1}^{2} x_{2} x_{4}^{2} - x_{1} x_{2}^{2} x_{4}^{2} + x_{0} x_{1} x_{3} x_{4}^{2} - x_{1}^{2} x_{3} x_{4}^{2} + x_{1} x_{2} x_{3} x_{4}^{2}$  \\ 
\hline  $5H-3E_{0}-3E_{1}-3E_{2}-3E_{3}-3E_{4}-3E_{5}-E_{01}-2E_{02}-2E_{03}-2E_{04}-2E_{05}-2E_{12}-2E_{13}-2E_{14}-2E_{15}-2E_{23}-E_{24}-2E_{25}-2E_{34}-E_{35}-E_{45}-E_{012}-E_{013}-E_{023}-E_{024}-E_{025}-E_{034}-E_{035}-E_{045}-E_{123}-E_{124}-E_{125}-E_{134}-E_{135}-E_{145}$  &  $-x_{0} x_{1} x_{2}^{2} x_{3} + x_{0} x_{1} x_{2} x_{3}^{2} - x_{0}^{2} x_{1}^{2} x_{4} + x_{0}^{2} x_{1} x_{2} x_{4} + x_{0} x_{1}^{2} x_{2} x_{4} + x_{0}^{2} x_{1} x_{3} x_{4} + x_{0} x_{1}^{2} x_{3} x_{4} - x_{0}^{2} x_{2} x_{3} x_{4} - 2 x_{0} x_{1} x_{2} x_{3} x_{4} - x_{1}^{2} x_{2} x_{3} x_{4} + x_{0} x_{2}^{2} x_{3} x_{4} + x_{1} x_{2}^{2} x_{3} x_{4} - x_{0} x_{1} x_{3}^{2} x_{4} - x_{0} x_{1} x_{2} x_{4}^{2} + x_{0} x_{2} x_{3} x_{4}^{2} + x_{1} x_{2} x_{3} x_{4}^{2} - x_{2}^{2} x_{3} x_{4}^{2}$  \\ 

\end{longtable}


\textbf{Description of Table \ref{unpairedNeg}:} Table \ref{unpairedNeg} lists curve classes on $\overline{M}_{0,7}$ that sweep out unknown divisors.  These might be negative curve classes on the divisor they sweep out, or they could be nef.  These curves were facets of the cone of known-to-be-effective divisors.  After finding a particular irreducible, rational representative of these curve classes (over characteristic 101), we computed the normal bundle of this representative, which appears below.  Since each of the normal bundles has degree $-1$, the splitting type cannot change upon generalization, which shows these curves sweep out divisors.


\begin{longtable} { || p{38em} | p{7em} || }
\hline
\multicolumn{2} { | c | }{\textbf{Table \ref{unpairedNeg}:} Possible Negative Curves on $\overline{M}_{0,7}$ that sweep out Unknown Divisors}\label{unpairedNeg}\\
\hline
Curve Class & Normal Bundle \\
\hline
\hline
\endfirsthead
\hline
\multicolumn{2}{ | c | }{\textbf{Table \ref{unpairedNeg}:} Possible Negative Curves on $\overline{M}_{0,7}$ that sweep out Unknown Divisors}\\
\hline
Curve Class & Normal Bundle \\
\hline
\hline
\endhead
\hline
\endfoot
\hline
\hline
\multicolumn{2}{ | c | }{End of Possible Negative Curves on $\overline{M}_{0,7}$ that sweep out Unknown Divisors}\\
\hline
\endlastfoot
\hline  $3l-e_{24}-e_{012}-e_{013}-e_{014}-e_{023}-e_{035}-2e_{045}-2e_{125}-e_{134}-e_{135}-e_{234}$  &  $\mathcal{O}\oplus\mathcal{O}\oplus\mathcal{O}(-1)$  \\ 
\hline  $3l-e_{23}-e_{35}-e_{45}-e_{012}-e_{014}-e_{015}-e_{024}-e_{025}-e_{034}-e_{123}-e_{134}$  &  $\mathcal{O}\oplus\mathcal{O}\oplus\mathcal{O}(-1)$  \\ 
\hline  $3l-e_{14}-e_{012}-e_{013}-e_{023}-e_{024}-e_{035}-2e_{045}-2e_{125}-e_{134}-e_{135}-e_{234}$  &  $\mathcal{O}\oplus\mathcal{O}\oplus\mathcal{O}(-1)$  \\ 
\hline  $3l-e_{04}-e_{14}-e_{25}-e_{34}-e_{35}-e_{012}-e_{015}-e_{023}-e_{123}$  &  $\mathcal{O}\oplus\mathcal{O}\oplus\mathcal{O}(-1)$  \\ 
\hline  $3l-e_{02}-e_{14}-e_{25}-e_{34}-e_{35}-e_{013}-e_{015}-e_{024}-e_{123}$  &  $\mathcal{O}\oplus\mathcal{O}\oplus\mathcal{O}(-1)$  \\ 
\hline  $3l-e_{02}-e_{13}-e_{14}-e_{45}-e_{015}-e_{024}-e_{034}-e_{035}-e_{123}-e_{125}$  &  $\mathcal{O}\oplus\mathcal{O}\oplus\mathcal{O}(-1)$  \\ 
\hline  $3l-e_{01}-e_{24}-e_{34}-e_{013}-e_{024}-e_{035}-e_{045}-e_{123}-e_{125}-e_{145}-e_{235}$  &  $\mathcal{O}\oplus\mathcal{O}\oplus\mathcal{O}(-1)$  \\ 
\hline  $4l-e_{13}-e_{34}-e_{35}-e_{45}-e_{012}-2e_{014}-e_{015}-2e_{023}-e_{035}-e_{124}-2e_{125}-e_{234}$  &  $\mathcal{O}\oplus\mathcal{O}\oplus\mathcal{O}(-1)$  \\ 
\hline  $4l-e_{13}-e_{24}-e_{35}-e_{45}-e_{013}-2e_{014}-2e_{023}-e_{025}-e_{045}-e_{124}-e_{125}-e_{135}-e_{234}$  &  $\mathcal{O}\oplus\mathcal{O}\oplus\mathcal{O}(-1)$  \\ 
\hline  $4l-e_{13}-e_{23}-2e_{45}-e_{013}-2e_{014}-e_{023}-e_{024}-e_{025}-e_{035}-e_{124}-e_{125}-e_{135}-e_{234}$  &  $\mathcal{O}\oplus\mathcal{O}\oplus\mathcal{O}(-1)$  \\ 
\hline  $4l-e_{12}-e_{13}-e_{45}-e_{014}-e_{015}-2e_{023}-2e_{024}-2e_{035}-e_{045}-2e_{125}-2e_{134}$  &  $\mathcal{O}\oplus\mathcal{O}\oplus\mathcal{O}(-1)$  \\ 
\hline  $4l-e_{04}-e_{24}-2e_{012}-2e_{013}-e_{035}-2e_{045}-2e_{125}-e_{134}-e_{135}-e_{145}-2e_{234}-e_{235}$  &  $\mathcal{O}\oplus\mathcal{O}\oplus\mathcal{O}(-1)$  \\ 
\hline  $4l-e_{02}-e_{04}-e_{23}-e_{35}-e_{45}-e_{013}-e_{015}-e_{024}-e_{035}-e_{124}-2e_{125}-2e_{134}$  &  $\mathcal{O}\oplus\mathcal{O}\oplus\mathcal{O}(-1)$  \\ 
\hline  $4l-e_{02}-e_{03}-e_{45}-e_{013}-2e_{014}-e_{025}-e_{035}-e_{045}-e_{123}-e_{124}-2e_{125}-e_{135}-2e_{234}$  &  $\mathcal{O}\oplus\mathcal{O}\oplus\mathcal{O}(-1)$  \\ 
\hline  $4l-e_{02}-e_{03}-e_{23}-2e_{45}-e_{014}-e_{015}-e_{024}-e_{035}-e_{123}-2e_{125}-2e_{134}$  &  $\mathcal{O}\oplus\mathcal{O}\oplus\mathcal{O}(-1)$  \\ 
\hline  $4l-e_{01}-2e_{23}-e_{24}-e_{45}-e_{013}-e_{014}-e_{025}-e_{034}-e_{035}-e_{045}-e_{124}-e_{125}-e_{135}$  &  $\mathcal{O}\oplus\mathcal{O}\oplus\mathcal{O}(-1)$  \\ 
\hline  $4l-e_{01}-e_{13}-e_{14}-e_{23}-e_{25}-e_{45}-e_{015}-e_{023}-2e_{024}-e_{035}-e_{125}-e_{134}$  &  $\mathcal{O}\oplus\mathcal{O}\oplus\mathcal{O}(-1)$  \\ 
\hline  $4l-e_{01}-e_{03}-2e_{24}-e_{35}-e_{45}-e_{014}-e_{015}-e_{025}-e_{034}-2e_{123}-e_{125}$  &  $\mathcal{O}\oplus\mathcal{O}\oplus\mathcal{O}(-1)$  \\ 
\hline  $4l-e_{01}-e_{03}-e_{14}-e_{25}-e_{34}-e_{45}-e_{015}-2e_{024}-e_{035}-2e_{123}-e_{125}$  &  $\mathcal{O}\oplus\mathcal{O}\oplus\mathcal{O}(-1)$  \\ 
\hline  $4l-e_{01}-e_{03}-e_{14}-e_{24}-e_{25}-e_{35}-e_{45}-e_{015}-e_{024}-e_{034}-2e_{123}$  &  $\mathcal{O}\oplus\mathcal{O}\oplus\mathcal{O}(-1)$  \\ 
\hline  $4l-e_{01}-e_{02}-2e_{35}-e_{45}-e_{014}-e_{015}-e_{024}-e_{025}-e_{034}-2e_{123}-e_{134}-e_{234}$  &  $\mathcal{O}\oplus\mathcal{O}\oplus\mathcal{O}(-1)$  \\ 
\hline  $4l-e_{01}-e_{02}-e_{23}-e_{24}-e_{45}-e_{013}-e_{014}-2e_{035}-e_{045}-2e_{125}-e_{134}-e_{234}$  &  $\mathcal{O}\oplus\mathcal{O}\oplus\mathcal{O}(-1)$  \\ 
\hline  $5l-e_{12}-e_{13}-e_{34}-2e_{45}-e_{012}-3e_{014}-2e_{023}-e_{025}-2e_{035}-2e_{125}-e_{135}-2e_{234}$  &  $\mathcal{O}\oplus\mathcal{O}\oplus\mathcal{O}(-1)$  \\ 
\hline  $5l-e_{03}-2e_{23}-e_{35}-2e_{45}-e_{012}-e_{013}-e_{014}-e_{015}-2e_{024}-e_{035}-e_{124}-2e_{125}-2e_{134}$  &  $\mathcal{O}\oplus\mathcal{O}\oplus\mathcal{O}(-1)$  \\ 
\hline  $5l-e_{02}-e_{13}-e_{23}-e_{24}-e_{34}-2e_{45}-e_{013}-2e_{014}-e_{015}-e_{024}-2e_{035}-e_{123}-2e_{125}$  &  $\mathcal{O}\oplus\mathcal{O}\oplus\mathcal{O}(-1)$  \\ 
\hline  $5l-e_{01}-2e_{23}-e_{35}-2e_{45}-e_{013}-e_{014}-e_{015}-e_{023}-2e_{024}-e_{035}-e_{124}-2e_{125}-2e_{134}$  &  $\mathcal{O}\oplus\mathcal{O}\oplus\mathcal{O}(-1)$  \\ 
\hline  $5l-2e_{01}-e_{24}-e_{35}-2e_{45}-e_{015}-e_{023}-2e_{024}-e_{034}-e_{035}-2e_{123}-2e_{125}-2e_{134}$  &  $\mathcal{O}\oplus\mathcal{O}\oplus\mathcal{O}(-1)$  \\ 
\hline  $5l-2e_{01}-e_{14}-e_{25}-e_{45}-2e_{023}-2e_{024}-2e_{035}-e_{045}-e_{123}-2e_{125}-2e_{134}-e_{135}-e_{234}$  &  $\mathcal{O}\oplus\mathcal{O}\oplus\mathcal{O}(-1)$  \\ 
\end{longtable}

\textbf{Description of Table \ref{nefCurves}:} Table \ref{nefCurves} lists extreme nef curve classes on $\overline{M}_{0,7}$ which cannot be written as a postive linear combination of known negative curves and other nef curves.  We found 796 of these.  No curves analogous to these exist for $\overline{M}_{0,6}$, as every nef curve on $\overline{M}_{0,6}$ is a positive linear combination of negative curves on boundary divisors.  These curves were found by testing facets of the cone of known-to-be-effective divisors.

\begin{proof}[Proof of Theorem \ref{nefCurvesThm}]
We found free representatives $f:\mathbb{P}^1\rightarrow \overline{M}_{0,7}$ for each curve in Table \ref{nefCurves} using Algorithm \ref{curveAlg}.  Such maps $[f]\in \text{Mor}_{\text{Spec}(\mathbb{Z})}(\mathbb{P}^1,\overline{M}_{0,7})\rightarrow \text{Spec}(\mathbb{Z})$ are smooth points of the quasi-projective morphism scheme, and therefore deform to free maps over characteristic $p$ for all but finitely many prime $p$.  In addition to the curves appearing below, we also found free representatives for 1,336 unlisted $S_7$-inequivalent curve classes.

As mentioned above, each curve class we tested was a facet of the cone $\mathfrak{E}$ generated by divisors in Tables \ref{ExtDivs0},\ref{ExtDivs1}, and \ref{boundaryCurves}.  As such, each curve annihilates an entire face of $\overline{\text{Eff}}(\overline{M}_{0,7})$, proving their extremality in $\text{Nef}_1(\overline{M}_{0,7})$.  We later expressed the 1,336 unlisted extreme nef curves as positive sums of nef and negative curves in Tables \ref{ExtDivs0}, \ref{ExtDivs1}, \ref{nefCurves}, and \ref{boundaryCurves}.  

While the $S_7$-orbits of curves in Table \ref{nefCurves} contains around 3.5 million classes, the $S_7$-orbits of all $796+1,336 = 2,132$ nef curves contains $9,874,620$ classes.  This is slightly less than advertised $9,875,000+$ classes in $2,135+$ orbits.  The remaining $380+$ curves and $3+$ orbits are $S_7$-equivalent to nonnegative sums of curves in Table \ref{boundaryCurves}.  For instance, the $S_7$-orbits of $l$, $l-e_{ij}$, $l-e_{ijk}$, and $l-e_{ij} - e_{hkm}$ contain $7$, $105$, $140$, and $210$ curves, respectively.  We easily see such curves are nef (free) and extreme in $\text{Nef}_1(\overline{M}_{0,7})$ because each annihilates $41$ linearly independent boundary divisors.  The quantity of extreme nef curves expressible as linear combinations of negative curves in Table \ref{boundaryCurves} is unknown, as these were never recorded.
\end{proof}

\begin{rem}
Using negative curves to bound an exterior approximation of $\text{Eff}(X)$ can simplify computations drastically.  For instance, the orbits of negative curves in Tables \ref{ExtDivs0}, \ref{ExtDivs1}, and \ref{boundaryCurves} only contain around 300,000 classes.  Including these as constraints allows us to exclude roughly $6,375,000$ extreme nef curve classes and bound a smaller possible region for unfound extreme divisor.  
\end{rem}

\begin{rem}
The first two curve classes in Table \ref{nefCurves} do not admit representatives with globally generated normal bundles.  Over characteristic 2 each curve is represented by the nodal union of a negative curve on a divisor in the $S_7$-orbit of $\boldsymbol{D_{37}}$ and a negative curve ($e_{345}$ and $e_{25} -e_{125}$, respectively) on either $E_{345}$ or $E_{25}$.  This nodal union is nef over characteristic 2, and thus over characteristic 0 as well.
\end{rem}


\begin{longtable}  { || p{38em} | p{7em} || }
\hline
\multicolumn{2} { | c | }{\textbf{Table \ref{nefCurves}:} Extreme Nef Curves on $\overline{M}_{0,7}$ in Characteristic 0}\label{nefCurves}\\
\hline
Curve Class & Normal Bundle \\
\hline
\hline
\endfirsthead
\hline
\multicolumn{2}{ | c | }{(Continuation) \textbf{Table \ref{nefCurves}:} Extreme Nef Curves on $\overline{M}_{0,7}$ in Characteristic 0}\\
\hline
Curve Class & Normal Bundle \\
\hline
\hline
\endhead
\hline
\endfoot
\hline
\hline
\multicolumn{2}{ | c | }{End of Extreme Nef Curves on $\overline{M}_{0,7}$ in Characteristic 0}\\
\hline
\endlastfoot

\hline  $2l-e_{013}-e_{014}-e_{023}-e_{025}-e_{045}-e_{124}-e_{125}-e_{135}-e_{234}$  &  $\mathcal{O}\oplus\mathcal{O}\oplus\mathcal{O}(-1)$  \\ 
\hline  $3l-e_{12}-e_{13}-e_{34}-e_{45}-e_{014}-e_{015}-e_{023}-e_{024}-e_{035}-e_{125}$  &  $\mathcal{O}\oplus\mathcal{O}\oplus\mathcal{O}(-1)$  \\ 
\hline  $2l-e_{01}-e_{024}-e_{035}-e_{045}-e_{123}-e_{125}-e_{134}$  &  $\mathcal{O}\oplus\mathcal{O}\oplus\mathcal{O}$  \\ 
\hline  $2l-e_{01}-e_{23}-e_{024}-e_{035}-e_{125}-e_{134}$  &  $\mathcal{O}\oplus\mathcal{O}\oplus\mathcal{O}$  \\ 
\hline  $3l-e_{013}-e_{014}-2e_{023}-e_{035}-2e_{045}-e_{124}-2e_{125}-e_{134}-e_{135}-e_{234}$  &  $\mathcal{O}\oplus\mathcal{O}\oplus\mathcal{O}$  \\ 
\hline  $3l-e_{012}-2e_{024}-2e_{035}-e_{045}-2e_{125}-2e_{134}-e_{135}-e_{234}$  &  $\mathcal{O}\oplus\mathcal{O}\oplus\mathcal{O}(1)$  \\ 
\hline  $3l-e_{012}-2e_{023}-e_{024}-e_{035}-2e_{045}-e_{124}-2e_{125}-2e_{134}-e_{135}$  &  $\mathcal{O}\oplus\mathcal{O}\oplus\mathcal{O}$  \\ 
\hline  $3l-e_{35}-e_{013}-e_{023}-2e_{024}-e_{035}-e_{045}-e_{124}-2e_{125}-2e_{134}$  &  $\mathcal{O}\oplus\mathcal{O}\oplus\mathcal{O}$  \\
\end{longtable}

\fontsize{10}{12}\selectfont

\textbf{Description of Table \ref{boundaryCurves}:} Table \ref{boundaryCurves} lists negative curves on boundary divisors.  These are included for posterity, as they are reasonably simple to find.  We note that the negative curves listed do not generate the cone of nef curves on either boundary divisor.  For instance, $e_{01}$ is the pushforward of an extreme nef class on $E_{01}$, but this curve class cannot be generated by the curves listed for $E_{01}$ alone.  Instead, the $S_7$-orbits of these curve classes generate the pushforwards of all nef classes from each boundary divisor.


\begin{longtable}  { || p{10em} | p{35em} || }
\hline
\multicolumn{2} { | c | }{\textbf{Table \ref{boundaryCurves}:} Boundary Divisors and Negative Curves on $\overline{M}_{0,7}$}\label{boundaryCurves}\\
\hline
\textit{(Orbit Size)} Boundary Divisor & Negative Curve \\
\hline
\hline
\endfirsthead
\hline
\multicolumn{2}{ | c | }{(Continuation) \textbf{Table \ref{boundaryCurves}:} Boundary Divisors and Negative Curves on $\overline{M}_{0,7}$}\\
\hline
\textit{(Orbit Size)} Boundary Divisor & Negative Curve \\
\hline
\hline
\endhead
\hline
\endfoot
\hline
\hline
\multicolumn{2}{ | c | }{End of Boundary Divisors and Negative Curves on $\overline{M}_{0,7}$}\\
\hline
\endlastfoot

\multirow{3}{10em}{\textit{(35) }$E_{01}$} &  $l-e_0-e_1+e_{01}$ \\
\cline{2-2}  & $e_{01}-e_{012}$ \\
\cline{2-2}  & $2e_{01}-e_{012}-e_{013}-e_{014}-e_{015}$ \\
\hline
\multirow{31}{10em}{\textit{(21) } $E_0$ } & $e_0 - e_{01}$\\
\cline{2-2}  & $e_0 - e_{012} -e_{034}$\\
\cline{2-2}  & $2e_0- e_{01} -e_{02} -e_{034} -e_{035}$\\
\cline{2-2}  & $2e_0- e_{013} -e_{014}-e_{015} -e_{023}-e_{024}-e_{025}$\\
\cline{2-2}  & $2e_0 - e_{012} -e_{013} -e_{024} -e_{035} -e_{045}$\\
\cline{2-2}  & $3e_0- e_{01} -e_{02} -e_{03} -e_{04} -e_{05}$\\
\cline{2-2}  & $3e_0 -2e_{01} - e_{023}- e_{024}- e_{025}- e_{034}- e_{035}- e_{045}$\\
\cline{2-2}  & $3e_0- e_{01}-e_{03} -e_{014}-e_{015} -e_{023}-e_{024}-e_{025}$\\
\cline{2-2}  & $3e_0 -e_{01} -e_{02} -e_{03} -e_{014} -e_{025} -e_{045}$\\
\cline{2-2}  & $3e_0 -e_{03} -e_{05} -e_{012} -e_{014} -e_{024} -e_{025} -e_{034}$\\
\cline{2-2}  & $3e_0 -2e_{05} -e_{04} -e_{012} -e_{013} -e_{023}$\\
\cline{2-2}  & $3e_0  -e_{013} -e_{014} -2e_{023} -e_{025} -e_{034} -2e_{045}$\\
\cline{2-2}  & $3e_0 -e_{05}  -e_{013} -e_{015} -e_{024} -e_{025} -2e_{034}$\\
\cline{2-2}  & $4e_0 -2e_{01} -e_{02} -e_{023} -e_{025} -e_{034} -e_{035} -e_{04} -e_{045}$\\
\cline{2-2}  & $4e_0 -e_{01} -e_{02} -e_{03} -e_{015} -e_{023} -e_{025} -e_{034} -2e_{045}$\\
\cline{2-2}  & $4e_0 -2e_{01} -e_{04} -e_{023} -2e_{025} -e_{034} -e_{035} -e_{045}$\\
\cline{2-2}  & $4e_0- e_{01} -e_{02} -e_{03} -2e_{05} -e_{014} -e_{024} -e_{034}$\\
\cline{2-2}  & $4e_0- e_{04} -e_{05} -e_{014} -e_{015} -e_{023} -2e_{025} -2e_{034}$\\
\cline{2-2}  & $4e_0 - 2e_{05} -e_{013} -e_{014} -e_{015} -e_{023} -e_{024} -e_{025} -2e_{034}$\\
\cline{2-2}  & $4e_0 -2e_{05} -e_{01} -e_{013} -e_{023} -e_{024} -e_{025} -2e_{034}$\\
\cline{2-2}  & $5e_0 - 2e_{01} -e_{012} -e_{013} -e_{023} -2e_{025} -2e_{034} -3e_{045}$\\
\cline{2-2}  & $5e_0 -2e_{01} -e_{023} -e_{024} -e_{025} -e_{03} -e_{034} -e_{04} -2e_{05}$\\
\cline{2-2}  & $5e_0-2e_{05} -2e_{014} -e_{015} -e_{023} -e_{024} -2e_{025} -3e_{034}$\\
\cline{2-2}  & $5e_0 -2e_{01} -2e_{02} -e_{014} -e_{023} -e_{025} -2e_{034} -e_{035} -2e_{045}$\\
\cline{2-2}  & $5e_0 - 2e_{03} -e_{04} -2e_{05} -2e_{012} -e_{013} -e_{014} -e_{024} -e_{025}$\\
\cline{2-2}  & $5e_0 - 2e_{01} -e_{02} -e_{03} -e_{04} -e_{05} -e_{023} -e_{025} -e_{034} -e_{045}$\\
\cline{2-2}  & $5e_0-e_{01} -e_{02} -3e_{05} -e_{013} -e_{014} -e_{023} -e_{024} -2e_{034}$\\
\cline{2-2}  & $6e_0 - 2e_{02} -e_{04} -3e_{05} -2e_{013} -e_{014} -e_{023} -2e_{034} $\\
\cline{2-2}  & $6e_0- 2e_{01} -2e_{02} -e_{03} -e_{05} -e_{014} -e_{023} -e_{025} -2e_{034} -2e_{045}$\\
\cline{2-2}  & $7e_0- 2e_{01} -4e_{05} -e_{013} -e_{014} -2e_{023} -2e_{024} -e_{025} -3e_{034}$\\
\cline{2-2}  & $7e_0 -2e_{01} -2e_{04} -2e_{05} -e_{012} -e_{013} -3e_{023} -2e_{025} -2e_{034} -e_{045}$ \\
\end{longtable}



\vskip.5in

\printbibliography


\end{document}